\documentclass{amsart}

\usepackage[margin=4.2cm]{geometry}
\usepackage{amsmath,amssymb,amsfonts,amsthm}
\usepackage[usenames,dvipsnames]{color}
\usepackage{mathtools}
\usepackage[all,cmtip]{xy}
\usepackage{enumerate}
\usepackage[colorlinks=true,linkcolor=red,citecolor=blue]{hyperref}

\newtheorem{theorem}{Theorem}[section]
\newtheorem{proposition}[theorem]{Proposition}
\newtheorem{lemma}[theorem]{Lemma}
\newtheorem{corollary}[theorem]{Corollary}

\theoremstyle{definition}
\newtheorem{definition}[theorem]{Definition}

\theoremstyle{remark}
\newtheorem{remark}[theorem]{Remark}
\newtheorem{example}[theorem]{Example}

\newcommand{\PP}{\mathbb{P}}
\newcommand{\MT}{\mathcal{MT}}

\newcommand{\lmod}{\!\!\!\mod}
\newcommand{\can}{\mathrm{can}}
\newcommand{\QQ}{\mathbb{Q}}
\newcommand{\LL}{\mathbb{L}}
\newcommand{\RR}{\mathbb{R}}
\newcommand{\CC}{\mathbb{C}}
\newcommand{\ZZ}{\mathbb{Z}}
\newcommand{\GG}{\mathbb{G}}

\newcommand{\Pe}{\mathcal{P}}
\newcommand{\mm}{\mathfrak{m}}

\newcommand{\Or}{\mathcal{O}}

\newcommand{\s}{\mathsf{s}}
\newcommand{\vol}{\mathrm{vol}}
\newcommand{\To}{\longrightarrow}
\newcommand{\eps}{\varepsilon}
\newcommand{\A}{\mathcal{A}}

\newcommand{\B}{\mathrm{B}}
\newcommand{\dR}{\mathrm{dR}}
\renewcommand{\min}{\backslash}
\renewcommand{\c}{\mathrm{c}}

\DeclareMathOperator{\sv}{\mathsf{sv}}

\DeclareMathOperator{\id}{id}

\title{Single-valued integration  and double copy}

\author{Francis Brown}
\address{All Souls College, Oxford, Oxford OX1 4AL, UK}
\email{francis.brown@all-souls.ox.ac.uk}

\author{Cl\'{e}ment Dupont}
\address{Institut Montpelli\'{e}rain Alexander Grothendieck, Universit\'{e} de Montpellier, CNRS, Montpellier,  France}
\email{clement.dupont@umontpellier.fr}


\date{}

\begin{document}

\maketitle

\begin{abstract}  We study a single-valued integration pairing between differential forms and dual differential forms which subsumes some classical  constructions in mathematics and physics.
It can be interpreted as a $p$-adic period pairing at the   infinite prime. 
  The single-valued integration pairing is defined by transporting the action of complex conjugation from singular to de Rham cohomology via the comparison isomorphism. We  show   how quite general families of period integrals admit canonical single-valued versions and prove some general formulae for  them.  This implies an elementary   `double copy' formula expressing certain singular volume integrals over the complex points of a smooth projective variety as a quadratic expression in ordinary period integrals of half the dimension. We provide
several examples, including non-holomorphic modular forms, archimedean N\'{e}ron--Tate heights on curves, single-valued multiple zeta values and polylogarithms.
The results of the present paper are used in \cite{BD2} to prove a recent conjecture of Stieberger which relates the coefficients in a Laurent expansion of two different kinds of periods of twisted cohomology on the moduli spaces of curves $\mathcal{M}_{0,n}$ of genus zero with $n$ marked points. 
We also study a morphism between certain rings of `motivic' periods, called the de Rham projection,
which provides a bridge  between complex periods and single-valued periods in many situations of interest.
\end{abstract}

\section{Introduction}
Classical theories of integration concern integrals of closed differential forms $\omega$ over a  domain of integration  $\gamma$:
\begin{equation}\label{introIdef} I = \int_{\gamma} \omega    \ .
\end{equation} In the setting of algebraic geometry, integration can be interpreted as  a canonical   pairing between de Rham cohomology and singular homology for algebraic varieties. 
When the varieties depend algebraically on parameters  one typically obtains multi-valued functions of the parameters.

In this paper, we study a pairing between de Rham cohomology and its dual, de Rham homology. It assigns a number to a differential form and a `dual differential form'. When they depend on parameters,  the pairing is a    \emph{single-valued function} of the parameters, hence the name `single-valued integration'. 
The only  extra ingredient of the construction, in addition to the usual integration of forms over cycles (\ref{introIdef}), is complex conjugation (the `real Frobenius'). Single-valued integration  can in this way be   interpreted as a $p$-adic period pairing at the   infinite prime. 

This idea clarifies  many classical single-valued functions in mathematics and physics, but can also generate new objects when applied to familiar examples. 
In the sequel to this paper we apply our results  to the  case of the moduli space of curves of genus zero to settle a  recent question in string perturbation theory.

\subsection{Framework} \label{sect: Framework} Let $k$ be a field which admits an embedding $\sigma: k \hookrightarrow \CC$. Typically, we shall take $k=\QQ$, $k$ a general number field, or $k=\CC$. We  work in  the following setting:

\vspace{0.1in}
$(\star)_k$ \quad  : \quad  Let    $X$  be a smooth projective variety of dimension $n$ over   $k$, and let  $A, B\subset X$ be divisors with no common irreducible component such that $A\cup B$ is a normal crossing divisor in $X$.
\vspace{0.1in}

By resolution of singularities, any period 
in the sense of Kontsevich--Zagier \cite{kontsevichzagier} can be expressed as a period of the relative cohomology of  $(\star)_k$ for $k$ a number field. More precisely, let 
$$H_{\dR}= H^r_{\dR}(X \backslash A, B \backslash (A \cap B))  \qquad \qquad (\hbox{or, for short }  H_{\dR}^r(X \backslash A \lmod B))$$
denote relative algebraic de Rham cohomology.  It is a finite dimensional vector space over $k$. Let us denote the relative Betti or singular cohomology by 
$$H_{\B,\sigma}= H_{\mathrm{sing}}^r(X_\sigma\backslash A_\sigma(\CC), B_\sigma\backslash (A_\sigma \cap B_\sigma)(\CC)\,;\,\QQ) \qquad  (\hbox{or, }  H_{\B,\sigma}^r(X \backslash A \lmod B))  \ ,$$
where we use the notation $X_\sigma=X\times_{k,\sigma}\CC$, and so on.
It is a finite dimensional vector space over $\QQ$, dual to the singular homology
of $X_{\sigma}\backslash A_{\sigma}(\CC)$ relative to $B_{\sigma}(\CC)$.  According to de Rham \cite{derhamthese} and Grothendieck \cite{Grothendieck}, integration of algebraic differential forms over cycles  (\ref{introIdef}) can be interpreted as a canonical  and natural comparison isomorphism
$$\mathrm{comp}_{\sigma,\dR}  :   H_{\dR} \otimes_{k,\sigma} \CC \overset{\sim}{\To} H_{\B,\sigma} \otimes_{\QQ} \CC \ .$$
Let $\overline{\sigma} : k \hookrightarrow \CC$ denote the complex conjugate embedding.   Since complex conjugation is continuous, it defines  a natural  $\QQ$-linear isomorphism called the real Frobenius isomorphism:
$$F_{\infty}: \, H_{\B,\sigma} \overset{\sim}{\To} H_{\B,\overline{\sigma}}\ .$$
The \emph{single-valued period map}  is defined to be  the $\CC$-linear isomorphism
$$\s_{\sigma} \ : \   H_{\dR} \otimes_{k,\sigma} \CC \To   H_{\B,\sigma}  \otimes_{\QQ} \CC \overset{F_{\infty} \otimes \id}{\To} H_{\B,\overline{\sigma}}\otimes_{\QQ} \CC \To H_{\dR} \otimes_{k, \overline{\sigma}} \CC  $$ 
where the first and third isomorphisms are   $\mathrm{comp}_{\sigma,\dR}$ and $\mathrm{comp}^{-1}_{\overline{\sigma},\dR}$.
The map $\s_{\sigma}$  defines a $(\sigma,\overline{\sigma})$-bilinear pairing between de Rham cohomology classes and  de Rham homology classes, defined to be elements of the dual 
$ H_{\dR}^{\vee}$. It is typically transcendental, and quite different from the natural duality pairing between de Rham homology and cohomology (which takes values in $k$).
It satisfies analogues of all the usual rules for integration including change of variables formulae,  a  version of Stokes' theorem, and so on. This requires a certain formalism of matrix coefficients in Tannakian categories (\S\ref{subsec: tannakian intro}). 
\medskip

The single-valued periods of $H_{\B/\dR}$ are defined to be the entries of the matrix of the single-valued period map in any rational basis of $H_{\dR}$. They generate an interesting class of numbers (and functions when $X,A,B$ depend algebraically on parameters) with many applications.  Examples include: the single-valued   polylogarithms \cite{ganglzagier}, multiple zeta values \cite{brownSVMZV}, regulators \cite{beilinsondeligne}, and non-holomorphic modular forms \cite{brownCNHMF3}. They often arise  in physics since physical problems tend to have well-defined (as opposed to multi-valued) solutions. Examples include string amplitudes \cite{stiebergersvMZV, StiebergerTaylor}, Regge limits \cite{Dixon2012}, and the theory of graphical functions  \cite{schnetzgraphical}.   
Furthermore, 
 single-valued periods are invariants of the de Rham periods which occur in the de Rham coaction (\S\ref{sect: dR periods}) on `motivic' periods. 

\begin{example}\label{example log intro} Let $a\in \CC^{\times}\backslash \{1\}$  and let $\gamma$ be a path from $1$ to $a$. The logarithm 
$$\log a = \int_{\gamma}  \frac{dz}{z}$$
is a multi-valued function of $a$ whose monodromy lies in $2\pi i\, \ZZ$. This is due to the ambiguity in the choice of path $\gamma$. 
It is a period of (the family  over $\GG_m\backslash \{1\}$)  $$H^1(\PP^1 \backslash \{0,\infty\} \lmod \{1,a\})$$ which admits a de Rham basis  $[\frac{dz}{a-1}]$,
$[\frac{dz}{z}]$ and a  Betti basis  $[\gamma]$, $[\gamma_0]$, where $\gamma_0$ is a small positive loop around the origin.  With respect to these bases, the comparison isomorphism can be written as the period matrix 
$$P  =  \begin{pmatrix}  \int_{\gamma}  \frac{dz}{a-1}  &  \int_{\gamma}  \frac{dz}{z}  \\  \int_{\gamma_0}  \frac{dz}{a-1}   & \int_{\gamma_0}  \frac{dz}{z}  \end{pmatrix}=  \begin{pmatrix}  1   &  \log a \\  0  & 2  \pi i \end{pmatrix} \cdot
$$In   the same de Rham basis,  the single-valued map is represented by the matrix
$$ \overline{P}^{-1} P =    \begin{pmatrix}  1   &  \log |a|^2 \\  0  & -1\end{pmatrix} \cdot $$
Thus the single-valued period associated to the logarithm is $\log |a|^2$, and the single-valued version of $2\pi i$ is $-1$. 
 See \S\ref{par: example lefschetz} and \S\ref{par:example log} for a more detailed discussion.
\end{example} 

\begin{example}
The previous example has a natural generalisation where $\mathbb{P}^1$ is replaced with a smooth projective curve $X$ of genus $g$. The relevant cohomology group is thus 
$$H=H^1(X\backslash\{a_1,a_2\},\{b_1,b_2\})$$
where $a_1,a_2,b_1,b_2$ are four distinct points of $X$.  
The corresponding single-valued periods are related to the archimedean N\'{e}ron--Tate height pairing of the degree zero divisors $(a_2)-(a_1)$ and $(b_2)-(b_1)$ on $X$. More precisely, we prove in \S\ref{par: Greens} that this pairing arises as a quotient of a single-valued period of $\Lambda^{g+1}H$ by a single-valued period of $\Lambda^gH^1(X)$.
\end{example}

 \begin{remark}\label{rem p adic} The single-valued period map is a variant at the infinite prime of a $p$-adic period map where the role of Betti cohomology and the real Frobenius are played by crystalline cohomology and the crystalline Frobenius, respectively. In the case $k=\QQ$ we get for all primes of good reduction a $p$-adic period map, which is a $\QQ_p$-linear map
 $$\mathrm{per}_p: H_\dR\otimes_\QQ\QQ_p \longrightarrow H_{\mathrm{crys}}\stackrel{F_p}{\longrightarrow} H_{\mathrm{crys}} \longrightarrow H_\dR\otimes_\QQ\QQ_p\ .$$
 (Note that this is different from and more elementary than Fontaine's period maps which involve $p$-adic \'{e}tale cohomology and more complicated period rings \cite{fontainecorps}.) In the setting of Example \ref{example log intro}, for $a\in\mathbb{Z}$ not divisible by $p$, the $p$-adic period matrix equals
 $$\left(\begin{matrix} 1 & \log_p(a^{1-p}) \\ 0 & p\end{matrix}\right)$$
 where $\log_p$ denotes the $p$-adic logarithm \cite[\S 2.10]{deligneP1}. In particular, the $p$-adic period corresponding to $2\pi i$ is $p$ for all primes $p$.
 \end{remark}

For the remainder of this introduction we now restrict to the case where $\sigma=\overline{\sigma}$ is a real embedding and drop the reference to $\sigma$ in the notations. In this case the single-valued map $\s$ is defined over $\RR$, i.e., defines an $\RR$-linear automorphism of $H_\dR\otimes_k\RR$, and  admits a simple Tannakian description (\S\ref{subsec: tannakian intro}).

\subsection{Duality and formulae} 
The above definition of the single-valued period can be difficult to compute since one needs to know the entire period matrix. A related problem is how to represent de Rham homology classes, which are unfamiliar.
 
 Both issues can be resolved using  duality. 
 Recall that classical Poincar\'{e} duality gives a perfect pairing on the cohomology of a compact oriented manifold and allows one to interpret homology classes as cohomology classes of complementary degree.
 In order to extend this to the periods of mixed, as opposed to only pure, motives, we require a version with singularities (Theorem \ref{thm: PDNCD}). It states that there is a natural perfect pairing 
 \begin{equation} \label{introPVduality} \langle \quad  , \quad \rangle :  H^r(X \backslash B\lmod A)  \otimes  H^{2n-r}(X \backslash A \lmod B)  \To \QQ (-n)\ .
\end{equation}
In this manner, we can interpret de Rham homology classes as classes of  differential forms with  different singularities. This leads to a different perspective on single-valued periods as a pairing between classes of  differential forms:
 \begin{eqnarray} \label{introsvonforms}
H^r_\dR( X \backslash B \lmod A) \otimes H^{2n-r}_\dR( X \backslash A \lmod B)  &\To &  \RR  \\
\left[\nu\right]\otimes \left[\omega\right] & \mapsto &   (2 \pi i)^{-n}   \langle [\nu], \s [\omega] \rangle . \nonumber
\end{eqnarray}

Our first theorem gives a formula for this pairing when $\omega$ and $\nu$ have logarithmic singularities. Let us fix $r=n$ for simplicity and work with global algebraic differential $n$-forms $\omega$ and $\nu$ with logarithmic singularities along $A$ and $B$ respectively (see Theorem \ref{thm:sv recipe general} and Theorem \ref{thm:sv recipe} for more general and precise statements). Note that $\omega$ is a differential form on $X\min A$ which is automatically closed and vanishes along $B$ for degree reasons, and likewise $\nu$ with $A$, $B$ interchanged.

\begin{theorem} \label{introtheorem2}   The single-valued period (\ref{introsvonforms}) is computed by the following absolutely convergent integral:
\begin{equation}  \label{introtheorem2eq}   (2\pi i)^{-n}  \int_{X(\CC)}  \nu\wedge \overline{\omega} \ .
\end{equation}
\end{theorem}

The subtle point here is that the form $\nu\wedge\overline{\omega}$ is not a smooth form on $X(\CC)$. For example,  one cannot simply differentiate under the integral with respect to a parameter,  nor can one apply  Stokes' theorem in a naive manner (these operations are possible, however, see \S\ref{par:sv functoriality}).

This theorem allows for an  interpretation of certain $dz \wedge d\overline{z}$ integrals as single-valued periods. Since the single-valued period homomorphism on de Rham periods  has  a
large kernel  this imposes severe restrictions on the 
types of numbers and functions which can occur as such integrals. 

Although (\ref{introsvonforms}) is defined in full generality,  the integral  (\ref{introtheorem2eq}) does not converge in the general case when 
$\omega$, $\nu$ have higher-order poles.  Indeed Felder and Kazhdan   have recently studied similar kinds of integrals  using  a version of Riemann's mapping theorem in families \cite{felderkazhdanRiemann}  or by a zeta regularisation \cite{felderkazhdanregularization}.   We do not know how the latter definition compares to ours.  In any case, we can avoid these subtle issues since every  cohomology class may be represented by a form with logarithmic singularities.

From the previous theorem we deduce the following relation between elementary integrals.

\begin{corollary} \label{introcorDC} (Double copy formula). In the setting of the previous theorem:
\begin{equation}  \label{introDC}  \int_{X(\CC)}  \nu\wedge \overline{\omega}  =  \sum_{[\gamma],[\delta]} \langle  [\gamma]^{\vee}, [\delta]^{\vee} \rangle  \int_{\gamma} \nu \int_{\overline{\delta}}\omega
\end{equation} 
where $[\gamma]$ ranges over a basis for $H_n^{\B}(X \backslash B \,\lmod A)$ and $[\gamma]^{\vee}$ denotes the dual basis, and likewise $[\delta]$ with $A$, $B$ interchanged.
\end{corollary}

It follows from the compatibility between the duality \eqref{introPVduality} and the comparison isomorphism that the right-hand side of (\ref{introDC}) always computes the single-valued period for any $\omega$, $\nu$. The previous formula expresses certain $dz \wedge d  \overline{z}$ integrals  
 as finite quadratic expressions in ordinary period integrals. They are reminiscent of the so-called `KLT relations' in the physics literature which will be discussed 
 in the sequel to this paper \cite{BD2}.  On the left-hand side, one multiplies the two forms  $\nu$ and $\overline{\omega}$ together before integrating; on the right, one integrates  the forms individually first, and then multiplies the integrals together. The rational coefficients $\langle[\gamma]^\vee,[\delta]^\vee\rangle$ appearing in the formula are the entries of the inverse transpose matrix of the intersection matrix of representatives for the classes $[\gamma]$ and $[\delta]$.
 
 \begin{example}\label{introExlog}
 In the case of the logarithm, we have $\nu=d\log((z-a)/(z-1))$ and formula \eqref{introDC} reads :
 $$\int_{\mathbb{P}^1(\CC)} \nu\wedge \frac{d\overline{z}}{\overline{z}} = \int_{\delta_{0,\infty}} \nu\,\int_{\overline{\gamma_0}}\frac{dz}{z}  -\int_{\gamma_1} \nu\,\int_{\overline{\delta_{1,a}}}\frac{dz}{z}$$
 where $\gamma_j$ denotes a small positive loop around $j$, and $\delta_{0,\infty}$ and $\delta_{1,a}$ denote disjoint paths from $0$ to $\infty$ and from $1$ to $a$ respectively. Indeed, both sides of the equality equal $2\pi i\,\log|a|^2$ (see \S\ref{par:example log}).
 \end{example}
 
 It is interesting, and presumably straightforward, to  extend the above results to the case of cohomology with coefficients, or over a non-trivial base (i.e., in families). 

\subsection{Separated objects and projection} 
We now explain in which sense\footnote{In the literature, single-valued analogues of polylogarithms are widely referred to as  single-valued `cousins'. What we propose in this section, then, is a  mathematical definition of the word `cousin'.} single-valued periods can be assigned to ordinary period integrals. 
For this, we describe a construction which is unrelated to the single-valued period map, and  has possible applications of its own (notably for the definition of canonical volume forms). It enables us, under a certain condition of separatedness, to associate an algebraic differential form to a singular homology class.

We say that $H^n (X\backslash A\, \lmod B)$ is \emph{separated} if its Hodge numbers $h^{p,q}$ vanish unless $p=q=0$ or $p,q>0$. For instance, all effective mixed Hodge structures of Tate type satisfy this condition.
We fix an algebraic closure $\overline{k}$ of the field $k$ and an embedding $\overline{k}\hookrightarrow \CC$ compatible with $\sigma$. The next theorem states that in the separated case there is a natural map from homology classes to logarithmic  differential forms.

\begin{theorem}  \label{introthmseparated} If $H^n (X\backslash A\, \lmod B)$ is separated, there exists a natural map 
$$c_0^{\vee} :  H_n^\B(X\backslash A \lmod B) \To    H^{n}_{\dR}(X\backslash B \lmod A) \otimes_{k} \overline{k} \ . $$
The image $c_0^{\vee} (\gamma)$ of a  class $\gamma$ can be represented by a canonical global algebraic differential form  $\nu_{\gamma}$ on $X\times_k\overline{k}$ with logarithmic singularities along $B\times_k\overline{k}$. \end{theorem}

 The map $c_0^\vee$ replaces a cycle (Betti homology class) on $X\backslash A$ with boundary along $B$ with a form on $X\backslash B$, which can be viewed as a de Rham homology class in $H_\dR^\vee$ via the perfect pairing \eqref{introPVduality}. It thus allows one to associate a single-valued counterpart to an ordinary period. In view of Theorem \ref{introtheorem2}, this can informally be denoted by 
$$\int_{\gamma} \omega \qquad \quad \leadsto  \qquad    \quad (2\pi i)^{-n}\int_{X(\CC)} \nu_\gamma\wedge \overline{\omega}\ .$$
where $\omega$ is a global algebraic differential form on $X$ with logarithmic singularities along $A$.
This is called the \emph{de Rham projection} and requires the formalism of matrix coefficients in Tannakian categories to be defined (\S\ref{subsec: tannakian intro}).
One of its key features is that it provides a bridge between complex periods and single-valued periods in a way which is compatible with all algebraic relations  (assuming some form of the period conjecture).

We give an explicit recipe for computing the map $c_0^\vee$ (see Theorem \ref{coro:c0vee}) which roughly says that the poles of $\nu_{\gamma}$ lie along the irreducible components of $B$ which support the boundary of $\gamma$.

\begin{example} The logarithmic  differential form  
$$\nu_{\gamma} =  d\log\left(\frac{z-a}{z-1}\right)=\frac{dz}{z-a} - \frac{dz}{z-1} = \frac{(a-1)\, dz}{(z-1)(z-a)}  $$
is the image under $c_0^\vee$ of the class of the path $\gamma$ from $1$ to $a$. This justifies the claim made in Example \ref{introExlog}:
$$\log(a) = \int_{\gamma} \frac{dz}{z}  \quad  \qquad \leadsto \qquad    \quad   \log |a|^2 = \frac{1}{2\pi i} \int_{\PP^1(\CC)} \nu_\gamma\wedge\frac{d\overline{z}}{\overline{z}} \ .$$
 \end{example}

In a sequel to this paper \cite{BD2} we apply these ideas to integrals in superstring perturbation theory to prove a conjecture of Stieberger, which states that closed string amplitudes in genus zero are the single-valued projections of open string amplitudes. The underlying geometry is the  moduli spaces of genus zero curves. We also study, in this special case, the extension of our results to the case of cohomology with coefficients.

\subsection{Tannakian interpretation}\label{subsec: tannakian intro}
Our results are naturally phrased in the language of matrix coefficients in Tannakian categories. We lift the cohomology groups $H_{\dR/\B}$,  together with the comparison isomorphism, to an object in a Tannakian category of realisations $\mathcal{H}$. It plays the role of a category of motives. Indeed,  any reasonable category of motives $\mathcal{M}$ admits a functor to $\mathcal{H}$, which is hoped to be fully faithful. Every period can be lifted, via a geometric interpretation  of the integral which defines it, to an $\mathcal{H}$-period, which is an  element of  a ring $\Pe^\mm_{\mathcal{H}}$ of matrix coefficients  in $\mathcal{H}$. The period integral itself  can be recovered from the $\mathcal{H}$-period by applying the period homomorphism $\mathrm{per}:\Pe^\mm_{\mathcal{H}} \to \mathbb{C}$. A natural variant of the ring $\mathcal{P}^\mm_{\mathcal{H}}$ is the ring $\mathcal{P}^{\mm,\dR}_{\mathcal{H}}$ of de Rham $\mathcal{H}$-periods, i.e., functions on the Tannaka group of $\mathcal{H}$ associated to the de Rham fiber functor (de Rham motivic Galois group).  These occur in the coaction dual to the action of the  motivic Galois group, see \S \ref{sect: dR periods}. The single-valued period map gives rise to a homomorphism $\s:\Pe^{\mm,\dR}_{\mathcal{H}}\to \RR$ and allows in particular to realise  functions on the motivic Galois group as real numbers.
Note that the $p$-adic variant (Remark \ref{rem p adic}) leads, in a Tannakian framework which incorporates crystalline cohomology and the crystalline Frobenius, to a $p$-adic period map which allows one to realise functions on the motivic Galois group as $p$-adic numbers. In other terms, de Rham periods give the correct framework to study both single-valued periods and these kinds of $p$-adic periods.

The map $c_0^\vee$ from Theorem \ref{introthmseparated} defines in this context a morphism 
$$\pi^{\mm,\dR}: \Pe^{\mm,+}_{\mathcal{H},\mathrm{sep}}\otimes_\QQ\CC \longrightarrow \Pe^{\mm,\dR}_{\mathcal{H}}\otimes_\QQ\CC$$
where $\Pe^{\mm,+}_{\mathcal{H},\mathrm{sep}}$ is a subalgebra of $ \Pe^{\mm}_{\mathcal{H}}$. This is called the \emph{de Rham projection}. 
It provides a connection between   periods on the one hand and single-valued periods (or $p$-adic periods) on the other, and is in fact defined over a number field for motivic periods of objects of the form  $(\star)_{\overline{\QQ}}$.

\subsection{Contents}
In Section \S\ref{sect: MotPeriods} we review the framework of matrix coefficients in Tannakian categories (`motivic' periods), define the single-valued pairing in a general context, and prove some of its abstract properties.  In Section \S\ref{sect: Verdierduality} we review the formalism of Verdier duality, and prove the main integral formula for the single-valued periods associated to pairs of logarithmic forms (Theorems \ref{thm:sv recipe general}, \ref{thm:sv recipe}). In Section \S\ref{sect: dRProj} we discuss separated mixed Hodge structures, whose matrix coefficients admit a projection map to de Rham periods. We compute a formula for this map, and prove  Theorem \ref{introthmseparated} amongst other things. Section \S\ref{par:sv functoriality} briefly considers some of the subtle functoriality properties of the single-valued integrals. 

There are many  potentially interesting applications of this  theory.   In this paper, we study some of  the  very simplest examples,   most of which are treated in \S\ref{sect: examples}. They  include: regulators (for number fields) \S \ref{par:MT case}, periods of  the universal elliptic curve leading to non-holomorphic Eisenstein series \S\ref{par: UnivElliptic}, Green's functions and N\'eron--Tate heights on curves \S\ref{par: Greens},  and multiple zeta values \S\ref{par: SVMZV}. We recast the theory of archimedean height pairings on curves in the language of de Rham periods, and provide an integral formula for single-valued multiple zeta values, for which no  closed expression was previously known. We deduce that Deligne's associator, previously  defined only via a fixed point equation involving the non-explicit action of a motivic Galois group \cite{brownSVMZV},  actually admits an elementary integral formula. 

\subsection{Acknowledgements} 
This project has received funding from the European Research Council (ERC) under the European Union’s Horizon 2020 research and innovation programme (grant agreement no. 724638). Both authors thank  the IHES for hospitality. The second author was partially supported by ANR grant ANR-18-CE40-0017.  This paper was initiated during the trimester ``Periods in number theory, algebraic geometry and physics'' which took place at the HIM Bonn in 2018,  to which both authors offer their thanks. Many thanks to Andrey Levin, whose talk during this programme on the dilogarithm inspired this project, to Federico Zerbini for discussions, and to the anonymous reviewer for helpful comments.

\section{`Motivic' periods and the single-valued period homomorphism} \label{sect: MotPeriods}

\subsection{A category of realisations}

\subsubsection{Over the rationals}
We mostly work in a $\QQ$-linear Tannakian category $\mathcal{H}$ of systems of realisations as in \cite{deligneP1}.
 Its objects consist of triples 
$V= (V_{\B}, V_{\dR}, c) $
where $V_\B$, $V_{\dR}$ are finite-dimensional $\QQ$-vector spaces and $c$ is an isomorphism 
$c: V_{\dR} \otimes_{\QQ}\CC  \overset{\sim}{\To} V_\B \otimes_{\QQ}\CC$
with the following extra structures: 
\begin{itemize}
    \item  $V_\B$, and $V_{\dR}$ are equipped with an increasing filtration $W$ (the weight filtration). The isomorphism $c$ respects the weight filtrations on $V_\B \otimes_\QQ \CC, V_{\dR} \otimes_\QQ \CC.$
    \item $V_{\dR}$ is equipped with a decreasing filtration $F$ (the Hodge filtration). The data of $V_\B,W, cF$ defines a graded-polarizable $\QQ$-mixed Hodge structure.
    \item There is an involution 
    $F_{\infty} : V_\B \To V_\B$
    called the real Frobenius, such that the following diagram commutes
    \begin{equation}\label{eq: compatibility frobenius comp c definition}
    \begin{aligned}
    \xymatrix{
    V_\dR\otimes_\QQ\CC \ar[r]^c \ar[d]_{\id\otimes c_\dR} & V_\B\otimes_\QQ\CC \ar[d]^{F_\infty\otimes c_\B}\\
    V_\dR\otimes_\QQ\CC \ar[r]_c & V_\B\otimes_\QQ\CC
    }
    \end{aligned}
    \end{equation}
    where $c_\B$, $c_{\dR}$ denote the  action of complex conjugation on coefficients.
\end{itemize}
The morphisms in this category are linear maps $\phi$ between components of objects $\phi_\B, \phi_{\dR}$ which respect all the above data.  As shown in \cite{deligneP1}, this category is neutral Tannakian over $\QQ$ and has, in particular, two fiber functors 
$$\omega_{\B/\dR} : \mathcal{H} \To \mathrm{Vec}_{\QQ}$$
which send $(V_{\B}, V_{\dR}, c)$ to $V_{\B/\dR}$ respectively.

\subsubsection{Variant for number fields}  \label{sect: TannakianNumberFields} Let $k$ be a number field. Let $\mathcal{H}(k)$ denote the $\QQ$-linear category
whose objects consist  of:  $V_{\dR}$, a finite dimensional  vector space over $k$;  for every embedding $\sigma: k \hookrightarrow \CC$ a finite dimensional   vector space $V_{\B,\sigma}$  over $\QQ$;  a  set of $\CC$-linear isomorphisms
$$c_{\sigma} : V_{\dR}\otimes_{k,\sigma}\CC \overset{\sim}{\To} V_{\B, \sigma} \otimes_{\QQ} \CC\ .$$
They are equipped with weight and Hodge filtrations that make every $V_{\B,\sigma}$ into a graded-polarizable $\QQ$-mixed Hodge structure.
They also possess real Frobenius isomorphisms
$F_{\infty} :  V_{\B, \sigma} \overset{\sim}{\rightarrow} V_{\B, \overline{\sigma}}$ which are compatible with complex conjugation, i.e., such that the following diagram commutes:
    
     $$\xymatrix{
    V_\dR\otimes_{k,\sigma}\CC \ar[r]^{c_\sigma} \ar[d]_{\id\otimes c_\dR} & V_{\B,\sigma}\otimes_\QQ\CC \ar[d]^{F_\infty\otimes c_\B}\\
    V_\dR\otimes_{k,\overline{\sigma}}\CC \ar[r]_{c_{\overline{\sigma}}} & V_{\B,\overline{\sigma}}\otimes_\QQ\CC
    }$$
This category is Tannakian  with  fiber functors $\omega_{\dR}: \mathcal{H}(k) \rightarrow \mathrm{Vec}_{k}$ and $\omega_{\B,\sigma}: \mathcal{H}(k) \rightarrow \mathrm{Vec}_{\QQ}$ for every $\sigma$. It is neutralised by the latter. Note that $\mathcal{H}= \mathcal{H}(\QQ)$.

\begin{example} \label{ex: cohomXAB} Let $X, A, B$ satisfy $(\star)_k$.  For any $r$,
let 
$$H_{\dR} =H^r_{\dR}(X\backslash A \lmod B) \quad \hbox{ and } \quad  H_{\B,\sigma}=H^r_{\B,\sigma}(X\backslash A \lmod B)$$
be  the algebraic de Rham and  relative singular cohomology with respect to an embedding $\sigma: k \hookrightarrow \CC$, as defined in \S\ref{sect: Framework}.  
The comparison isomorphism 
$$\mathrm{comp}_{\sigma} : H_{\dR} \otimes_{k, \sigma} \CC \overset{\sim}{\To}   H_{\B,\sigma}\otimes_{\QQ} \CC$$
was defined by Grothendieck \cite{Grothendieck} in the basic case, building upon work of de Rham \cite{derhamthese}. Deligne \cite{delignehodge2, delignehodge3} showed that the above data has a natural graded-polarizable mixed Hodge structure, and hence the object
$$H^r(X\backslash A \lmod B) = \left((H_{\B,\sigma})_{\sigma} , H_{\dR} ,  (\mathrm{comp}_{\sigma})_{\sigma}\right)$$
is an object of $\mathcal{H}(k)$. The real Frobenius involution $F_{\infty}: H_{\B,\sigma} \overset{\sim}{\rightarrow} H_{\B,\overline{\sigma}}$ is induced by complex conjugation  $X_{\overline{\sigma}}(\CC) \overset{\sim}{\rightarrow} X_{\sigma}(\CC)$. 
\end{example}

\begin{remark} To define single-valued periods, one only requires the existence and properties of the real Frobenius, and can drop the filtrations $W$ and $F$. However, these filtrations play an essential role in defining the de Rham projection (see \S\ref{sect: Separated}).
\end{remark}

\subsection{\texorpdfstring{$\mathcal{H}$}{H}-periods}
Our results can be phrased in terms of matrix coefficients in the Tannakian category $\mathcal{H}$. We refer to \cite[\S 2]{brownnotesmot} for more details on this formalism.

\begin{definition} 
An $\mathcal{H}$-period is a regular function on the torsor of tensor isomorphisms from the fiber functor $\omega_\dR$ to the fiber functor $\omega_\B$. The space of $\mathcal{H}$-periods is denoted~$\mathcal{P}^\mm_{\mathcal{H}}=\mathcal{O}( \mathrm{Isom}_{\mathcal{H}}^{\otimes}(\omega_{\dR}, \omega_\B))$.
\end{definition}

Concretely, $\Pe^\mm_{\mathcal{H}}$ is the $\QQ$-vector space spanned by matrix coefficients denoted by a symbol
$$[ H,  \gamma, \omega]^{\mm}$$
where $H$ is an object of $\mathcal{H}$, $\gamma \in H_\B^{\vee}$ and $\omega\in H_{\dR}$, modulo the following relations: bilinearity in $\gamma$, $\omega$ and naturality with respect to morphisms in $\mathcal{H}$. The space  of $\mathcal{H}$-periods  $\Pe^{\mm}_{\mathcal{H}}$ is naturally a ring, equipped with a period homomorphism 
$$\mathrm{per} : \Pe^{\mm}_{\mathcal{H}} \To \CC$$
which on generators is defined by 
$$\mathrm{per}\, [ H,  \gamma, \omega]^{\mm} = \gamma( c(\omega))$$
where $H= (H_\B, H_{\dR}, c).$ It can be thought of as the integral of $\omega$ along $\gamma$. Thus, $\mathcal{H}$-periods should be thought of as `formal' lifts of periods. The period map is far from injective but its restriction to the subring of $\mathcal{H}$ generated by objects of the form given in Example \ref{ex: cohomXAB} for $k=\mathbb{Q}$ is hoped to be injective (see \S\ref{subsec: motivic terminology}).\medskip

The ring $\Pe^{\mm}_{\mathcal{H}}$ admits an involution, denoted by $F_{\infty}$, which acts on matrix coefficients via the formula $F_{\infty} [ H,  \gamma, \omega]^{\mm}  = [ H,  F_{\infty}^{\vee}\gamma, \omega]^{\mm} $. By \eqref{eq: compatibility frobenius comp c definition} it corresponds to complex conjugation:
$$ \mathrm{per} (F_{\infty} \xi ) =  \overline{\mathrm{per} (\xi) } \qquad \hbox{ for all } \xi \in \Pe^{\mm}_{\mathcal{H}}\ .$$

\begin{remark}
The superscript $\mathfrak{m}$ stands for `motivic' because the category $\mathcal{H}$ \emph{formally plays the same role} as a category of motives in the context of periods.  We refer the reader to \S\ref{subsec: motivic terminology} for a discussion of this abuse of terminology.
\end{remark}

\begin{example}\label{example cauchy} (Cauchy's theorem) Let $X = \PP^1$, $A=  \{0, \infty\}$, and $B = \varnothing$. Then 
$$H^1(X\backslash A \lmod B) = ( \QQ, \QQ, 1 \mapsto 2 \pi i)  =: \QQ(-1)$$
is the Lefschetz mixed Hodge structure. Let $\gamma$ denote a positively oriented circle around $0$ in $X\min A(\CC)=\CC^\times$ and  define the Lefschetz $\mathcal{H}$-period 
$$\LL^{\mm} = \left[ \QQ(-1), [\gamma], \left[\textstyle{\frac{dz}{z}}\right]\right]^{\mm}\ .$$
It is the `motivic' version of $2\pi i$ since its period is $\mathrm{per}(\LL^\mm)= 2\pi i.$ 
Since $\overline{\gamma} = - \gamma$, we verify that  $F_{\infty} \LL^{\mm} = - \LL^{\mm}.$ 
\end{example} 

\subsection{de Rham \texorpdfstring{$\mathcal{H}$}{H}-periods} \label{sect: dR periods} A de Rham $\mathcal{H}$-period is defined in an analogous way to $\mathcal{H}$-periods. The de Rham Galois group of the Tannakian category $\mathcal{H}$ is the group scheme of tensor automorphisms of the fiber functor $\omega_\dR$:
$$G^\dR_{\mathcal{H}}=\mathrm{Aut}_{\mathcal{H}}^{\otimes}(\omega_{\dR})\ .$$
The ring of de Rham $\mathcal{H}$-periods $\Pe^{\mm,\dR}_{\mathcal{H}}$ is by definition the affine ring of $G^\dR_{\mathcal{H}}$. Concretely, it is the $\QQ$-vector space  spanned by matrix coefficients denoted by
$$[ H,  f, \omega]^{\mm, \dR}$$
where $f\in H_{\dR}^{\vee}$ and $\omega \in H_{\dR}$, modulo the relations of bilinearity in $f$, $\omega$ and naturality. 

Since $\Pe^{\mm,\dR}_{\mathcal{H}}$ is the ring of functions on a group scheme it is endowed with a structure of a Hopf algebra whose coproduct is given by the formula
$$\Pe^{\mm,\dR}_{\mathcal{H}} \To \Pe^{\mm,\dR}_{\mathcal{H}} \otimes_\QQ \Pe^{\mm,\dR}_{\mathcal{H}} \quad  , \;[H,f,\omega]^{\mm,\dR} \mapsto  \sum_i\,[H,f,e_i^\vee]^{\mm,\dR} \otimes [H,e_i,\omega]^{\mm,\dR} $$
where the sum is over a basis $e_i$ (and $e_i^\vee$ denotes the dual basis)
of $H_{\dR}$.
It does not depend on the choice of basis.
In the same vein, we have a coaction of the Hopf algebra $\Pe^{\mm,\dR}_{\mathcal{H}}$ on the algebra of $\mathcal{H}$-periods, which is given by the similar formula
\begin{equation*}\label{eq: motivic coaction general formula}
\Pe^{\mm}_{\mathcal{H}} \To \Pe^{\mm}_{\mathcal{H}} \otimes_\QQ \Pe^{\mm,\dR}_{\mathcal{H}} \quad  , \;[H,\gamma,\omega]^{\mm} \mapsto  \sum_i\,[H,\gamma,e_i^\vee]^{\mm} \otimes [H,e_i,\omega]^{\mm,\dR} \ .
\end{equation*}

  \subsection{Single-valued periods}
  Similarly, there is a notion of Betti $\mathcal{H}$-periods (which shall not be used in this paper), and an associated Tannaka group
$G^\B_{\mathcal{H}} = \mathrm{Aut}_{\mathcal{H}}^{\otimes} (\omega_\B)$. 
The isomorphism $c$ which forms part of the data of an object of $\mathcal{H}$ (or, equivalently, the  period homomorphism $\mathrm{per}$) gives an isomorphism of affine group schemes 
\begin{equation} \label{GBisomGdR} G_{\mathcal{H}}^\B \times_{\QQ} \CC \overset{\sim}{\To} G_{\mathcal{H}}^{\dR} \times_{\QQ} \CC \ . \end{equation}
The real Frobenius $F_{\infty}$ defines an automorphism of the fiber functor $\omega_\B$ and hence a point  in  $G^\B_{\mathcal{H}} (\QQ)$ which we denote by $F_{\infty}$, without risk of confusion.
\begin{definition} The element $\s \in G^{\dR}_{\mathcal{H}}(\CC)$ is  the image of $F_{\infty} 
\in G^\B_{\mathcal{H}} (\QQ)$ under the canonical isomorphism (\ref{GBisomGdR}).
\end{definition}

The action of the element  $\s$  is  given by the composition 
$$\s : H_{\dR} \otimes_{\QQ} \CC \overset{c}{\To}  H_\B \otimes_{\QQ} \CC \overset{F_{\infty}\otimes \id}{\To} H_\B \otimes_{\QQ} \CC \overset{c^{-1}}{\To} H_{\dR} \otimes_{\QQ} \CC\ .$$
It follows from the compatibility \eqref{eq: compatibility frobenius comp c definition} that we have 
$$(\id\otimes c_\dR) \circ \s = \s \circ (\id\otimes c_\dR)$$
where $c_\dR$ denotes complex conjugation, and therefore $\s$ is real-valued: $\s \in G^{\dR}_{\mathcal{H}}(\RR)$. Furthermore, since $F_{\infty}$ is an involution, we have $\s^2 = 1$.  The element $\s$ is functorial and respects the weight filtration.  

\begin{definition}
The single-valued period homomorphism 
$$\s : \mathcal{P}^{\mm,\dR}_{\mathcal{H}} \To \RR$$
is defined by $\s \in \mathrm{Hom}(\Or(G^{\dR}_{\mathcal{H}}), \RR).$ On matrix coefficients it is given by  $$\s [ H, f, \omega]^{\mm,\dR}  = f(  \s (\omega)).$$
\end{definition}

\begin{example} 
The  Lefschetz de Rham $\mathcal{H}$-period  is
 $\LL^{\mm,\dR} = [\QQ(-1), [\textstyle{\frac{dz}{z}}]^{\vee}, [\textstyle{\frac{dz}{z}}]  ]^{\mm,\dR}$. 
Since $F_{\infty}$ acts via $-1$ on $\QQ(-1)_\B$, 
it follows from the definition that  $\s(\LL^{\mm,\dR}) =-1 $.\end{example}

If $P$ denotes the matrix of $c$ with respect to  suitable bases of $H_{\dR}, H_{\B}$, then the map $\s$ in the same basis of $H_{\dR}$ is represented by the matrix
$$P_{\s} = P^{-1} F_{\infty} P =  \overline{P}^{-1} P$$
which we call the \emph{single-valued period matrix}. It clearly satisfies $P_{\s}^2= I$, $\operatorname{tr} (P_{\s}) = \operatorname{tr}(F_{\infty})$, $\det(P_\s)=\det(F_\infty)$. 

\begin{proposition}
For an $\mathcal{H}$-period $[H, f, \omega]^{\mm,\dR} \in \Pe^{\mm,\dR}_{\mathcal{H}}$, the single-valued period can be expressed as the following quadratic expression
$$\s \, [ H, f, \omega]^{\mm,\dR} =    \sum_{i} \, \mathrm{per}\left( [H^{\vee}, F_{\infty}\gamma_i^{\vee}, f]^{\mm} \right)  \cdot \mathrm{per} \left([H, \gamma_i, \omega]^{\mm} \right)$$
where 
 the sum is over a basis $\gamma_i $ of $ H^{\vee}_\B$ (and $\gamma_i^{\vee}$ denotes the dual basis). \end{proposition}
\begin{proof}
This is a restatement of the definition using the usual rule for matrix multiplication, and using the fact that the object in $\mathcal{H}$ dual to   $H=(H_{\B}, H_{\dR}, c)$ is 
$H^{\vee} =(H_{\B}^{\vee}, H^{\vee}_{\dR}, c^*)$,  where $c^*$ is the inverse transpose of $c$.
\end{proof}

\begin{remark} \label{rem: motivicsv} A `motivic' version 
$\s^{\mm} : \Pe^{\mm,\dR}_{\mathcal{H}} \rightarrow \Pe^{\mm}_{\mathcal{H}}$ of the single-valued period 
was defined in \cite[(4.3)]{brownnotesmot}. It is given by the same  formula as the previous proposition except that  we drop both occurrences of the word  `per' (see \emph{loc. cit.} Remark 8.1). It satisfies $\mathrm{per}\circ \s^\mm=\s$.
\end{remark}

\subsection{Variant for number fields}\label{par:motivic periods number fields}
The above constructions admit a variant over number fields $k$. Given an embedding $\sigma: k \hookrightarrow \CC$, 
let $\CC_{\sigma}$ denote a copy of the complex numbers $\CC$ with $k$-linear structure given by $\sigma: k \hookrightarrow \CC$. The associated single-valued period homomorphism is obtained as the composite
$$\s_{\sigma}: H_{\dR}\otimes_{k}  \CC_{\sigma}  \overset{c_{\sigma}}{\To}  H_{\B,\sigma}\otimes_{\QQ}\CC \overset{F_{\infty} \otimes \id}{\To}H_{\B,\overline{\sigma}}\otimes_{\QQ}\CC \overset{c^{-1}_{\overline{\sigma}}}{\To}   H_{\dR}\otimes_{k}\CC_{\overline{\sigma}}\ .$$
This single-valued period homomorphism, in the special case of mixed Tate objects, is related  to the `big period map' in  \cite[Definition 2.3]{goncharovexponential}.
Since the $k$-linear structures on $\CC_{\sigma}$ and $\CC_{\overline{\sigma}}$ are different when $\sigma$ is not real, this does \emph{not} in general  induce an isomorphism of fiber functors. It only does so after extending scalars to the ring $\CC_{\overline{\sigma}} \otimes_k   \CC_{\sigma}$. In this case,
since every map in the above is functorial with respect to morphisms and respects the tensor product, we deduce a single-valued period homomorphism:
$$\s_{\sigma} : \Pe^{\mm,\dR}_{\mathcal{H}(k)} \To  \CC_{\overline{\sigma}} \otimes_k   \CC_{\sigma}$$
where $\Pe^{\mm,\dR}_{\mathcal{H}(k)}$ is the Hopf algebra of equivalence classes of matrix coefficients $[H, f, \omega]^{\mm,\dR}$ with $f\in H_{\dR}^{\vee}, \omega \in H_{\dR}$.  Equivalently, it defines a point $\s_{\sigma} \in G^{\dR}_{\mathcal{H}(k)}(\CC_{\overline{\sigma}}\otimes_k \CC_{\sigma})$ where  $G^{\dR}_{\mathcal{H}(k)} = \mathrm{Aut}^{\otimes}_{\mathcal{H}(k)}(\omega_{\dR})$ is an affine group scheme over $k$.
Compatibility with complex conjugation (which acts on $\CC_{\overline{\sigma}} \otimes_k   \CC_{\sigma}$ by $\overline {x\otimes y} = \overline{y} \otimes \overline{x}$)  implies that
\begin{equation} \label{ssigmarelation} \overline{\s_{\sigma}} = \s_{\overline{\sigma}}\ .
\end{equation} 
If $P_{\sigma}$ (resp. $P_{\overline{\sigma}}$) are  matrix representatives for  $c_{\sigma}$ (resp. $c_{\overline{\sigma}}$) with respect to suitable bases, then the map $\s_{\sigma}$ is represented by the matrix 
\begin{equation} \label{Pssigmaformula}  P_{\s_{\sigma}}= P_{\overline{\sigma}}^{-1} F_{\infty} P_{\sigma}= \overline{P_{\sigma}}^{-1} P_{\sigma}
\end{equation} 
with values   in  $ \CC_{\overline{\sigma}} \otimes_k   \CC_{\sigma} $.
It satisfies the following `single-valued period relations'
\begin{equation} P_{\s_{\sigma}} \overline{P_{\s_{\sigma}}}= I 
\end{equation} 
where $I$ is the identity matrix. 
This follows from (\ref{ssigmarelation}) and  the identity $\s_{\sigma} \s_{\overline{\sigma}}=\mathrm{id}$, which holds since $F_{\infty}$ is an involution (and follows  directly from (\ref{Pssigmaformula})).

In the case when $\sigma: k \hookrightarrow \RR$ is a real embedding, the map $\s_{\sigma}$ defined above does in fact define an isomorphism of fiber functors, and therefore defines a point 
\begin{equation} \label{ssigmapointonGdrReal} \s_{\sigma} \in G^{\dR}_{\mathcal{H}(k)}(\RR)  \quad \hbox{ satisfying } \quad  \s_{\sigma}^2=1 \ .
\end{equation}
where it is understood that $\mathbb{R}$ is viewed as a $k$-algebra via $\sigma$. In this case the period matrix $P_{\s_{\sigma}}$ has real entries and  satisfies 
\begin{equation} \label{Pequations} 
P_{\s_{\sigma}}^2= I \qquad \hbox{ and  } \qquad  \operatorname{tr} ( P_{\s_{\sigma}}) =\operatorname{tr} (F_{\infty}) \;\, , \; \det( P_{\s_{\sigma}}) =\det(F_{\infty})
\end{equation} 
since $P_{\s_{\sigma}} = P_{\sigma}^{-1} F_{\infty} P_{\sigma}$ is conjugate to $F_{\infty}.$

\begin{remark} 
Let $k$ be a number field. Objects of the category $\mathcal{H}(k)$ can be viewed as families of objects
 over $\mathrm{Spec}\, k$ \cite[\S 10.1]{brownnotesmot} . In particular, the comparison isomorphisms for all embeddings can be expressed as a canonical isomorphism:
 $$\  H_{\dR}\otimes_\QQ\CC \simeq \left(\bigoplus_{\sigma: k \hookrightarrow \CC} H_{\B,\sigma}\right)\otimes_\QQ\CC\ .$$
 The real Frobenius involution $F_{\infty}$ acts on the right-hand side and induces a single-valued involution  $\s$ on the left-hand side.   For the reasons given above, it does not induce a tensor isomorphism of $\omega_{\dR}\otimes_{\QQ} \CC$ in general. 
 
 However, in the case when $k$ is totally real, it does.  Let 
 $R_{k/\QQ} \, G^{\dR}_{\mathcal{H}(k)}$  denote the restriction of scalars from $k$ to $\QQ$ of the affine group scheme 
  $G^{\dR}_{\mathcal{H}(k)}=\mathrm{Aut}^{\otimes}_{\mathcal{H}(k)} \omega_{\dR}$, which is defined over $k$. It satisfies  $R_{k/\QQ} \,G^{\dR}_{\mathcal{H}(k)}(\RR)= G^{\dR}_{\mathcal{H}(k)}(\RR\otimes_{\QQ} k)$.  We deduce that in this case 
   the real Frobenius therefore  defines a point
 $$\s \in R_{k/\QQ} \,G^{\dR}_{\mathcal{H}(k)}(\RR)\  \qquad \qquad  \hbox{for } k \hbox{ totally real} \ ,$$
 which is given by the collection $(\s_{\sigma})_{\sigma: k\hookrightarrow \RR}.$ 
\end{remark}

\subsection{Mixed Tate case}\label{par:MT case}
Let $k$ be a number field and let $\MT(k)$ denote the category of mixed Tate motives over $k$ \cite{levinetatemotives}. Since the Hodge realisation functor is fully faithful \cite[Proposition 2.14]{delignegoncharov}, it embeds as a   full subcategory of 
$\mathcal{H}(k)$. It  has two particularities:

\begin{enumerate}[(i)]
\item It admits a canonical fiber functor $\omega_{\can}: \MT(k) \rightarrow \mathrm{Vec}_{\QQ}$ with the property that $M_{\dR}= M_{\can} \otimes_{\QQ} k$.
\item  This functor (and hence the de Rham realisation) is graded in even degrees. 
\end{enumerate}

Property $(i)$ implies that for every embedding $\sigma: k \hookrightarrow \CC$ there is  a comparison isomorphism $c_{\sigma}:  M_{\can} \otimes_{\QQ} \CC \cong M_{\B,\sigma} \otimes_{\QQ} \CC$. The single-valued involution 
$$\s_{\sigma}: M_{\can}\otimes_{\QQ}  \CC \overset{c_{\sigma}}{\To}  M_{\B,\sigma}\otimes_{\QQ}\CC \overset{F_{\infty} \otimes \id}{\To}M_{\B,\overline{\sigma}}\otimes_{\QQ}\CC \overset{c^{-1}_{\overline{\sigma}}}{\To}   M_{\can}\otimes_{\QQ}\CC $$
is therefore an isomorphism of fiber functors and hence defines a point 
$$\s_{\sigma} \in G^{\can}_{\MT(k)} (\CC)$$
where $G^{\can}_{\MT(k)} = \mathrm{Aut}^{\otimes}_{\MT(k)}(\omega_{\can})$ is the Tannaka group associated to the canonical fiber functor. Its affine ring $\Pe_{\MT(k)}^{\mm,\can}$ is defined by matrix coefficients $[M, f,v]^{\can}$ where $f\in M_{\can}^{\vee}$ and $v \in M_{\can}$ subject to the usual relations (bilinearity, naturality). The key point is that these relations are bilinear in $v$, $f$ over $\QQ$ (as opposed to $k$), and hence the single-valued involution defines a homomorphism $\s_{\sigma}: \Pe^{\mm,\can}_{\MT(k)} \rightarrow \CC$.  As before, we have $\s_{\sigma} \s_{\overline{\sigma}}=1$ and $\overline{\s_{\sigma}} = \s_{\overline{\sigma}}.$

Property $(ii)$ implies that one can define 
a different element,  which is denoted by $\sv_{\sigma}$ (without the subscript $\sigma$ in the case when $k = \QQ$) which satisfies $\sv_{\sigma}( \LL^{\mm,\dR})=1$ \cite{brownSVMZV}. 
It is the element of $G^{\can}_{\MT(k)}(\CC)$ defined as the product $\sv_\sigma = \tau(-1)\s_\sigma$, where $\tau(-1)\in G^{\can}(\QQ)$ is multiplication by $(-1)^w$ in weight $2w$. Stated differently,  
let $\GG_m$ denote the group of automorphisms of $\QQ(-1)_{\can}$. The action of $G^{\can}$ on  $\QQ(-1)_{\can}$ defines a  morphism of affine group schemes $
G^{\can}_{\MT(k)} \rightarrow \GG_m.$
 The image of $\s_{\sigma}$ under this map is $-1$; the image of $\sv_{\sigma}$ is the identity. This implies that $\sv_{\sigma}$ in fact lies in  the complex points of the unipotent radical  $U^{\can}_{\MT(k)} \leq  G^{\mathrm{can}}_{\MT(k)} $.

 \begin{example}
  Let $n\geq 1$ and let $E$ be an extension of $\QQ(-n)$ by $\QQ(0)$ in the category $\MT(k)$.  Its period matrices, with respect to the graded $\omega_{\can}$ basis $e_{0}, e_n$, and a suitable basis of $E_{\B,\sigma}$  compatible with the weight filtration, are of the form 
 $$P_{\sigma} =  \begin{pmatrix}  1 & a_{\sigma} \\ 0  &(2\pi i)^n \end{pmatrix} $$
 where $a_{\sigma} \in \CC$ is well-defined up to addition of an element in $(2\pi i)^n \QQ$. The operator $\sv_{\sigma}$  is represented by the matrix
 $$\tau(-1) \overline{P_{\sigma}}^{-1} P_{\sigma} =  \begin{pmatrix}  1 & r_{\sigma} \\ 0  & 1 \end{pmatrix} $$
 where $r_{\sigma} = \sv_{\sigma} [ E, e_0^{\vee}, e_n]^{\can}$ is given by 
 $$r_{\sigma} = 2 \begin{cases}     \mathrm{Re}\,  a_{\sigma} & \hbox{ if } n \hbox{ odd;} \\
  i\, \mathrm{Im}\, a_{\sigma} & \hbox{ if } n \hbox{ even.} \end{cases} \ $$
  On the other hand consider the regulator $\mathrm{reg}_{\sigma}$ \cite[\S 1.6]{beilinsondeligne}: 
  $$\mathrm{reg}_{\sigma}: \mathrm{Ext}^1_{\MT(k)} (\QQ(-n), \QQ(0)) \overset{\sigma}{\To} 
  \mathrm{Ext}^1_{\RR-\mathrm{MHS}} (\RR(-n), \RR(0)) \overset{\sim}{\To} \CC /( 2\pi i)^n \RR  $$
  where the first map assigns to an extension $E$  the extension of $\RR$-mixed Hodge structures $E_{\B, \sigma} \otimes_{\QQ} \RR$.  The last map is the class of $a_{\sigma}$ modulo $(2\pi i)^n \RR$, which is well-defined. Finally, if one identifies
  $\CC /( 2\pi i)^n \RR $ with $i^{n-1} \RR$ via the inclusion of the latter in $\CC$ then one concludes that 
 $$ r_{\sigma}  =2 \, \mathrm{reg}_{\sigma} \ . $$
  In other words,  $2\, \mathrm{reg}_\sigma$ is the image of  $\log( \sv_\sigma)$ in the abelianisation of the  Lie algebra of the unipotent radical $U^{\mathrm{can}}_{\MT(k)}$. 
  In fact, much of \cite{beilinsondeligne} can be profitably recast using the language of unipotent de Rham polylogarithms.
\end{example}

    \subsection{A general motivic setting}\label{subsec: motivic terminology}
            Most of the contructions in this article are in the setting of the Tannakian category $\mathcal{H}(k)$. It plays the role of a category of motives (over $k$, with coefficients in $\mathbb{Q}$), and suffices to  speak of  single-valued periods and the de Rham projection. Any Tannakian category of motives $\mathcal{M}(k)$ should come equipped with a realisation functor
        \begin{equation}\label{eq: hodge realization motives}
        \mathcal{M}(k) \longrightarrow \mathcal{H}(k)
        \end{equation}
        through which the Betti and de Rham fiber functors factor. It is reasonable to hope, by analogy with the classical Hodge conjecture, that this realisation functor is fully faithful. This would imply that the rings of matrix coefficients of $\mathcal{M}(k)$, which are defined in an analogous way to those of $\mathcal{H}(k)$, embed into those of $\mathcal{H}(k)$ (in other words, $ \mathcal{P}^{\mm}_{\mathcal{M}(k)} \rightarrow \mathcal{P}^{\mm}_{\mathcal{H}(k)}$ and $\mathcal{P}^{\mm,\dR}_{\mathcal{M}(k)} \rightarrow \mathcal{P}^{\mm,\dR}_{\mathcal{H}(k)}$ are injective). This motivates our use of the superscript $\mathfrak{m}$, for `motivic'. Note that the objects $H^r(X\backslash A\,\lmod B)\in \mathcal{H}(k)$ defined in Example \ref{ex: cohomXAB} are expected to lift to objects in $\mathcal{M}(k)$, and should generate $\mathcal{M}(k)$ as a Tannakian category. All of our constructions should also naturally lift to a motivic setting.
        
        An example of such a category $\mathcal{M}(k)$ which is unconditionally defined is Nori's category of mixed motives $\mathcal{NM}(k)$. In this case the full faithfulness of \eqref{eq: hodge realization motives} is conjectured in \cite[Conjecture 3.22]{levineK}. By definition, the objects $H^r(X\backslash A\,\lmod B)\in \mathcal{H}(k)$ lift to objects of $\mathcal{NM}(k)$, and generate it as a Tannakian category.
        
        In the mixed Tate case the situation is favorable: the category $\mathcal{MT}(k)$ is defined unconditionally (it should be a full subcategory of any Tannakian category $\mathcal{M}(k)$ as above), and  the realisation functor $\MT(k)\to \mathcal{H}(k)$ is known to be fully faithful. In fact we have already used this fact in the previous paragraph to view $\MT(k)$ as a full subcategory of $\mathcal{H}(k)$. It is conjectured, but not known,  that $\MT(k)$ is a full subcategory of $\mathcal{NM}(k)$, which is one of  the reasons for working primarily with $\mathcal{H}(k)$.

   \subsection{Periods over a general base} It is possible to define motivic periods in families and their associated single-valued periods (see the notes \cite[\S 8.3]{brownnotesmot}). Since the essential idea is the same, but requires a certain amount of technical background, we shall not describe this in any great detail here, only to note that examples such as the logarithm $\log(z)$, and its single-valued version $\log |z|^2$, make perfect sense as a function of the parameter $z$. 
   
   The construction does, however, clarify the origin of the term `single-valued'. For periods over a base $S$, we replace the de Rham component  $H_{\dR}$ with an algebraic vector bundle with integrable connection and regular singularities over $S$, and the Betti component $H_\B$ with a  local system over $S(\CC)$.
    The key point is that an algebraic vector bundle is locally trivial in the Zariski topology, so admits everywhere locally a de Rham basis. 
    Since the single-valued period only depends on  de Rham data, it is necessarily  single-valued (as opposed  to classical periods, which use the Betti local system and inherently have monodromy). As remarked by the reviewer, this is nothing special about $\s$ but applies to any element of $G^{\dR}(\CC)$ (the corresponding automorphism of fiber functors). Alas, the only non-trivial element of $G^{\dR}(\CC)$ which is known is precisely the element $\s$. The point of this paper is to show, in addition, that it is computable in many interesting cases.

    \subsection{Further remarks}
    Let $M$ be an object of $\mathcal{H}$. 
    Its Betti Tannaka  group $G_{\mathcal{H}}^{\B}(M)$ is the image of $G^{\B}_{\mathcal{H}}$ in $\mathrm{GL}(M_\B)$ \cite[\S 3.6]{brownnotesmot}. This group is a variant of a Mumford--Tate group, which is defined to be the Tannaka group  of the category of mixed Hodge structures: more precisely there is a morphism from the Mumford--Tate group of $M_{\B}$ into $G^{\B}_{\mathcal{H}}(M_{\B})$, but it is not necessarily an  isomorphism since the category $\mathcal{H}$ is an enrichment of the category of mixed Hodge structures (it encodes additional data).
    Let $\xi \in \Pe_{\mathcal{H}}^{\mm,\dR}$ denote a de Rham $\mathcal{H}$-period. There exists a minimal object $M(\xi)$ of $\mathcal{H}$ \cite[\S 2.4]{brownnotesmot} associated to $\xi$. It has  the property that $M(\xi)_{\dR}$ is the right $\Or(G_{\mathcal{H}}^{\dR})$-comodule generated by $\xi$ (under the right coaction of $\Pe_{\mathcal{H}}^{\mm,\dR}$ on itself).

     \begin{proposition}
The single valued motivic period matrix  of $M$  (Remark \ref{rem: motivicsv}) has  rational entries if and only if
$F_{\infty}$ is central in  $G_{\mathcal{H}}^{\B}(M)$.

The single-valued motivic period of $\xi$ is rational if and only if $F_{\infty}$ is central in $G_{\mathcal{H}}^{\B}(M(\xi)). $
        \end{proposition} 
        
   \begin{proof}
   Let $c^{\mm}:M_{\dR}  \otimes_{\QQ} \Pe_{\mathcal{H}}^{\mm} \overset{\sim}{\rightarrow} M_\B \otimes_{\QQ} \Pe_{\mathcal{H}}^{\mm}$ denote the universal comparison isomorphism \cite[\S 4.1]{brownnotesmot}. 
   It follows from its definition that the right action of  $G_{\mathcal{H}}^\B$ induced by its action on the coefficient ring $\Pe_{\mathcal{H}}^{\mm}$ (which we shall denoted by $|_g$) is given by  left matrix  multiplication: i.e.,    $ g \, c^{\mm} = c^{\mm}|_g$. 
   The motivic single-valued involution is $\s^{\mm} = (c^{\mm})^{-1} F_{\infty} c^{\mm}$.
     It follows that  $G_{\mathcal{H}}^\B$ acts via  
     $$\s^{\mm}|_{g} =  (c^{\mm})^{-1} (g^{-1} F_{\infty} g )c^{\mm}\ .$$ 
     The element $\s^{\mm}$ is therefore  $G_{\mathcal{H}}^\B$-invariant if and only if $g^{-1} F_{\infty} g= F_{\infty}$ for all $g$, i.e.,   $F_{\infty}$ is central in  $G_{\mathcal{H}}^{\B}(M)$. Since, by the Tannaka theorem, the space of $G_{\mathcal{H}}^{\B}$-invariants of $\mathcal{P}^{\mm}_{\mathcal{H}}$ are exactly $\QQ$, this proves the first part.
     
     For the second part,  note that the left action of $G^{\dR}_{\mathcal{H}}$ on de Rham periods satisfies   $\s^{\mm}(g\,  \xi^{dR})= g^{-1} \s^{\mm}(\xi) g$.  Therefore if $\s^{\mm}(\xi)$ is rational, it follows that all single-valued motivic periods of $M(\xi)$ are too, since the latter is generated by $\xi$ under the action of  $G^{\dR}_{\mathcal{H}}$. Furthermore, $\xi$ is itself a de Rham period of $M(\xi)$, so the second statement follows by applying the first statement to $M(\xi).$
      \end{proof}
  
    The argument shows more generally that the centraliser of $F_{\infty}$ in $G^\B_{\mathcal{H}}$ acts trivially on the motivic version  of the single-valued period matrix. For example, we immediately deduce that the single-valued versions of  (motivic) algebraic numbers with abelian Galois group are rational numbers.

    \begin{remark}
The ring of single-valued periods can be thought of group-theoretically as follows.   Let $T_{\mathcal{H}} = \mathrm{Isom}^\otimes_{\mathcal{H}}(\omega_{\dR}, \omega_{\B}) = \mathrm{Spec}\, \Pe^{\mm}_{\mathcal{H}}$.  There is a morphism $$p \mapsto p^{-1} F_{\infty} p :T_{\mathcal{H}}\To G_{\mathcal{H}}^{\dR}$$
of schemes over $\QQ$. Let $S_{\mathcal{H}} \subset G_{\mathcal{H}}^{\dR}$ denote the Zariski closure of the image. It is contained in the subscheme of elements squaring to the identity.  The morphism $T_{\mathcal{H}} \rightarrow S_{\mathcal{H}}$  is equivariant with respect to the action of $G_{\mathcal{H}}^{\dR}$ by multiplication on the right on $T_{\mathcal{H}}$, and by conjugation on $S_{\mathcal{H}}$. Its 
non-empty fibers are torsors over the centraliser of $F_{\infty}$ in $G_{\mathcal{H}}^\B$, which acts on $T_{\mathcal{H}}$ via $p \mapsto  c p$ for any $c \in G_{\mathcal{H}}^\B$
satisfying $cF_{\infty} = F_{\infty} c.$ In order to discuss what this construction means in terms of relations between single-valued periods we 
replace $\mathcal{H}$ with a Tannakian category of motives $\mathcal{M}=\mathcal{M}(\mathbb{Q})$ as in \S\ref{subsec: motivic terminology} (or alternatively with the subcategory of $\mathcal{H}$ generated by objects $H^r(X\backslash A\,\lmod B)$ as in Example \ref{ex: cohomXAB}) and perform the same formal constructions. It would be desirable to have a description of the kernel $I_{\mathcal{M}}$ of the natural map $\Pe^{\mm,\dR}_{\mathcal{M}}=\Or(G_{\mathcal{M}}^{\dR}) \rightarrow \Or(S_{\mathcal{M}})$, which encodes the relations between motivic lifts of single-valued periods, i.e., it is the ideal of \emph{motivic} relations (relative to $\mathcal{M}$) between single-valued periods. It contains the ideal generated by $\mathbb{L}^{\dR}+1$.
The period conjecture for $\mathcal{M}$ states that the comparison isomorphism should be Zariski dense in $T_{\mathcal{M}}(\mathbb{C})$, and hence the single-valued period $\mathsf{s}$ should also be Zariski dense in $S_{\mathcal{M}}(\mathbb{C})$. The classical period conjecture would imply that  $I_{\mathcal{M}}$ is   the ideal of \emph{all}  relations amongst single-valued periods.  
    \end{remark}

\subsection{Examples in small rank} 
   Let $H= (H_\B, H_{\dR}, c)$ be an object of $\mathcal{H}$.  Let 
     $$H_\B = H_\B^+ \oplus H_\B^-$$
     denote the decomposition of $H_B$ into $F_{\infty}$-eigenspaces. Since the single-valued period matrix $P_{\s}$
     is conjugate to $F_{\infty}$, its trace is $\mathrm{tr}\,  P_{\s} = \dim H_B^+- \dim H_B^-$ and its determinant is  $\det P_{\s} = (-1)^{\dim H_\B^-}.$
    \subsubsection{Rank 1} Let $M$ be of rank one. Then $P_{\s}= \det P_{\s}= \det F_{\infty}=F_{\infty}$.
    In particular,  for any object of rank one, the single valued period is $\pm 1$.

\subsubsection{Rank 2}  Suppose that  $\dim H_\B^+ = \dim H_\B^- =1$ (for example,   $H$ has Hodge numbers $(p,q)$ and $(q,p)$ with $p\neq q$). Denote the period matrix with respect to  a basis of $H_{\dR}$ and a basis of eigenvectors of $H_\B^+ \oplus H^-_\B$ by 
    $$P = \begin{pmatrix} a^+ & b^+ \\ i a^- & i b^- \end{pmatrix} \ $$
    where $a,b\in \RR.$ The single-valued period matrix is 
    $$P_{\s}= \overline{P}^{-1} P =\frac{1}{a^+ b^- - b^+a^- }\begin{pmatrix}  a^+ b^-+a^-b^+  &   2 b^+b^-   \\ -2a^+ a^- &  - (a^+ b^- + a^-b^+) \end{pmatrix} \cdot$$
    It is clearly real-valued, satisfies $P_{\s}^2 = I$, and has vanishing trace. 
    
     For an example,  see \S\ref{par: UnivElliptic}. 
   Other  examples of such objects arise  from modular forms \cite{brownCNHMF3}, although the  quasi-periods of modular forms are not well-known. If one chooses  a basis of $M_{\dR}$ adapted to the Hodge filtration, the terms $a^+a^-$ can be interpreted as the Petersson norm of a modular form. The term $b^+b^-$ can be thought of as a `Petersson norm' of a weakly holomorphic modular form. The diagonal terms $a^+b^-+a^-b^+$ appear as coefficients in the associated weak harmonic lift.

\section{Verdier duality and differential forms} \label{sect: Verdierduality}

	\subsection{Relative cohomology and Verdier duality}\label{subsec: Verdier}
	
		Classical Poincar\'{e} duality allows one to interpret a homology class as a cohomology class of complementary degree. We need the following classical variant \cite[Theorem 2.4.5]{hubermuellerstach} for which we include a proof in order to verify the compatibility with mixed Hodge structures (alternatively one could use [\emph{loc. cit.}] and Hodge realisation for Nori motives).
		
	    \begin{theorem}\label{thm: PDNCD}
	    Let $(X,A,B)$ satisfy $(\star)_\CC$ (defined in  \S\ref{sect: Framework}).  For every integer $r$ there is a perfect pairing of $\QQ$-mixed Hodge structures: 
	    \begin{equation}\label{eq:PDNCDpairing}
	    H^r(X\min B \lmod A) \otimes H^{2n-r}(X\min A \lmod B) \longrightarrow \QQ(-n)\ .
	    \end{equation}
	    \end{theorem}
	    
	    We focus on the algebraic case but nevertheless note that this theorem is more generally true for $X$ a compact complex manifold (without the reference to mixed Hodge structures, and without the Tate twist).
    
        The proof of Theorem \ref{thm: PDNCD} is sheaf-theoretic.  For the rest of \S\ref{subsec: Verdier},  we  abuse  notation and use the same letters to denote algebraic varieties and  their  underlying complex  manifolds.
        It is convenient to compute the relative cohomology of the pair $(X\min A,B\min A\cap B)$ as the cohomology of $X$ with values in the complex of sheaves 
        $$\mathcal{F}(A,B) := R(j_{X\min A}^X)_*(j_{X\min A\cup B}^{X\min A})_!\QQ_{X\min A\cup B}[n]\ ,$$
         where  $j_{U}^Y$ always denotes an open immersion of $U$ into $Y$. We view $\mathcal{F}(A,B)$ as an object of the 
         full subcategory  $\mathrm{D}(X)$ of the bounded derived category of sheaves of $\QQ$-vector spaces on $X$ consisting of complexes of sheaves with constructible cohomology.
         
         The shift by $n$ ensures that $\mathcal{F}(A,B)$ lies in the abelian subcategory $\mathrm{Perv}(X)\subset \mathrm{D}(X)$ of perverse sheaves on $X$ (this is because all the open immersions that we consider are  affine). In order to say something about mixed Hodge structures we  note that $\mathcal{F}(A,B)$ lifts to an object,   denoted by the same symbol, in the category $\mathrm{MHM}(X)$ of mixed Hodge modules over $X$ \cite{saitoMHM}. This is because mixed Hodge modules have a six-functor formalism that lifts that  of perverse sheaves.
        
         \begin{proposition}\label{prop: commutejj}
         Under the assumptions of Theorem \ref{thm: PDNCD}, the natural morphism
        $$  (j_{X\backslash B}^X)_!R(j_{X\backslash A\cup B}^{X\backslash B})_*\QQ_{X\backslash A\cup B}[n] \longrightarrow R(j_{X\backslash A}^X)_*(j_{X\backslash A\cup B}^{X\backslash A})_!\QQ_{X\backslash A\cup B}[n]$$
        is an isomorphism of mixed Hodge modules.
        \end{proposition}
    
        \begin{proof}
        This morphism follows from the adjunction between $(j_{X\backslash A}^X)^*$ and $R(j_{X\backslash A}^X)_*$ applied to the base change isomorphism $(j_{X\backslash A}^X)^*(j_{X\backslash B}^X)_!\simeq (j_{X\backslash A\cup B}^{X\backslash A})_!(j_{X\backslash A\cup B}^{X\backslash B})^*$. Since the functor from mixed Hodge modules to perverse sheaves is exact and faithful, it is enough to check that this morphism is an isomorphism in $\mathrm{D}(X)$. This is a local statement and we may assume that we have $X=Y\times Z$, $A=A'\times Z$ and $B=Y\times B'$ with $A'$ and $B'$ normal crossing divisors in $Y$ and $Z$, respectively. We thus have an isomorphism $\QQ_{X\backslash A\cup B}\simeq \QQ_{Y\backslash A'}\boxtimes \QQ_{Z\backslash B'}$ which induces isomorphisms
        $$(j_{X\backslash B}^X)_!R(j_{X\backslash A\cup B}^{X\backslash B})_*\QQ_{X\backslash A\cup B}\simeq R(j_{Y\backslash A'}^Y)_*\QQ_{Y\backslash A'}\boxtimes (j_{Z\backslash B'}^Z)_!\QQ_{Z\backslash B'}$$
        and 
        $$R(j_{X\backslash A}^X)_*(j_{X\backslash A\cup B}^{X\backslash A})_!\QQ_{X\backslash A\cup B}\simeq R(j_{Y\backslash A'}^Y)_*\QQ_{Y\backslash A'}\boxtimes (j_{Z\backslash B'}^Z)_!\QQ_{Z\backslash B'}\ .$$
        The morphism in the statement of the proposition is easily seen to be compatible with these isomorphisms, and the proof is complete.
        \end{proof}
        
        For $X$ a complex algebraic variety, the Verdier duality functor $\mathbb{D}_X$ on $\mathrm{D}(X)$ restricts to an involution of $\mathrm{Perv}(X)$ that lifts to $\mathrm{MHM}(X)$ \cite{saitoMHM}.  If $f:X\rightarrow Y$ is a morphism of complex algebraic varieties then we have isomorphisms $Rf_*\circ \mathbb{D}_X \simeq \mathbb{D}_Y\circ Rf_!$ and $Rf_!\circ \mathbb{D}_X \simeq \mathbb{D}_Y\circ Rf_*$ in $\mathrm{D}(X)$. If $Y$ is a point then $\mathbb{D}_Y$ is simply linear duality and so applying $\mathbb{H}^r(Y,-)$ to the first isomorphism gives
        \begin{equation}  \label{HkDxFHc} \mathbb{H}^r(X,\mathbb{D}_X\mathcal{F}) \simeq \mathbb{H}_\c^{-r}(X,\mathcal{F})^\vee\ .
        \end{equation} 
        In this setting, and in the case when $X$ is smooth,  classical Poincar\'{e} duality follows from the isomorphism $\mathbb{D}_X\QQ_X[n]\simeq \QQ_X[n](n)$.

         \begin{proof}[Proof of Theorem \ref{thm: PDNCD}] Since $R(j_{X\min A}^X)_! \simeq (j_{X\min A}^X)_!$ is exact,
        the compatibility between Verdier duality and  pushforward functors gives an isomorphism in $\mathrm{MHM}(X)$:
        $$\mathbb{D}_X\mathcal{F}(A,B) \simeq (j_{X\backslash A}^X)_!R(j_{X\backslash A\cup B}^{X\backslash A})_*\QQ_{X\backslash A\cup B}[n](n) \overset{\ref{prop: commutejj}}{\simeq} \mathcal{F}(B,A)(n)\ ,$$
        The  statement of the theorem follows from (\ref{HkDxFHc}) since $X$ is compact.
        \end{proof}
        
        \begin{remark}
        The above argument shows that for $A$, $B$ not necessarily normal crossing, there is always a natural  map
        $$H^r(X\min B \lmod A) \To H^{2n-r}(X\min A \lmod B)^{\vee}(-n)\ .$$
        It is not an isomorphism in general, as the following counter-example shows. Let $X = \PP^2(\CC)$, $A \cong \PP^1(\CC)$ be  a line and $B \cong \PP^1(\CC) \cup \PP^1(\CC)$ be the union of two other lines, such that all three lines meet at a point. On the one hand, $H^\bullet(X\min A \;\mathrm{mod}\; B)$ is the cohomology of $\CC^2$ relative to two parallel lines, which is the cohomology of $\CC$ relative to two points. This is concentrated in degree $1$ with $H^1\simeq \QQ(0)$. On the other hand, $H^\bullet(X\min B \;\mathrm{mod}\; A)$ is the cohomology of $\CC\times\CC^*$ relative to a line $\CC \times \{\mathrm{pt}\}$, which is the cohomology of $\CC^*$ relative to a point. This is concentrated in degree $1$ with $H^1\simeq \QQ(-1)$.
        \end{remark}
        
        \begin{remark}
        The proof of Proposition \ref{prop: commutejj} shows that  the same result  holds if $(X,A,B)$ is `locally a product'. We focus on the normal crossing situation in order to obtain explicit formulae for the pairing in terms of differential forms.
        \end{remark}
       
    \subsection{Verdier duality and differential forms}\label{subsec: verdier diff}
    We want to interpret the pairing \eqref{eq:PDNCDpairing}, after extending  coefficients to $\CC$, in terms of differential forms. We first discuss a simpler situation. Let $X$ be a  complex manifold of dimension $n$ (not necessarily compact), and let $j:U\hookrightarrow X$ be an open. We have an isomorphism in $\mathrm{D}(X)$:
    $$ \mathbb{D}_X(j_!\CC_U) \simeq Rj_*\CC_U[2n]\ ,$$
    which induces a perfect pairing of complex vector spaces:
	\begin{equation}\label{eq:PDsimpler}
	\mathbb{H}^{r}(X,j_!\CC_U)\otimes \mathbb{H}^{2n-r}_\c(X,Rj_*\CC_U) \rightarrow \CC 
	\end{equation}
	by using  (\ref{HkDxFHc}).
	
		\begin{remark}
		The group $\mathbb{H}^\bullet(X,j_!\CC_U)$ is by definition the relative cohomology of the pair $(X,X\min U)$. 
		For an interpretation in classical terms of the group $\mathbb{H}_\c^\bullet(X,Rj_*\mathbb{C}_U)$ it is convenient to choose a compactification $\overline{X}$ of $X$ and write it as $\mathbb{H}^\bullet(\overline{X},(j_X^{\overline{X}})_!Rj_*\mathbb{C}_U)$. In the appropriate geometric framework (which can be achieved by resolution of singularities) one is thus in the setting of Proposition \ref{prop: commutejj} which allows us to interpret $\mathbb{H}_\c^\bullet(X,Rj_*\mathbb{C}_U)$ as a relative cohomology group.
		\end{remark}
	
	We now interpret the pairing \eqref{eq:PDsimpler} in terms of differential forms.  Recall that the classical Poincar\'{e} lemma gives a quasi-isomorphism $\CC_U\simeq \A^\bullet_U$, where $\A^\bullet_U$ denotes the complex of smooth complex-valued differential forms on $U$. The existence of partitions of unity for smooth functions implies that every sheaf of $\mathcal{C}^\infty_X$-modules is soft on a complex manifold $X$ \cite[III.2.9, III.3.9]{iversen}.
	This is the case for all sheaves of smooth forms that  arise in the rest of this section. In particular these sheaves are acyclic with respect to the global section functors $\Gamma$, $\Gamma_\c$  with and without compact support \cite[III.2.7, IV.2.2]{iversen}. 
	
		\begin{remark}
		 Note that the perfect pairing $\mathbb{H}^r_\c(X,j_!\CC_U)\otimes \mathbb{H}^{2n-r}(X,Rj_*\CC_U) \rightarrow \CC$, where we have switched $\mathbb{H}$ and $\mathbb{H}_\c$, is easy to describe: we have $\mathbb{H}^r_\c(X,j_!\CC_U)\simeq H^r_\c(U)$ and $ \mathbb{H}^{2n-r}(X,Rj_*\CC_U)\simeq H^{2n-r}(U)$, and the pairing is nothing but the Poincar\'{e} duality pairing for $U$.
		\end{remark}	
	
	We obtain quasi-isomorphisms
    $$j_!\CC_U\simeq j_!\A^\bullet_U \;\;\textnormal{ and }\;\; Rj_*\CC_U\simeq j_*\A^\bullet_U\ ,$$
    where the second quasi-isomorphism uses the fact that  each $\A^r_U$ is a soft sheaf. The pairing \eqref{eq:PDsimpler} can then be viewed as a pairing
    		$$ H^r(\Gamma(X,j_!\A^\bullet_U))\otimes H^{2n-r}(\Gamma_\c(X,j_*\A^\bullet_U)) \rightarrow \CC\ .$$
    		Recall that classical Poincar\'{e} duality is induced by the integration of the cup-product of smooth differential forms. From this and the general formalism of Verdier duality it follows that this last pairing is induced by the pairing of complexes
    		$$\Gamma(X,j_!\A^\bullet_U)\otimes \Gamma_\c(X,j_*\A^\bullet_U) \rightarrow \CC[2n] \;\; , \;\; \nu\otimes\omega \mapsto \int_X \nu\wedge\omega\ .$$
		Note that this  integral is well-defined: for $\nu$ a section of $j_!\A^r_U$ and $\omega$ a section of $j_*\A^{2n-r}_U$ with compact support, the wedge product $\nu\wedge\omega$ extends to a smooth differential form of degree $2n$ on $X$ with compact support, and can be integrated on $X$.\medskip

		Unfortunately, this description of the pairing \eqref{eq:PDsimpler} is impractical. Indeed, the complex $\Gamma(X,j_!\A^\bullet_U)$ consists of forms on $X$ that vanish in the neighbourhood of every point of $X\min U$. Such forms are typically constructed using partitions of unity and do not arise naturally in an algebraic context. We now give (in a special case) a description of the pairing \eqref{eq:PDsimpler} that will allow us to work with more general differential forms.
		
	\subsection{The case of one normal crossing divisor}\label{par:caseonencd}

		 In this subsection $X$ is a (not necessarily compact) complex manifold, $D$ is a normal crossing divisor in $X$, and $j:U\hookrightarrow X$ denotes its complement. In local charts, $X=\{(z_1,\ldots,z_n)\in\CC^n\; , \; |z_i|<1\}$ is a polydisk and $D=\{z_1\cdots z_s=0\}$ is a union of coordinate hyperplanes. We can make use of the complex of sheaves of (smooth) logarithmic forms $\A^\bullet_X(\log D)$, whose sections are locally given as linear combinations of forms
		$$\alpha\wedge \frac{dz_{i_1}}{z_{i_1}}\wedge\cdots\wedge\frac{dz_{i_m}}{z_{i_m}}\ ,$$
		with $1\leq i_1<\cdots <i_m\leq s$ and $\alpha$ a smooth complex-valued differential form on $X$. Also, for $\mathcal{F}$ a sheaf of $\mathcal{O}_X$-modules on $X$, we denote by $\mathcal{F}(-D)$ the subsheaf of $\mathcal{F}$ whose local sections are of the form $(z_1\cdots z_s)\omega$ for $\omega$ a local section of $\mathcal{F}$.

		We let $p:\widetilde{X}\rightarrow X$ denote the iterated real-oriented blow-up along the irreducible components of $D$ \cite{gillamrealoriented}. 
		It is an isomorphism above $X\min D$ and replaces each irreducible component of $D$ by the sphere bundle of its normal bundle. The result is that  $\widetilde{X}$ is a smooth manifold with corners \cite{joycecorners}. In concrete terms, computing the pullback of a form by $p$ locally amounts to passing to polar coordinates $z_k=\rho_ke^{i\theta_k}$ around each irreducible component $z_k=0$ of $D$. 
		
		\begin{definition}\label{defi:polar smooth}
		A \emph{polar-smooth form} on $(X,D)$ is a smooth form on $X\min D$ which extends to a smooth form on $\widetilde{X}$.
		\end{definition}

		The sheaf of polar-smooth $r$-forms is simply the pushforward $p_*\A^r_{\widetilde{X}}\hookrightarrow j_*\A^r_{X\min D}$.

		\begin{example}
		For $(X,D)=(\CC,\{0\})$, the forms
		\begin{equation}\label{eq:big wedge product polar}
		\frac{dz\wedge d\overline{z}}{\overline{z}}= -2i\,e^{i\theta} d\rho\,d\theta \;\;\;\textnormal{ and }\;\;\; \frac{z}{\overline{z}} \,d\overline{z} = e^{i\theta}(d\rho-i \rho \, d\theta)\ . 
		\end{equation}
		are polar-smooth but do not extend to smooth forms on $X$.
		\end{example}
		
		\begin{lemma}\label{lem:wedge product integrable simpler}
		For $\nu$ a section of $\A^\bullet_X(\log D)(-D)$ and $\omega$ a section of $\A^\bullet_X(\log D)$, the wedge products $\nu\wedge \overline{\omega}$ and $\overline{\nu}\wedge\omega$ are polar-smooth forms on $(X,D)$.
		\end{lemma}

		\begin{proof}
		We shall write out  the proof for $\nu\wedge\overline{\omega}$.  The other case follows by applying complex conjugation.  Since the statement is local,  we reduce to the case where $X=\{(z_1,\ldots,z_n)\in\CC^n\; , \; |z_i|<1\}$ is a polydisk and $D=\{z_1\cdots z_s=0\}$ is a union of coordinate hyperplanes. We write
		$$\nu = (z_1\cdots z_s)\alpha\wedge\bigwedge_{i\in I}\frac{dz_i}{z_i}=\left(\prod_{i\in\{1,\ldots,s\}\min I}z_i\right)\alpha\wedge\bigwedge_{i\in I} dz_i\;\;\textnormal{ and  }\;\; \omega = \beta\wedge \bigwedge_{j\in J}\frac{dz_j}{z_j} $$
		with $I$ and $J$ subsets of $\{1,\ldots,s\}$ and $\alpha$, $\beta$ smooth forms on $X$. We can thus write $\nu\wedge\overline{\omega}$ as a wedge product
		\begin{equation}\label{eq:big wedge product}
		\varphi \wedge \left(\bigwedge_{i\in I\cap J}\frac{dz_i\wedge d\overline{z}_i}{\overline{z}_i} \right)\wedge\left(\bigwedge_{i\in J\min I} \frac{z_i}{\overline{z}_i}\,d\overline{z}_i\right)\ .
		\end{equation}	
		with $\varphi$ a smooth form on $X$. By \eqref{eq:big wedge product polar} they are polar-smooth forms on $(X,D)$.
		\end{proof}

		\begin{proposition}\label{prop:PD simpler log forms conjugate}
		\begin{enumerate}[1.]
		\item We have quasi-isomorphisms
		$$j_!\A^\bullet_U \stackrel{\sim}{\hookrightarrow} \A^\bullet_X(\log D)(-D) \;\; \textnormal{\textit{ and } } \;\; \overline{\A^\bullet_X(\log D)} \stackrel{\sim}{\hookrightarrow} j_*\A^\bullet_U \ .$$
		\item With  these identifications, the pairing \eqref{eq:PDsimpler} is induced by the pairing of complexes
		\begin{eqnarray*} \Gamma(X,\A^\bullet_X(\log D)(-D)) \otimes \Gamma_\c(X,\overline{\A^\bullet_X(\log D)}) & \To & \CC[2n] \\
		\nu\otimes\overline{\omega} & \mapsto & \int_X\nu\wedge\overline{\omega}\ .
		\end{eqnarray*}
		\end{enumerate}
		\end{proposition}
		
		\begin{proof}
		\begin{enumerate}[1.]
		\item We first justify the second quasi-isomorphism. The fact that the inclusion of $\Omega^\bullet_X(\log D)$ inside $j_*\A^\bullet_U$ is a quasi-isomorphism is classical (see \cite[Proposition 3.13]{deligneequadiff} or \cite[Proposition 8.18]{voisin}) and one can prove that the inclusion of $\Omega^\bullet_X(\log D)$ inside $\A^\bullet_X(\log D)$ is a quasi-isomorphism thanks to a Dolbeault complex argument as in the case $D=\varnothing$ (see the proof of \cite[Lemma 8.13]{voisin}). Complex conjugation is an automorphism of $j_*\A^\bullet_U$ and our claim follows.
		Regarding the first quasi-isomorphism, we note that the inclusion $j_!\A^\bullet_U \hookrightarrow \A^\bullet_X(\log D)(-D)$ is an isomorphism on restricting to $U$. Thus, we simply need to prove that the cohomology sheaf of the complex $\A^\bullet_X(\log D)(-D)$ vanishes at a point of $D$. (Note that this complex is therefore locally acyclic, as opposed to the complex  $\A^\bullet_X$, to which the usual Poincar\'e lemma applies,  which has cohomology in degree zero.)  
		This is a local statement and we reduce to the case where $X=\{(z_1,\ldots,z_n)\in\CC^n\; , \; |z_i|<1\}$ is a polydisk and $D=\{z_1\cdots z_s=0\}$ is a union of coordinate hyperplanes, with $s\geq 1$. Following the classical proof of the Poincar\'{e} lemma (see e.g. \cite[\S 4]{bottandtu}), we produce a contracting homotopy $h$ for the complex $\Gamma(X,\A^\bullet_X(\log D)(-D))$. An element $\nu$ of this complex can be uniquely written as a linear combination of forms
		$$\nu = f\,dz_I\wedge d\overline{z}_J$$
		with $I$ and $J$ subsets of $\{1,\ldots,n\}$ (we use the shorthand notations $dz_I=\bigwedge_{i\in I}dz_i$ and $d\overline{z}_J=\bigwedge_{j\in J}d\overline{z}_j$) and $f$ a smooth function such that $f/z_i$ is smooth on $X$ for every $i\in \{1,\ldots,s\}\min I$.  We set 
		$$h(\nu) = \left(\int_0^{z_1} f(u_1,z_2,\ldots,z_n)\, du_1\right) dz_{I\min \{1\}}d\overline{z}_J$$
		if $I$ contains $1$, and $h(\nu)=0$ otherwise. By using the change of variables $u_1=uz_1$ one rewrites the above integral as $z_1F$ with $F$ a smooth function on $X$, hence $h(\nu)$ is indeed an element of $\Gamma(X,\A^{\bullet-1}_X(\log D)(-D))$. One readily verifies, as in [\textit{loc. cit.}], that the commutator $d\circ h+h\circ d$ is the identity, and this finishes the proof of the second quasi-isomorphism. 
		\item Consider the following square:
		$$\xymatrix{
		\Gamma(X,j_!\A^\bullet_U) \ar[r]\ar[d] & \Gamma_\c(X,j_*\A^\bullet_U)^\vee[2n] \ar[d]\\
		\Gamma(X,\A^\bullet_X(\log D)(-D)) \ar[r] & \Gamma_\c(X,\overline{\A^\bullet_X(\log D)})^\vee [2n] 
		}$$
		where all maps respect the degree denoted by $\bullet$.
		The horizontal arrows are induced by
		$$ \nu \mapsto \int_X \nu\wedge (-) \ . $$ The vertical arrows are induced by the inclusions of complexes of 1. and therefore are quasi-isomorphisms because all the sheaves in question  are soft. The square is obviously commutative. The horizontal arrow along the top induces the pairing \eqref{eq:PDsimpler} and we only need to prove that the horizontal arrow along the bottom is a morphism of complexes.
		Equivalently, let $\nu$ be a section of $\A^\bullet_X(\log D)(-D)$ and let $\omega$ be a section  of $\A^\bullet_X(\log D)$ with compact support  such that the sum of their degrees is $2n-1$.  We need to prove that
		$$\int_X d(\nu\wedge \overline{\omega}) =  0 \ .$$
		Since   $\nu\wedge\overline{\omega}$ is not  in general  smooth on $X$, we cannot use the classical form of Stokes' theorem. But it is polar-smooth by Lemma \ref{lem:wedge product integrable simpler}  so we may work in the real-oriented blow-up $\widetilde{X}$. 
		By  Stokes' theorem for manifolds with corners \cite[Theorem 16.25]{leeintrosmoothmanifolds} we get
		$$\int_X d(\nu\wedge\overline{\omega}) = \sum_{i} \int_{\partial_i\widetilde{X}} (\nu\wedge\overline{\omega})_{|\partial_i\widetilde{X}}\ ,$$
		where the sum is over irreducible components $D_i$ of $D$, and $\partial_i\widetilde{X}$ denotes the sphere bundle of the normal bundle of $D_i$.  By a partition of unity argument  we are reduced to the local case where $X=\{(z_1,\ldots,z_n)\in\CC^n\; , \; |z_i|<1\}$ is a polydisk and $D=\{z_1\cdots z_s=0\}$ is a union of coordinate hyperplanes. Then 
		$$\widetilde{X} \simeq ([0,1)\times S^1)^s\times \{(z_{s+1},\ldots,z_n)\in \CC^{n-s}\; , \; |z_i|<1\}$$
		and the component $\partial_i\widetilde{X}$, for $i\in \{1,\ldots,s\}$, corresponds to the vanishing locus of the radial coordinate $\rho_i\in [0,1)$. By the local shape \eqref{eq:big wedge product} of a product $\nu\wedge\overline{\omega}$ we only need to consider  the two types of forms \eqref{eq:big wedge product polar}. Their restrictions to $\rho=0$ vanish in both cases.
		\end{enumerate}
		\end{proof}
		
		\begin{remark}\label{rem:without conjugation}
		The more natural version of Proposition \ref{prop:PD simpler log forms conjugate}, where we drop complex conjugation, also holds. The proof is even simpler because in this case the wedge product $\nu\wedge\omega$ is a smooth form on $X$, so $\int_X d(\nu\wedge\omega)$  vanishes by the classical version of Stokes' theorem. We will not use this in the rest of the article.
		\end{remark}

	\subsection{The case of two normal crossing divisors}\label{par:casetwoncd}
		We now fix a compact complex manifold $X$  with two divisors $A,B\subset X$ with no common irreducible component, such that $A\cup B$ is a simple normal crossing divisor. We  use  the complex of sheaves
		\begin{equation}\label{eq:log forms AB}
		\A^\bullet_X(\log A\cup B)(-B)\ .
		\end{equation}
		Its sections have logarithmic poles along $A$ and vanish along $B$. 
		We now show that it  computes the relative  cohomology of the pair $(X\min A, B\min A\cap B)$. We use the notation $H^\bullet_{\dR,\mathcal{C}^\infty}$ for the smooth version of de Rham cohomology, i.e., the cohomology of complexes of smooth differential forms. For instance, $H^\bullet_{\mathrm{dR},\mathcal{C}^\infty}(X\min A \lmod B)$ is defined to be the cohomology of the complex of sheaves $\Gamma(X\min A,(j_{X\min A\cup B}^{X\min A})_!\A^\bullet_{X\min A\cup B})$.
		
		\begin{proposition}\label{prop:5qis}
		We have a commutative diagram where every arrow is a quasi-isomorphism of complexes:
		$$\xymatrixcolsep{0.05pc}\xymatrix{
		(j_{X\min B}^X)_!(j_{X\min A\cup B}^{X\min B})_*\mathcal{A}^\bullet_{X\min A\cup B} \ar[rr]&  & (j_{X\min A}^X)_*(j_{X\min A\cup B}^{X\min A})_!\mathcal{A}^\bullet_{X\min A\cup B} \ar[d] \\
		(j_{X\min B}^X)_!\mathcal{A}^\bullet_{X\min B}(\log A\min B) \ar[u]\ar[dr]& & (j_{X\min A}^X)_*\mathcal{A}^\bullet_{X\min A}(\log B\min A)(-B\min A)  \\
		& \mathcal{A}^\bullet_X(\log A\cup B)(-B) \ar[ur] &
		}$$
		where $A\min B$ denotes $A \min (A \cap B)$ and $B\min A$ denotes $B\min (A\cap B)$.
		Thus, we have a canonical isomorphism
		$$H^r_{\dR,\mathcal{C}^\infty}(X\min A \lmod B) \simeq H^r(\Gamma(X,\A^\bullet_X(\log A\cup B)(-B))\ .$$
		\end{proposition}
		
		\begin{proof}
		Every arrow  is an  inclusion morphism  and  so the diagram commutes. The horizontal arrow along the top is a quasi-isomorphism by Proposition \ref{prop: commutejj} and the Poincar\'{e} lemma $\CC_{X\min A\cup B}\simeq \A^\bullet_{X\min A\cup B}$.	The first part of Proposition \ref{prop:PD simpler log forms conjugate} applied to $(X\min B,A\min B)$ and $(X\min A,B\min A)$ implies that the vertical arrows on both sides are quasi-isomorphisms (for the vertical arrow on the left-hand side we use the version without complex conjugation which is explained in the proof of Proposition \ref{prop:PD simpler log forms conjugate}). 
		Now we claim that the inclusion
		$$(j_{X\min B}^X)_!\A^\bullet_{X\min B}(\log A\min B) \hookrightarrow \A^\bullet_X(\log A\cup B)(-B)$$
		is a quasi-isomorphism. This inclusion is an isomorphism on restriction to $X\min B$, so it is enough to prove that the cohomology sheaf of the complex $\A^\bullet_X(\log A\cup B)(-B)$ vanishes at  a point of $B$. We leave it the reader to mimick the proof of the first part of Proposition \ref{prop:PD simpler log forms conjugate} in order to produce a contracting homotopy for the local sections of this complex of sheaves around a point of $B$ (integrate with respect to a local equation for an irreducible component of $B$).
		 Since the diagram commutes, the remaining arrow is also a quasi-isomorphism. The last statement follows from the fact that all complexes appearing in the diagram consist of soft sheaves.
		\end{proof}
		
		The next lemma and proposition generalise Lemma \ref{lem:wedge product integrable simpler} and Proposition \ref{prop:PD simpler log forms conjugate}.
		
		\begin{lemma}\label{lem:wedge product integrable}
		For $\nu$ a section of $\A^\bullet_X(\log A\cup B)(-A)$ and $\omega$ a section of $\A^\bullet_X(\log A\cup B)(-B)$, the wedge products $\nu\wedge\overline{\omega}$ and $\overline{\nu}\wedge\omega$ are polar-smooth forms on $(X,A\cup B)$.
		\end{lemma}
		
		\begin{proof}
		The proof is similar to that of Lemma \ref{lem:wedge product integrable simpler} and is left to the reader.
		\end{proof}		
		
		\begin{proposition}\label{prop:PD log forms conjugate}
		Under the identifications
		$$H^r_{\dR,\mathcal{C}^\infty}(X\min B \lmod A)\simeq H^r(\Gamma(X,\A^\bullet_X(\log A\cup B)(-A)))$$
		and
		$$H^{2n-r}_{\dR,\mathcal{C}^\infty}(X\min A  \lmod B)\simeq H^{2n-r}(\Gamma(X,\overline{\A^\bullet_X(\log A\cup B)(-B)}))$$
		induced by Proposition \ref{prop:5qis}, the pairing \eqref{eq:PDNCDpairing} is induced in (smooth) de Rham cohomology by the pairing of complexes: 
		\begin{eqnarray} \Gamma(X,\A^\bullet_X(\log A\cup B)(-A))\otimes \Gamma(X,\overline{\A^\bullet_X(\log A\cup B)(-B)})  & \To & \CC[2n]  \nonumber \\  \nonumber 
		\nu\otimes \overline{\omega} \qquad \qquad & \mapsto & \int_X \nu\wedge\overline{\omega}\ .
		\end{eqnarray} 
		\end{proposition}    
		
		\begin{proof}
		We form the following commutative diagram:
		$$\xymatrix{
		\Gamma(X,(j_{X\min B}^X)_*(j_{X\min A\cup B}^{X\min B})_!\A^\bullet_{X\min A\cup B}) \ar[r] & \Gamma(X,(j_{X\min B}^X)_!(j_{X\min A\cup B}^{X\min B})_*\A^\bullet_{X\min A\cup B})^\vee[2n] \\
		\Gamma(X\min B,(j_{X\min A\cup B}^{X\min B})_!\A^\bullet_{X\min A\cup B}) \ar[u]^{\simeq}\ar[r]& \Gamma_\c(X\min B,(j_{X\min A\cup B}^{X\min B})_*\A^\bullet_{X\min A\cup B})^\vee[2n] \ar[u]_{\simeq}\\
		\Gamma(X\min B,\A^\bullet_{X\min B}(\log A\min B)(-A\min B)) \ar[r]  \ar[u]\ar[d]& \Gamma_\c(X\min B,\overline{\A^\bullet_{X\min B}(\log A\min B)})^\vee[2n] \ar[u]\ar[d] \\
		\Gamma(X,\A^\bullet_X(\log A\cup B)(-A)) \ar[r] & \Gamma(X,\overline{\A^\bullet_X(\log A\cup B)(-B)})^\vee[2n]
		}$$
		The vertical arrows of the top square are isomorphisms. The vertical arrows on the right (top and bottom) use the fact that $\Gamma(X, -)=\Gamma_\c(X, -)$ since $X$ is compact.  The bottom two horizontal arrows are induced by 
		 $\nu \mapsto \int_{X\min B}\nu\wedge (-)$ and $\nu\mapsto \int_X\nu\wedge (-)$ respectively. 
		 This makes sense for the bottom arrow by  Lemma \ref{lem:wedge product integrable}. 
		The commutativity of the  middle square is the content of Proposition \ref{prop:PD simpler log forms conjugate} for $(X\min B, A\min B)$ since, in the statement of Proposition 3.10,  $X$ is not necessarily compact. The vertical arrows in the bottom square are quasi-isomorphisms induced by inclusion maps as in Proposition \ref{prop:5qis}. It remains to prove that the bottom horizontal arrow  is a morphism of complexes. This is done in the same way as in the proof of Proposition \ref{prop:PD simpler log forms conjugate} by working in the real-oriented blow-up $\widetilde{X}$. The proposition follows since the isomorphisms in the statement are induced by composing the vertical quasi-isomorphisms.
		\end{proof}
		
		\begin{remark}
		As in Remark \ref{rem:without conjugation}, the more natural version of Proposition \ref{prop:PD log forms conjugate}, where we drop complex conjugation, also holds, but we will not use it.
		\end{remark}
		
	\subsection{Computing the single-valued period map with differential forms}
	We recast  the previous discussion in terms of the single-valued period map. Let us fix $k=\QQ$ for simplicity (see Remark \ref{rem:sv recipe number field} below for the case of a general number field) and let $(X,A,B)$ satisfy $(\star)_\QQ$. For any integer $r$, consider the object of $\mathcal{H}:$
		$$H=H^{2n-r}(X\min A \lmod B )$$ 
	 defined in  Example \ref{ex: cohomXAB}. By Theorem \ref{thm: PDNCD} we have an isomorphism in $\mathcal{H}$:
		$$H^\vee\simeq H^{r}(X\min B \lmod A)(n)\ .$$
		The fact that this isomorphism preserves the rational structures in de Rham cohomology and is compatible with the action of the real Frobenius in the Betti realisation is not in fact proved in  Theorem \ref{thm: PDNCD}. It is true, but will not be needed in this paper.  \medskip
		
			Let $\omega$ be a closed differential form in $\Gamma(X,\A^{2n-r}_X(\log A\cup B)(-B))$  and denote by $[\omega]$ the corresponding class in $H_{\mathrm{dR}}\otimes_\QQ\CC\simeq H^{2n-r}_{\mathrm{dR},\mathcal{C}^\infty}(X\min A \;\mathrm{mod}\; B)$. Let $\nu$ be a closed differential form in $\Gamma(X,\A^{r}_X(\log A\cup B)(-A))$ and denote by $[\nu(n)]$ the  class in $H_{\mathrm{dR}}^\vee\otimes_\QQ\CC$ which corresponds to the class of $(2\pi i)^{-n}\nu$ in $H^{r}_{\mathrm{dR},\mathcal{C}^\infty}(X\min B \;\mathrm{mod}\; A)$. We then have a matrix coefficient
		$$[H,[\nu(n)],[\omega]]^{\mm,\dR} \in \mathcal{P}^{\mm,\dR}_{\mathcal{H}}\otimes_\QQ\CC\ .$$
		By the results proved earlier, every matrix coefficient can be represented using differential  forms $\omega, \nu$ of this type, so the following theorem covers the general case.
		
		\begin{theorem}\label{thm:sv recipe general}
		 The single-valued map is computed by the following formula:
		$$\s\,[H,[\nu(n)],[\omega]]^{\mm,\dR} = (2\pi i)^{-n}\int_{X(\CC)} \nu\wedge \mathrm{conj}^*(\omega) = (-2\pi i)^{-n}\int_{X(\CC)} \mathrm{conj}^*(\nu)\wedge\omega\ ,$$
		where $\mathrm{conj}:X(\CC)\rightarrow X(\CC)$ denotes complex conjugation.
		\end{theorem}
		
		\begin{proof}
		By definition we have 
		$$\s\,[H,[\nu(n)],[\omega]]^{\mm,\dR} = \langle [\nu(n)],\s[\omega] \rangle$$
		where in the right-hand side $\s$ denotes the composite
		$$\s : H_{\mathrm{dR},\mathcal{C}^\infty} \simeq H_{\mathrm{dR}} \otimes \CC \overset{c}{\To}  H_{\mathrm{B}} \otimes \CC \overset{F_{\infty}\otimes id}{\To} H_{\mathrm{B}} \otimes \CC \overset{c^{-1}}{\To} H_{\mathrm{dR}} \otimes \CC\simeq H_{\mathrm{dR},\mathcal{C}^\infty}\ .$$ 
		Recall that $F_\infty$ is induced in singular cohomology by $\mathrm{conj}$. By the functoriality of the comparison isomorphism between (smooth) de Rham and singular cohomology (the de Rham theorem), the above composite is  induced in (smooth) de Rham cohomology by $\mathrm{conj}$, i.e. $\s[\omega]=[\mathrm{conj}^*(\omega)]$. We note that $\mathrm{conj}^*(\omega)$ is a global section of $\overline{\A^{2n-r}_X(\log A\cup B)(-B)}$ because $\mathrm{conj}$ is anti-holomorphic. The first equality is then a restatement of Proposition \ref{prop:PD log forms conjugate}. The second equality follows on applying the change of variables  $\mathrm{conj}$ and noting that since it is anti-holomorphic, it multiplies the orientation of $X(\CC)$ by the sign $(-1)^n$.
		\end{proof}
		
	We are mostly interested  in the case where the differential forms $\omega$ and $\nu$ are algebraic. We state the following important special case as a separate theorem for future reference.  For $X$ a smooth projective variety of dimension $n$ over $\QQ$ and $D$ a normal crossing divisor in $X$, let  $\Omega^n_{X/\QQ}(\log D)$ denote the (Zariski) sheaf of algebraic $n$-forms with logarithmic singularities along $D$. Every such form is necessarily closed for degree reasons and we thus get a $\QQ$-linear map $\Gamma(X,\Omega^n_{X/\QQ}(\log D))\rightarrow H^n_\dR(X\min D)$ (see  Proposition \ref{prop:global log Hodge} for a Hodge-theoretic interpretation). If $(X,A,B)$ satisfies $(\star)_\QQ$ then every global section of $\Omega^n_{X/\QQ}(\log A)$ vanishes along $B$ for degree reasons and we get a $\QQ$-linear map $\Gamma(X,\Omega^n_{X/\QQ}(\log A))\rightarrow H^n_\dR(X\min A\;\mathrm{mod}\; B)$.

		\begin{theorem}\label{thm:sv recipe}
		 Let $(X,A,B)$ satisfy $(\star)_\QQ$ and set $H=H^n(X\min A \;\mathrm{mod}\; B)\in\mathcal{H}$. For $\omega\in \Gamma(X,\Omega^n_{X/\QQ}(\log A))$ and $\nu\in \Gamma(X,\Omega^n_{X/\QQ}(\log B))$ we have a matrix coefficient
		 $$[H,[\nu(n)],[\omega]]^{\mm,\dR} \in \mathcal{P}^{\mm,\dR}_{\mathcal{H}}\otimes_\QQ\CC$$
		 whose single-valued period is computed by the formula:
		$$\s\,[H,[\nu(n)],[\omega]]^{\mm,\dR} = (2\pi i)^{-n}\int_{X(\CC)} \nu\wedge\overline{\omega} = (-2\pi i)^{-n}\int_{X(\CC)} \overline{\nu}\wedge\omega\ .$$
		\end{theorem}
		
		\begin{proof}
		Clearly $\omega$ and $\nu$ are global sections of $\A^n_X(\log A\cup B)(-B)$ and $\A^n_X(\log A\cup B)(-A)$ respectively. Since they are algebraic defined over $\QQ$ we have $\mathrm{conj}^*(\omega)=\overline{\omega}$ and $\mathrm{conj}^*(\nu)=\overline{\nu}$, and the result follows from Theorem \ref{thm:sv recipe general}.
		\end{proof} 
		
		\begin{remark}
		The matrix coefficient $[H,[\nu(n)],[\omega]]^{\mm,\dR}$ actually lies in the rational vector space $\Pe^{\mm,\dR}_{\mathcal{H}}$, i.e.,  the class $[\nu(n)]$ lies in the rational vector space $H^n_\dR(X\min A\,\lmod B)^\vee$. This  follows from the fact (which we have claimed and not proved) that the duality isomorphism of Theorem \ref{thm: PDNCD} preserves the rational structures in de Rham cohomology. We will mostly be interested in situations when the class $[\nu(n)]$ arises from the de Rham projection (see \S\ref{sect: dRProj}), in which case we  show directly that it lies in $H^n_\dR(X\min A\,\lmod B)^\vee\otimes_{\QQ} K$ for a certain extension $K$ of $\QQ$  (see Lemma \ref{lem: coefficients c zero} and Remark \ref{rem: coefficients nu}).
		\end{remark}
		
		\begin{remark}\label{rem:sv recipe number field}
		Let $k$ be a number field and $\sigma:k\hookrightarrow \CC$ be a complex embedding. If $\sigma=\overline{\sigma}$ is real then 
		Theorems \ref{thm:sv recipe general} and \ref{thm:sv recipe} generalise easily with the same formulae  to the setting of the category $\mathcal{H}(k)$. In the general  case the Tannakian interpretation of $\s_{\sigma}$ is obscure (\S\ref{sect: TannakianNumberFields}) but  the formulae still make sense, when correctly interpreted. Let   $\mathrm{conj}_\sigma: X_\sigma(\CC) \rightarrow X_{\overline{\sigma}}(\CC)$ denote complex conjugation. Theorem \ref{thm:sv recipe general} should be replaced by the statement:
		$$\langle[\nu(n)],\s_\sigma [\omega]\rangle  = (2\pi i)^{-n}\int_{X_{\overline{\sigma}}(\CC)}\nu\wedge\mathrm{conj}_{\overline{\sigma}}^*(\omega) = (-2\pi i)^{-n} \int_{X_\sigma(\CC)} \mathrm{conj}_\sigma^*(\nu)\wedge\omega\ ,$$
		for $\omega$ (resp. $\nu$) a global section of $\A^{2n-r}_{X_\sigma(\CC)}(\log A_\sigma(\CC)\cup B_\sigma(\CC))(-B_\sigma(\CC))$ (resp. a global section of $\A^{r}_{X_{\overline{\sigma}}(\CC)}(\log A_{\overline{\sigma}}(\CC)\cup B_{\overline{\sigma}}(\CC))(-A_{\overline{\sigma}}(\CC))$). In the algebraic case, Theorem \ref{thm:sv recipe} should be replaced by the statement:
		$$\langle [\nu(n)],\s_\sigma [\omega] \rangle = (2\pi i)^{-n}\int_{X_{\overline{\sigma}}(\CC)}\nu_{\overline{\sigma}}\wedge\overline{\omega_{\overline{\sigma}}}  = (-2\pi i)^{-n} \int_{X_\sigma(\CC)} \overline{\nu_{\sigma}}\wedge\omega_\sigma\ ,$$
		for $\omega$ (resp. $\nu$) a global algebraic differential $n$-form on $X/k$ with logarithmic singularities along $A$  (resp. $B$). Here we have denoted by $\omega_\sigma$ the form on $X_\sigma(\CC)$ induced by $\omega$, and so on. Note that this last formula is consistent with the fact that $\overline{\s_\sigma}=\s_{\overline{\sigma}}$.
		\end{remark}
		
		\subsection{Double copy formula} 
		A byproduct of our integral formula for single-valued periods is a general and elementary `double copy' formula expressing singular volume integrals as quadratic expressions in ordinary period integrals.
		
		    \begin{corollary} \label{corDC} (Double copy formula). In the setting of Theorem \ref{thm:sv recipe} we have the equality:
            \begin{equation}  \label{DC}  \int_{X(\CC)}  \nu\wedge \overline{\omega}  =  \sum_{[\gamma],[\delta]} \langle  [\gamma]^{\vee}, [\delta]^{\vee} \rangle  \int_{\gamma} \nu \int_{\overline{\delta}}\omega 
            \end{equation} 
            where $[\gamma]$ ranges over a basis for $H_n^{\B}(X \backslash B \,\lmod A)$ and $[\gamma]^{\vee}$ denotes the dual basis, and similarly for $[\delta]$ with $A$, $B$ interchanged. The pairing on the right-hand side is the Verdier duality pairing \eqref{eq:PDNCDpairing} in Betti cohomology.
            \end{corollary}
          
            \begin{proof}
            According to Theorem \ref{thm:sv recipe} we have the equality
            $$\int_{X(\CC)}\nu\wedge\overline{\omega} = (2\pi i)^n\langle[\nu],\s[\omega]\rangle = (2\pi i)^n \langle[\nu],c^{-1}F_\infty c[\omega]\rangle$$
            where the pairing is the Verdier duality pairing \eqref{eq:PDNCDpairing} in de Rham cohomology, and the second equality is the definition of the single-valued pairing. Now the compatibility between Verdier duality and the comparison isomorphism implies that we have
            $$(2\pi i)^n\langle[\nu],c^{-1}F_\infty c[\omega]\rangle = \langle c[\nu],F_\infty c[\omega]\rangle\ ,$$
            where the pairing in the right-hand side is the Verdier duality pairing \eqref{eq:PDNCDpairing} in Betti cohomology. Now by definition of $c$ and $F_\infty$ we have
            $$c[\nu] = \sum_{[\gamma]}[\gamma]^\vee\int_\gamma\nu \quad \mbox{ and }\quad F_\infty c[\omega] = \sum_{[\delta]}[\delta]^\vee\int_{\overline{\delta}}\omega\ ,$$
            and the claim follows.
            \end{proof}
            
            The classes $[\gamma]$ can be represented by  chains on $X(\CC) \backslash B(\CC)$ with boundary contained in $A(\CC)$, and likewise $[\delta]$ with $A$, $B$ interchanged. The rational coefficients $\langle[\gamma]^\vee,[\delta]^\vee\rangle$ appearing in the formula are the entries of the inverse transpose matrix of the intersection matrix of representatives for the classes $[\gamma]$ and $[\delta]$.
		
		\subsection{Remarks on higher order poles}
		The proof of Proposition \ref{prop:PD simpler log forms conjugate} easily generalises. One can show in the same manner that 
		$$j_!\A^\bullet_U \stackrel{\sim}{\hookrightarrow} \A^\bullet_X(\log D)(-(p+1) D) \;\; \textnormal{ and } \;\; \overline{\A^\bullet_X(\log D)(q D)} \stackrel{\sim}{\hookrightarrow} j_*\A^\bullet_U $$
		are quasi-isomorphisms for $p, q\geq 0$ and that 
		$$\Gamma(X,\A^\bullet_X(\log D)((-p+1) D))\otimes \Gamma_\c(X,\overline{\A^\bullet_X(\log D)(q D)})  \rightarrow \CC[2n]$$
		$$\nu\otimes\overline{\omega}\mapsto \int_X\nu\wedge\overline{\omega}\ $$
		is a well-defined perfect pairing 
		whenever $p \geq q.$ The relevant local forms are obtained by multiplying  (\ref{eq:big wedge product polar}) with $z^{p}/\overline{z}^{q}$. Consequently, 
		
			\begin{proposition}\label{prop: higher order poles}
		Under the identifications
		$$H^r_{\mathrm{dR},\mathcal{C}^\infty}(X\min B \lmod A)\simeq H^r(\Gamma(X,\A^\bullet_X(\log A\cup B)(bB-(a+1)A)))$$
		and
		$$H^r_{\mathrm{dR},\mathcal{C}^\infty}(X\min A  \lmod B)\simeq H^r(\Gamma(X,\overline{\A^\bullet_X(\log A\cup B)(a A-(b+1)B)}))$$
		the pairing \eqref{eq:PDNCDpairing} is induced in (smooth) de Rham cohomology for any integers $a,b\geq 0$ by the integration pairing $\nu\otimes \overline{\omega} \mapsto  \int_X \nu\wedge\overline{\omega}$ on global sections. 
		\end{proposition}    
		
		\begin{remark}\label{rem: higher order poles}
		These formulae rely on the fact that $\nu\wedge\overline{\omega}$ is polar-smooth. 
		It would be interesting to give formulae for the pairing \eqref{eq:PDNCDpairing} involving more general differential forms, e.g., $\omega$, $\nu$ with high order poles  along $A$, $B$ such that  $\nu\wedge\overline{\omega}$ is not polar-smooth.  The main problem is that in general the integral $\int_X\nu\wedge\overline{\omega}$ is not absolutely convergent.  In the setting of Theorem \ref{thm:sv recipe general} one expects  the formula
		\begin{equation}\label{eq:sv lim epsilon general}
		\lim_{\varepsilon\rightarrow 0}\,(2\pi i)^{-n}\int_{X_\varepsilon}\nu\wedge\overline{\omega}
		\end{equation}
	to compute the single-valued period map for a good family of open subsets $X_\varepsilon$ approximating $X(\CC)$. Typically, $X_\varepsilon$ should look like the complement in $X(\CC)$ of a union of tubular $\varepsilon$-neighbourhoods of the irreducible components of $A$ and $B$. The main issue is that \eqref{eq:sv lim epsilon general} is sensitive to the shape of those tubular neighbourhoods.
	In the case of curves, this idea can be made to work using the neighbourhoods $|w|=\varepsilon$ around each singular point, where $w$ is a local holomorphic coordinate, using  the fact that the angular integral $\int_{0}^{2 \pi} e^{i n \theta} d\theta $ vanishes whenever $n\neq 0$.
	More generally, if $A\cup B$ is a disjoint union of smooth divisors, the notion of a cut-off function introduced by Felder and Kazhdan in \cite{felderkazhdanRiemann} should allow one to define well-behaved tubular neighbourhoods.
    \end{remark}

\section{The de Rham projection for separated motives} \label{sect: dRProj}

     In this section we shall state our results in the case $k=\QQ$ in order to keep the notations simple. There is no extra difficulty in the general case.

    \subsection{Separated mixed Hodge structures}  \label{sect: Separated} See also \cite[4.3]{brownnotesmot}.

    	\begin{definition}
    	An object $H$ of $\mathcal{H}$ is \emph{effective} if $W_{-1}H=0$ and $H=F^0H$.
		\end{definition}
    		
    	The Tate object $\QQ(-n)$ is effective if and only if $n\geq 0$. More generally, all objects $H^r(X\min A,B\min A\cap B)$ constructed in Example \ref{ex: cohomXAB} are effective. Effectivity means that the Hodge numbers $h^{p,q}=\dim (F^p\overline{F^q}\mathrm{gr}^W_{p+q}H)$ vanish unless $p,q\geq 0$.

        \begin{definition}
        An object $H$ of $\mathcal{H}$ is \emph{separated} if it is effective and the composite 
        $$\mathrm{gr}_0^W H = W_0H_{\mathrm{dR}} \rightarrow H_{\mathrm{dR}} \rightarrow H_{\mathrm{dR}}/F^1H_{\mathrm{dR}} = \mathrm{gr}^0_F H\ ,$$
        which  is injective, is in fact an isomorphism. 
        \end{definition}

        Separatedness means that there is a direct sum decomposition
        $$H_{\mathrm{dR}} = W_0 H_{\mathrm{dR}} \oplus F^1 H_{\mathrm{dR}}\ .$$
    		Equivalently,  $h^{p,q}=0$ unless $(p,q)=(0,0)$ or $p,q>0$. Note that effective mixed Tate objects (for which $h^{p,q}=0$ for $p\neq q$) are automatically separated. For a smooth projective curve $X$ and $A,B\subset X$ disjoint finite sets of points, the object $H^1(X\min A \;\mathrm{mod}\; B)$ is separated if and only if $X$ has genus zero.\medskip
        
        For a separated object $H$ one can form the composite
        $$c_0:H_{\mathrm{dR}}\otimes_\QQ\CC \twoheadrightarrow (H_{\mathrm{dR}}/F^1H_{\mathrm{dR}})\otimes_\QQ\CC \simeq W_0H_{\mathrm{dR}}\otimes_\QQ\CC \stackrel{\simeq}{\rightarrow} W_0H_{\mathrm{B}}\otimes_\QQ\CC \hookrightarrow H_{\mathrm{B}}\otimes_\QQ\CC\ .$$
        where the map $W_0H_{\mathrm{dR}}\otimes_\QQ\CC \rightarrow W_0H_{\mathrm{B}}\otimes_\QQ\CC$ is the comparison isomorphism.
        We will mostly be concerned with its transpose
        \begin{equation}\label{eq:c0vee definition}
        c_0^\vee: H_\B^\vee\otimes_\QQ\CC \rightarrow H_\dR^\vee\otimes_\QQ\CC\ .
        \end{equation}
        We will see below (Lemma \ref{lem: coefficients c zero} and Remark \ref{rem: coefficients c zero}) that in the cases of interest, when $H_{B/\dR}$ come from the cohomology of algebraic varieties over $\overline{\QQ}$, the morphisms $c_0$ and $c_0^\vee$ are defined over a number field $K\subset \mathbb{C}$, i.e., $c_0^\vee:H_\B^\vee\otimes_\QQ K \longrightarrow H_\dR^\vee\otimes_\QQ K$. 
      In the case of effective mixed Tate motives,  which are  separated, there exists a  similar morphism using the canonical fiber functor which  is  always defined over $\QQ$ (Remark \ref{rem: s vs sv MT}). 
        \medskip
       	 
       	 Let $\mathcal{H}^+_{\mathrm{sep}}\subset \mathcal{H}$  denote the full subcategory of separated objects. Using the fact that the functors $\mathrm{gr}_F$, $\mathrm{gr}^W$ are exact, one sees that this category is an abelian tensor category closed under subquotients. It is, however, not closed under duals and therefore not Tannakian. Nonetheless,  the map $c_0$ defines a natural transformation between the restrictions of 
       	 $\omega_{\mathrm{dR}}\otimes_\QQ\CC$ and $\omega_{\mathrm{B}}\otimes_\QQ\CC$ to $\mathcal{H}^+_{\mathrm{sep}}$. Let us denote by $\Pe^{\mm,+}_{\mathcal{H},\mathrm{sep}}$ (resp. $\Pe^{\mm,\dR,+}_{\mathcal{H},\mathrm{sep}}$) the subalgebra of $\Pe^{\mm}_{\mathcal{H}}$ (resp. $\Pe^{\mm,\dR}_{\mathcal{H}}$) spanned by matrix coefficients of separated  objects.

        \begin{definition}\label{defi:dR projection}
        The morphism of algebras
        $$\pi^{\mm,\dR}:\Pe^{\mm,+}_{\mathcal{H},\mathrm{sep}}\otimes_\QQ\CC \longrightarrow \Pe^{\mm,\dR,+}_{\mathcal{H},\mathrm{sep}}\otimes_\QQ\CC \hookrightarrow \Pe^{\mm,\dR}_{\mathcal{H}}\otimes_\QQ\CC$$
        defined on matrix coefficients by
         \begin{equation}\label{eq:definition dR projection}
       	 \pi^{\mm,\dR}\,[H,\gamma,\omega]^\mm = [H,c_0^\vee(\gamma),\omega]^{\mm,\dR} \ .
       	 \end{equation}
       is called the \emph{de Rham projection} for separated objects in $\mathcal{H}$.
        \end{definition}
        
       The Lefschetz $\mathcal{H}$-period $\mathbb{L}^\mm$ from Example \ref{example cauchy} is in the kernel of the de Rham projection. The map $\pi^{\mm,\dR}$ is not strictly speaking a projection, but gets its name from the fact that it arises by projecting modulo $F^1$.
       
       \begin{remark}\label{rem: s vs sv MT}In the setting of mixed Tate motives (\S\ref{par:MT case}) all effective objects are automatically separated. Furthermore, for an effective mixed Tate motive $M$ in $\MT(k)$ the comparison $W_0M_{\can}\otimes_\QQ\CC\stackrel{\sim}{\rightarrow} W_0M_{\B,\sigma}\otimes_\QQ\CC$ is defined over $\QQ$ because $W_0M$ is a direct sum of copies of $\QQ(0)$. One can repeat the construction above to define a projection morphism
       $$\pi^{\mm,\can}:\Pe^{\mm,(\can,(\B,\sigma)),+}_{\MT(k)}\longrightarrow \Pe^{\mm,\can}_{\MT(k)}\ ,$$
       where $\Pe^{\mm,(\can,(\B,\sigma)),+}_{\MT(k)}$ is defined analogously to $\Pe^{\mm,+}_{\mathcal{H},\mathrm{sep}}$.
       In the setting of the category $\MT(\mathbb{Z})\subset\MT(\QQ)$ this map was originally defined in \cite{brownSVMZV}. We note that the single-valued period map $\s_\sigma$ and its variant $\sv_\sigma$ defined in \S\ref{par:MT case} are the same on the image of $\pi^{\mm,\can}$ because the image of $c_0^\vee$ is in weight zero.
       \end{remark}
       
        In order to compute the de Rham projection in practice we need to  interpret $c_0^\vee(\gamma)$ in equation \eqref{eq:definition dR projection}. This is explained below. The following well-known  observation allows us to focus on cohomology groups of the type $H^r(X,D)$ or $H^r(X\min D)$.
        
        \begin{proposition}\label{prop:grW 0 and 2n}
        Let $(X,A,B)$ satisfy $(\star)_\QQ$. For every integer $r$:
       	\begin{align*}
       	\mathrm{gr}\, H^r(X\min A \lmod B) \simeq \mathrm{gr} \, H^r(X,B) \qquad \hbox{for }  \quad  \mathrm{gr} \  = \ \mathrm{gr}_0^W \ , \  \mathrm{gr}^0_F \  ; \\
       	\mathrm{gr}\,  H^r(X\min A \lmod B) \simeq \mathrm{gr} \,  H^r(X\min A)
       	\qquad \hbox{for }  \quad  \mathrm{gr}  \ =  \ \mathrm{gr}_{2n}^W \ ,  \ \mathrm{gr}^n_F \ . 
       	\end{align*}

        \end{proposition}
        \begin{proof}
        The first statement follows from the last and Theorem \ref{thm: PDNCD}. To prove the last statement consider the long exact sequence in cohomology:
        $$\cdots \rightarrow H^{r-1}(B\min A\cap B) \rightarrow H^r(X\min A \lmod B) \rightarrow H^r(X\min A) \rightarrow H^r(B\min A\cap B) \rightarrow \cdots$$
        It is compatible with mixed Hodge structures. Since the functors $\mathrm{gr}_{2n}^W$ and $\mathrm{gr}^n_F$ are exact, we need to prove that $\mathrm{gr}_{2n}^WH^r(B\min A\cap B)$ and $\mathrm{gr}^n_FH^r(B\min A\cap B)$ vanish for every $r$. Since  $B\min A\cap B$ has dimension $n-1$, the weight  (resp. Hodge) filtration on its cohomology only goes  up to degree $2(n-1)$ (resp. $n-1$).
        \end{proof}
   
   		\begin{proposition}\label{prop:relative separated}
   		Let $(X,A,B)$ satisfy $(\star)_\QQ$. For every integer $r$, the object $H^r(X\min A \;\mathrm{mod}\; B)$ is separated if and only if $H^r(X,B)$ is.
   		\end{proposition}
   		
   		\begin{proof}
   		This is a consequence of the first
   		isomorphism 
   		of Proposition \ref{prop:grW 0 and 2n}.
   		\end{proof}

    \subsection{The weight \texorpdfstring{$0$}{0} part of \texorpdfstring{$H^n(X,D)$}{Hn(X,D)}}
    
        Let $X$ be a smooth projective complex variety of dimension $n$ and let $D$ be a \emph{simple} normal crossing divisor, i.e. whose irreducible components are smooth (see Remark \ref{rem:simple ncd} below for the general case). We fix a linear order on the set of (smooth) irreducible components of $D=D_1\cup\cdots\cup D_r$. For a  subset $I\subset\{1,\ldots,r\}$, write
    	 $D_I=\bigcap_{i\in I}D_i$,  with the convention that $D_\varnothing=X$. It is a 
     smooth closed subvariety of $X$  of codimension $|I|$. For later use, we recall a  well-known spectral sequence which computes
    	 $H^\bullet(X,D)$. 
       		
    		\begin{proposition}\label{prop:ss relative}
    		\begin{enumerate}[1.]
    		\item There is a spectral sequence
    		\begin{equation}\label{eq:ss relative}
    		E_1^{p,q} = \bigoplus_{|I|=p} H^q(D_I) \;\; \Rightarrow \;\; H^{p+q}(X,D)
    		\end{equation}
    		in the category of mixed Hodge structures. It degenerates at $E_2$ and the differential $d_1$ is the alternating sum of the natural morphisms
    		$$H^q(D_J)\rightarrow H^q(D_I) \qquad 
    		\hbox {for  } J\subset I, |I\min J|=1 \ . $$
    		\item There is an exact sequence
    		\begin{equation}\label{eq:es relative}
    		0\rightarrow \mathrm{gr}_0^WH_n(X,D) \rightarrow \bigoplus_{|I|=n}H_0(D_I) \longrightarrow \bigoplus_{|J|=n-1} H_0(D_J) \ .
    		\end{equation}
    		\end{enumerate}
    		\end{proposition}
    		
    		\begin{proof}
    		\begin{enumerate}[1.]
    		\item We let $j:X\min D\hookrightarrow X$ denote the open immersion, and for every subset $I\subset\{1,\ldots,r\}$, let $i_I:D_I\hookrightarrow X$ denote the corresponding closed immersion. For $\mathcal{F}$ a sheaf of $\QQ$-vector spaces on $X$, we set
    		$$C^r(\mathcal{F}) = \bigoplus_{|I|=r}(i_I)_*(i_I)^*\mathcal{F}\ .$$
    		Let $d:C^r(\mathcal{F})\rightarrow C^{r-1}(\mathcal{F})$ denote the alternating sum of the natural morphisms
    		$$(i_J)_*(i_J)^*\mathcal{F}\rightarrow (i_I)_*(i_I)^*\mathcal{F}\ ,$$
    		for $J\subset I$. This forms a complex of sheaves $C^\bullet(\mathcal{F})$ on $X$. The natural morphism $j_!j^*\mathcal{F}\rightarrow \mathcal{F}=C^0(\mathcal{F})$ induces a map of complexes
    		$$j_!j^*\mathcal{F}\rightarrow C^\bullet(\mathcal{F})$$
    		which is easily seen to be a quasi-isomorphism, i.e., a resolution of $j_!j^*\mathcal{F}$. In particular, the complex $C^\bullet(\QQ_X)$ computes the cohomology $H^\bullet(X,D)$. One then filters it by the stupid filtration and applies the spectral sequence in hyper-cohomology. The compatibility with mixed Hodge structures uses Saito's category of mixed Hodge modules. Care has to be taken regarding $t$-structures, as is well explained  in \cite[\S 3.2]{petersen}. Now since each $E_1^{p,q}$ has a pure Hodge structure of weight $q$, all differentials $d_r$, for $r\geq 2$, vanish, which proves the first claim. 
    	
    		\item In particular, we have an isomorphism $\mathrm{gr}_0^WH^{n}(X,D)\simeq E_2^{n,0}$. By the description of the $E_1$ page, this is the cokernel of the map 
    		$$ \bigoplus_{|J|=n-1}H^0(D_J) \longrightarrow \bigoplus_{|I|=n} H^0(D_I)\ .$$ 
    		Taking the dual, via $\mathrm{gr}_0^WH_n(X,D)\simeq \left(\mathrm{gr}_0^WH^n(X,D)\right)^\vee$, implies the claim.
    		\end{enumerate}
    		\end{proof}
    		
    		If $X$ is a smooth projective complex curve and $D\subset X$ a finite set of points then \eqref{eq:es relative} follows by applying $\mathrm{gr}^W_0$ to the long exact sequence in relative homology
    		$$0\rightarrow H_1(X)\rightarrow H_1(X,D)\stackrel{\partial}{\longrightarrow} H_0(D) \rightarrow H_0(X)\rightarrow H_0(X,D)\rightarrow 0\ ,$$
    		where $\partial$ is the boundary morphism. In general, the relation between \eqref{eq:es relative} and boundary morphisms is as follows. For $D_i$ an irreducible component of $D$, let  $D'_i$ denote the union of the $D_j$'s for $j\neq i$. Then $D(i):=D_i\cap D'_i$ is a normal crossing divisor inside $D_i$. The long exact sequence in relative homology for the triple $(D'_i,D,X)$ is
		\begin{equation}\label{eq:les relative homology}
		\cdots\rightarrow H_n(X,D'_i)\rightarrow H_n(X,D) \rightarrow H_{n-1}(D,D'_i)\rightarrow H_{n-1}(X,D'_i)\rightarrow \cdots
		\end{equation}
    		 Excision gives an isomorphism $H_{n-1}(D,D'_i)\simeq H_{n-1}(D_i,D(i))$. We thus get a map
    		$$\partial_{D_i} : H_n(X,D)\rightarrow H_{n-1}(D_i,D(i))$$
    		that we call the \emph{partial boundary morphism} with respect to $D_i$. Concretely, in Betti homology, it sends a cycle on $X$ with boundary along $D$ to `the part of its boundary that is supported on $D_i$'. These maps can clearly be composed.

		\begin{proposition}\label{prop:partial boundary}
		The map 
		$$H_n(X,D)\twoheadrightarrow\mathrm{gr}_0^WH_n(X,D) \hookrightarrow \bigoplus_{|I|=n}H_0(D_I)$$ induced by \eqref{eq:es relative} is the composition of the partial boundary morphisms $\partial_{D_i}$ for $i\in I$.
    		\end{proposition}
    		
    		\begin{proof}
            This results from the fact that the spectral sequence \eqref{eq:ss relative} is compatible with the partial boundary morphisms. More precisely, let  $C^\bullet(X,D)$ denote the complex $C^\bullet(\QQ_X)$ from the proof of Proposition \ref{prop:ss relative} and let  $E_r^{p,q}(X,D)$ be the corresponding spectral sequence. One easily checks that the transpose of $\partial_{D_i}$ is induced by the natural morphism $C^{\bullet-1}(D_i,D(i))\rightarrow C^\bullet(X,D)$ corresponding to inclusion of the face $D_i$. This is reflected on the weight-graded quotients by a map $E_2^{p-1,q}(D_i,D(i))\rightarrow E_2^{p,q}(X,D)$. By composing these maps for $i\in I$ we get a map $E_2^{0,0}(D_I)\rightarrow E_2^{n,0}(X,D)$ which is nothing but the map $H^0(D_I)\rightarrow \mathrm{gr}_0^WH^n(X,D)$ dual to \eqref{eq:es relative}. The claim follows from taking linear duals.
    		\end{proof}

    		Our analysis of the weight zero part of the relative cohomology allows us to be more  precise about the coefficients involved in the morphism $c_0^\vee$ and hence the de Rham projection $\pi^{\mm,\dR}$.
    		
    		\begin{lemma}\label{lem: coefficients c zero}
    		Let $(X,A,B)$ satisfy $(\star)_\QQ$ and such that $H=H^n(X\min A \,\lmod B)$ is separated. There exists a number field $K$ such that for every $I$ with $|I|=n$, $D_I\times_{\mathbb{Q}}K$ is a disjoint union of copies of $\mathrm{Spec}(K)$. Then for every embedding $K\subset \mathbb{C}$ the morphism $c_0^\vee$ is defined over $K$:
    		$$c_0^\vee: H_n^\B(X\min A\lmod B)\otimes_\QQ K \longrightarrow H_n^\dR(X\min A\lmod B)\otimes_\QQ K\ .$$
    		For classes $\gamma\in H_\B^\vee$ and $\omega\in H_\dR$, we therefore have
    		$$\pi^{\mm,\dR}\,[H,\gamma,\omega]^\mm = [H,c_0^\vee(\gamma),\omega]^{\mm,\dR} \;\in\, \Pe^{\mm,\dR}_{\mathcal{H}}\otimes_\QQ K\ .$$
    		\end{lemma}
    		
    		\begin{proof}
    		For every $I$ with $|I|=n$, the zero-dimensional smooth variety $D_I$ is a disjoint union of $\mathrm{Spec}(F)$ for $F$ a number field. There are a finite number of number fields appearing in this way and one can choose a number field $K$ that contains all of them and such that the extension $K/\mathbb{Q}$ is Galois. Then for every $I$ with $|I|=n$, the comparison
    		$$H^0_\dR(D_I)\otimes_\QQ\CC \stackrel{\sim}{\longrightarrow} H^0_\B(D_I)\otimes_\QQ\CC$$
    		is defined over $K$. The result then follows from Proposition \ref{prop:grW 0 and 2n} and Proposition \ref{prop:ss relative}.
    		\end{proof}
    		
    		\begin{remark}\label{rem: coefficients c zero}
            The lemma implies that if we restrict our attention to relevant subcategories of $\mathcal{H}$, the de Rham projection will be defined over some fixed number field $K$. In certain contexts of interest one can even take $K=\QQ$ and thus $c_0^\vee$ is defined over $\QQ$. This is the case for the cohomology groups associated to the moduli spaces $\overline{\mathcal{M}}_{0,n}$ which are  studied in the sequel to this paper \cite{BD2} (see also \S\ref{par: SVMZV}). 
            \end{remark}

    \subsection{The highest weight  part of \texorpdfstring{$H^n(X\min D)$}{H(X min D)}}\label{par:highest weight}
    
    	We continue working with a smooth projective complex variety $X$ of dimension $n$ and a simple normal crossing divisor $D$ in $X$. We consider the Verdier dual of the last subsection and analyse it in terms of differential forms. The following spectral sequence is due to Deligne \cite{delignehodge2}.
        		
    		\begin{proposition}\label{prop:ss complement}
    		\begin{enumerate}[1.]
    		\item We have a spectral sequence
    		\begin{equation}\label{eq:ss complement}
    		E_1^{-p,q} = \bigoplus_{|I|=p} H^{q-2p}(D_I)(-p) \;\; \Rightarrow \;\; H^{-p+q}(X\min D)
    		\end{equation}
    		in the category of mixed Hodge structures. It degenerates at $E_2$ and the differential $d_1$ is the alternating sum of the Gysin morphisms
    		$$H^{q-2p}(D_I)(-p)\rightarrow H^{q-2p+2}(D_J)(-p+1) \quad \hbox{for }  J\subset I \  , \  |I\min J|=1 \ . $$
    		\item We have an exact sequence
    		\begin{equation}\label{eq:es complement}
    		0\rightarrow \mathrm{gr}_{2n}^WH^n(X\min D) \rightarrow \bigoplus_{|I|=n}H^0(D_I)(-n) \longrightarrow \bigoplus_{|J|=n-1} H^2(D_J)(-n+1) \ .
    		\end{equation}
    		\end{enumerate}
    		\end{proposition}
    		
    		\begin{proof}
    		This follows from applying Verdier duality to Proposition \ref{prop:ss relative}.
    		\end{proof}
    		
    		If $X$ is a smooth projective Riemann surface and $D\subset X$ a finite set of points then \eqref{eq:es complement} is obtained by applying $\mathrm{gr}^W_2$ to the  long exact sequence:
    		$$0\rightarrow H^1(X)\rightarrow H^1(X\min D)\stackrel{\mathrm{Res}_D}{\longrightarrow} H^0(D)(-1) \rightarrow H^2(X)\rightarrow H^2(X\min D)\rightarrow 0\ ,$$
    	where $\mathrm{Res}_D$ is the Poincar\'{e}--Leray residue along $D$. In general, the relation between \eqref{eq:es complement} and residue morphisms is as follows. The Verdier dual of the long exact sequence \eqref{eq:les relative homology} is
    		$$\cdots\rightarrow H^n(X,D'_i)\rightarrow H^n(X\min D) \rightarrow H^{n-1}(D_i\min D(i))(-1) \rightarrow H^{n+1}(X,D'_i)\rightarrow \cdots$$
    		(note that the arrows go in the same direction since \eqref{eq:les relative homology} is written in terms of homology). We need to compute the morphism in the middle
    		$$H^n(X\min D)\rightarrow H^{n-1}(D_i\min D(i))(-1)$$
    		in (smooth) de Rham cohomology. As the next proposition shows, it is, up to a sign, given by the residue $\mathrm{Res}_{D_i}$ along $D_i$. Recall (see e.g. \cite[III.4]{phambook}) that the latter is defined locally on logarithmic forms by the formula
    		$$\mathrm{Res}_{D_i}\left(\frac{dz}{z}\wedge\alpha+\beta\right) = 2\pi i\, \alpha_{|D_i}\ ,$$
    		where $z$ is a local coordinate for $D_i$ and $\alpha$, $\beta$ do not have singularities along $D_i$ .
    		
    		\begin{proposition}\label{prop:sign residue}
    		We have a commutative diagram
    		$$\xymatrixcolsep{4pc}\xymatrix{
    		H_n^\B(X,D)\otimes_\QQ\CC \ar[r]^{\partial_i} \ar[d]_{\simeq} & H_{n-1}^\B(D_i,D(i)) \otimes_\QQ\CC \ar[d]^{\simeq}\\
    		H^n_{\dR,\mathcal{C}^\infty}(X,D)^\vee \ar[r]^{\partial_i^{\dR}} & H^{n-1}_{\mathrm{dR},\mathcal{C}^\infty}(D_i,D(i))^\vee \\
    		H^n_{\mathrm{dR},\mathcal{C}^\infty}(X\backslash D) \ar[u]^{\simeq}\ar[r]^-{(-1)^{n-1}\mathrm{Res}_{D_i}} & H^{n-1}_{\mathrm{dR},\mathcal{C}^\infty}(D_i\min D(i))\ar[u]_{\simeq}
    		}$$
    		where the top vertical arrows are induced by the comparison between (smooth) de Rham and Betti cohomology, and the bottom vertical arrows are induced by the Verdier  duality  isomorphism as in Theorem \ref{thm: PDNCD}, where we recall that the dimension of $X$ is $n$.
    		\end{proposition}
    		
    		\begin{proof}
    		Since the only subtle point is the sign $(-1)^{n-1}$ we give a  proof in the case where $D$ has a single irreducible component and drop the index $i$, leaving the general case as an exercise for the reader. The map $\partial^\dR$ in the diagram is defined so that the upper square commutes; its transpose $(\partial^\dR)^\vee:H^{n-1}_{\mathrm{dR},\mathcal{C}^\infty}(D) \rightarrow H^n_{\mathrm{dR},\mathcal{C}^\infty}(X,D)$ is computed as follows. For a closed smooth $(n-1)$-form $\omega$ on $D$ one has $(\partial^\dR)^\vee[\omega]=[d\widetilde{\omega}]$ where $\widetilde{\omega}$ is any smooth $(n-1)$-form on $X$ whose restriction to $D$ is $\omega$. This is because Stokes' theorem then gives
    		$$\int_\gamma d\widetilde{\omega} = \int_{\partial\gamma} \omega $$
    		for every relative $n$-chain $\gamma$ on $(X,D)$. One then needs to prove that the lower square commutes. We use Proposition \ref{prop:PD simpler log forms conjugate} and Remark \ref{rem:without conjugation} and let $\nu$ be a closed section of $\A^n_X(\log D)$, and $\omega$ be a closed section of $\A^{n-1}_D$. For $\widetilde{\omega}$ chosen as above, we need to prove the equality:
    		$$\int_X \nu\wedge d\widetilde{\omega} = (-1)^{n-1} \int_D  \mathrm{Res}_D(\nu)\wedge\omega\ .$$
    		We note as in Lemma \ref{lem:wedge product integrable simpler} that since $d\widetilde{\omega}$ vanishes along $D$, the left-hand side is an absolutely convergent integral. For small enough $\varepsilon>0$, we let $T_\varepsilon\rightarrow D$ denote an open tubular $\varepsilon$-neighbourhood of $D$ inside $X$ and write $X_\varepsilon=X\min T_\varepsilon$. We then have
    		$$\int_X\nu\wedge d\widetilde{\omega} = \lim_{\varepsilon\rightarrow 0} \int_{X_\varepsilon}\nu\wedge d\widetilde{\omega} \ .$$
    		By Stokes' theorem we have
    		$$\int_{X_\varepsilon}\nu\wedge d\widetilde{\omega} = (-1)^n \int_{X_\varepsilon}d(\nu\wedge\widetilde{\omega}) = (-1)^n \int_{\partial X_\varepsilon} (\nu\wedge\widetilde{\omega})_{|\partial X_\varepsilon}\ .$$
    		Now $\partial X_\varepsilon$ is the sphere bundle of the normal bundle of $D$ inside $X$, with the opposite orientation from that induced from $X$, and Cauchy's theorem gives
    		$$\lim_{\varepsilon\rightarrow 0}\int_{\partial X_\varepsilon}(\nu\wedge\widetilde{\omega})_{|\partial X_\varepsilon} = -\int_D \mathrm{Res}_D(\nu\wedge\widetilde{\omega}) = -\int_D \mathrm{Res}_D(\nu)\wedge\omega\ ,$$
    		which completes the proof.
    		\end{proof}

    		\begin{proposition}\label{prop:sign iterated residue}
    		We have a commutative diagram
    		$$\xymatrix{
    		H^\B_n(X,D)\otimes_\QQ\CC \ar[r]^-{\partial} \ar[d]_{\simeq} & \bigoplus_{|I|=n} H_0^\B(D_I)\otimes_\QQ\CC \ar[d]^{\simeq} \\
    		H^n_{\dR,\mathcal{C}^\infty}(X\min D) \ar[r]_-{R} & \bigoplus_{|I|=n} H^0_{\dR,\mathcal{C}^\infty}(D_I)
    		}$$
    		where 
    		\begin{enumerate}
    		    \item the vertical isomorphisms are induced by the comparison between (smooth) de Rham and Betti cohomology and by Verdier duality;
    		    \item  $\partial$ is the composition of the partial boundary maps $\partial_{D_i}$ for $i\in I$;
    		    \item  $R$ is $(-1)^{\frac{n(n-1)}{2}}$ times the composition of the  maps $\mathrm{Res}_{D_i}$ for $i\in I$. 
    		\end{enumerate}
    		\end{proposition}
  
    		\begin{proof} This follows from iteratively applying Proposition \ref{prop:sign residue} and noting that $1+2+\cdots+(n-1)=\frac{n(n-1)}{2}$.
    		\end{proof}
    		
    		The next proposition is well-known and gives a concrete interpretation of $F^nH^n(X\min D)$ in terms of the complex of holomorphic logarithmic forms $\Omega^\bullet_X(\log D)\subset \A^\bullet_X(\log D)$.
    		
    		\begin{proposition}\label{prop:global log Hodge}
    		The map $\Gamma(X,\Omega^n_X(\log D)) \rightarrow H^n_{\mathrm{dR},\mathcal{C}^\infty}(X\min D)$  sending a (necessarily closed) global logarithmic form to its cohomology class yields an isomorphism:
    		\begin{equation}\label{eq:log forms Hodge filtration}
        	 \Gamma(X,\Omega^n_X(\log D)) \simeq F^nH^n_{\mathrm{dR},\mathcal{C}^\infty}(X\min D) \ .
    		\end{equation}
    		\end{proposition}
    		
    		\begin{proof}
    		By \cite{delignehodge2}, the Hodge filtration on the cohomology of $X\min D$ is induced, via the isomorphism $H^k(X\min D)\simeq \mathbb{H}^k(X,\Omega^\bullet_X(\log D))$, by the  filtration $F^p\Omega^\bullet_X(\log D)=\Omega^{\bullet\geq p}_X(\log D)$. The corresponding spectral sequence $E_1^{p,q}=H^q(X,\Omega^p_X(\log D))$ degenerates at $E_1$, so $E_1^{p,q} = \mathrm{gr}^p_FH^{p+q}(X\min D)$. Now set $(p,q)=(n,0)$.
    		\end{proof}
    		
    		\begin{corollary}\label{coro:global log Hodge pair}
    		Let $(X,A,B)$ satisfy $(\star)_\CC$. Then the natural map $\Gamma(X,\Omega^n_X(\log A))\rightarrow H^n_{\mathrm{dR},\mathcal{C}^\infty}(X\min A\;\mathrm{mod}\; B)$ sending a (necessarily closed) global logarithmic form to its cohomology class yields an isomorphism:
    		$$\Gamma(X,\Omega^n_X(\log A)) \simeq F^nH^n_{\mathrm{dR},\mathcal{C}^\infty}(X\min A \;\mathrm{mod}\;B)\ .$$
    		\end{corollary}
    		
    		\begin{proof}
    		This follows from Proposition \ref{prop:global log Hodge} and the isomorphism of Proposition \ref{prop:grW 0 and 2n} for $\mathrm{gr}^n_F=F^n$.
    		\end{proof}

    \subsection{Computation of the de Rham projection}
    		Let $X$ be a smooth projective variety over $\QQ$ of dimension $n$ and let $D$ be a simple normal crossing divisor in $X$. We form the following commutative diagram, where the second and third row are induced by the exact sequences \eqref{eq:es relative} and \eqref{eq:es complement}. The vertical maps between these two rows are induced by the comparison between de Rham and Betti cohomology and by Verdier duality. We recall from Proposition \ref{prop:sign iterated residue} that $\partial$ denotes the composite of partial boundary maps, and $R$ denotes the composite of residue maps times the sign $(-1)^{\frac{n(n-1)}{2}}$.
    		
    		$$\xymatrix{
    		 &H_n^{\mathrm{B}}(X,D)\otimes_\QQ\CC \ar@{->>}[d]& \\
    		 &\mathrm{gr}_0^WH_n^{\mathrm{B}}(X,D)\otimes_\QQ\CC\; \ar@{^(->}[r]^{\partial}\ar[d]_{\simeq} & \displaystyle\bigoplus_{|I|=n}H_0^{\mathrm{B}}(D_I)\otimes_\QQ\CC \ar[d]_-{\simeq} \\
    		 &\mathrm{gr}_{2n}^W H^n_{\mathrm{dR}}(X\min D)\otimes_\QQ\CC\;\ar@{^(->}[r]^{R} & \displaystyle\bigoplus_{|I|=n}H^0_{\mathrm{dR}}(D_I)\otimes_\QQ\CC \\
    	     \Gamma(X,\Omega_X^n (\log D))\ar[r]^{\simeq}&F^nH^n_{\mathrm{dR}}(X\min D)\otimes_\QQ\CC \ar[u]^{\stackrel{?}{\simeq}}\ar@{^(->}[d]& \\
    		&H^n_{\mathrm{dR}}(X\min D)\otimes_\QQ\CC & 
    		}$$
    		The arrow marked $?$ in the diagram is an isomorphism if and only if $H^n(X,D)$ is separated because of Verdier duality: $H^n(X\min D)\simeq H^n(X,D)^\vee(-n)$. In this case, the vertical composite
    		$$H_n^{\mathrm{B}}(X,D)\otimes_\QQ\CC \longrightarrow H^n_\dR(X\min D)\otimes_\QQ\CC\simeq H_n^\dR(X,D)\otimes_\QQ\CC\ .$$ 
    		is the map $c_0^\vee$, introduced in \eqref{eq:c0vee definition}, for the object $H^n(X,D)\in\mathcal{H}$.
    		We have thus proved the following theorem, which gives a means to compute the map $c_0^\vee$ (and hence the de Rham projection) in practice.
    		
    		\begin{theorem}\label{coro:c0vee}
    		Let $X$ be a smooth projective variety over $\QQ$ of dimension $n$, let $D$ be a simple normal crossing divisor in $X$, and assume that $H^n(X,D)$ is separated. For a class $\gamma\in H_n^\B(X,D)$ in Betti homology, $c_0^\vee(\gamma)\in H^n_\dR(X\min D)\otimes_\QQ\CC$ is the class of  the unique global logarithmic form $$(2\pi i)^{-n}\,\nu_\gamma\;\in \Gamma(X,\Omega^n_X(\log D))$$ 
    		such that for every set $\{i_1,\ldots,i_n\}$ of $n$ irreducible components of $D$ we have
    		$$\mathrm{Res}_{D_{i_n}}\cdots \mathrm{Res}_{D_{i_1}}((2\pi i)^{-n}\,\nu_\gamma) = (-1)^{\frac{n(n-1)}{2}} \partial_{D_{i_n}}\cdots \partial_{D_{i_1}}(\gamma)\ .$$
    		\end{theorem}

    		\begin{remark}\label{rem:c0vee general}
    		In the general case of an object $H=H^n(X\min A \;\mathrm{mod}\; B)$, for $(X,A,B)$ satisfying $(\star)_\QQ$, Proposition \ref{prop:grW 0 and 2n} and Proposition \ref{prop:relative separated} imply that in order to compute $c_0^\vee(\gamma)$, for $\gamma\in H_\B^\vee$, one is reduced to the case of $H^n(X,B)$. Indeed, one simply applies the recipe given in Theorem \ref{coro:c0vee} to the image of $\gamma$ in $H_n^\B(X,B)$.
    		Note that we have the inclusion $\Gamma(X,\Omega^n_X(\log B))\subset \Gamma(X,\A^n_X(\log A\cup B)(-A))$, so that a differential form $\nu_\gamma$ from Theorem \ref{coro:c0vee} satisfies the assumption of Theorem \ref{thm:sv recipe}.
            \end{remark}
            
            \begin{remark}\label{rem: coefficients nu}
            Let us fix, as in Lemma \ref{lem: coefficients c zero}, a number field $K$ such that for every $I$ with $|I|=n$, $D_I\times_\QQ K$ is a disjoint union of $\mathrm{Spec}(K)$. Then the logarithmic form $\nu_\gamma$ lives in the $K$-subspace
            $$\Gamma(X_K,\Omega^n_{X_K}(\log D_K)) \subset \Gamma(X,\Omega^n_X(\log D))\ ,$$
            where $X_K=X\times_\QQ K$, $D_K=D\times_\QQ K$, and $\Omega_{X_K}(\log D_K)$ denotes the (Zariski) sheaf of algebraic differential $n$-forms on $X_K$ with logarithmic singularities along $D_K$.
            \end{remark}

            \begin{remark}\label{rem:simple ncd}
            We worked with simple normal crossing divisors $D$ in order to apply Poincar\'{e} duality to the multiple intersections $D_I$. In order to treat the general case one has to replace those with their normalisations as in \cite{delignehodge2}, and we leave the details to the reader.
            \end{remark}

    	\subsection{Examples}
    
            \subsubsection{The case of $\mathbb{P}^1$} \label{ex:dR projection P1}
            
            Let $A,B\subset\PP^1(\CC)$ be  disjoint finite sets of points and set $H=H^1(\PP^1\min A,B)$. Let  $\gamma_{b_1,b_2}$ denote any path on $\mathbb{P}^1(\CC)\min A$ from $b_1 \in B$ to $b_2 \in B$. We have 
            \begin{equation}\label{eq:c0veeP1}
            c_0^\vee([\gamma_{b_1,b_2}]) = \frac{1}{2\pi i} \,d\log\left(\frac{z-b_2}{z-b_1}\right)\ .
            \end{equation}
            (If $b_i=\infty$ then we replace $z-b_i$ in this expression by $1$.)
            This follows from Theorem \ref{coro:c0vee} and the computations 
            $$\partial_{b_1}[\gamma_{b_1,b_2}] = -1 \;\; , \;\; \partial_{b_2}[\gamma_{b_1,b_2}]= +1$$
            and
            $$\mathrm{Res}_{b_1}\left(\frac{1}{2\pi i}d\log\left(\frac{z-b_2}{z-b_1}\right)\right) = -1 \;\; , \;\; \mathrm{Res}_{b_2}\left(\frac{1}{2\pi i}d\log\left(\frac{z-b_2}{z-b_1}\right)\right) = +1\ .$$

            It is important to note (see Remark \ref{rem:c0vee general}) that the differential form \eqref{eq:c0veeP1} only depends on the class of $\gamma_{b_1,b_2}$ in $H_1^{\mathrm{B}}(\mathbb{P}^1,B)$, i.e., on the pair $(b_1,b_2)$.

            \subsubsection{The case of $(\mathbb{P}^1)^n$}\label{ex: dR projection P1 powers}
            
            We set $H=H^n((\PP^1)^n,D)$ where $D$ is the union of the divisors $z_i=0$ and $z_i=1$ for $i=1,\ldots,n$. The Betti homology group $H_\B^\vee$ is one-dimensional with basis the class of the hypercube $[0,1]^n$. We can then compute:
            \begin{equation}\label{}
            c_0^\vee([0,1]^n) = (-1)^{\frac{n(n+1)}{2}}(2\pi i)^{-n} \frac{dz_1\wedge\cdots\wedge dz_n}{z_1(1-z_1)\cdots z_n(1-z_n)}\cdot 
            \end{equation}\label{eq:c0veeP1powers}
            This follows from Theorem \ref{coro:c0vee} and the computations
            $$\partial_{z_n=1}\partial_{z_{n-1}=1}\cdots \partial_{z_1=1}\,[0,1]^n = +1 $$
            and 
            $$\mathrm{Res}_{z_n=1}\mathrm{Res}_{z_{n-1}=1}\cdots\mathrm{Res}_{z_1=1}\left((2\pi i)^{-n}\frac{dz_1\wedge\cdots \wedge dz_n}{z_1(1-z_1)\cdots z_n(1-z_n)}\right) = (-1)^n\ .$$
            (Since $H$ has rank one it is enough to do the computation at one point.)

    \section{Functoriality of the formula for single-valued integrals}\label{par:sv functoriality}
    
        Let $u:H\rightarrow H'$ be a morphism in $\mathcal{H}$ and let $\omega\in H_\dR$ and $f'\in H'_\dR$ be two classes. Then we have an equality of de Rham periods 
        $$[H',f',u_\dR(\omega)]^{\mm,\dR} = [H,u_\dR^\vee(f'),\omega]^{\mm,\dR}\ .$$
        In particular these two matrix coefficients have the same image under  the single-valued period map $\s$. When $u$ has geometric origin this has an interpretation in terms of our formula for the single-valued period map (Theorem \ref{thm:sv recipe general}) and gives rise to single-valued analogues of the usual rules of integration.
        We discuss three important special cases.
        
            \subsection{Change of variables}
                
                Let $(X,A,B)$ and $(X',A',B')$ satisfy $(\star)_\QQ$ with $X$ of dimension $n$ and $X'$ of dimension $n'$, and set $H=H^k(X\min A\;\mathrm{mod}\;B)$ and $H'=H^k(X'\min A'\;\mathrm{mod}\; B')$. Let 
	            $$\varphi:(X',B')\rightarrow (X,B)$$
	            be a morphism of pairs such that $\varphi^{-1}(A)\subseteq A'$. 
	            Then it induces a morphism of pairs $(X'\min A',B'\min A'\cap B') \rightarrow (X\min A,B\min A\cap B)$ and we get a pullback morphism
	            $$\varphi^*:H\rightarrow H'\ ,$$
	            whose classical interpretation in terms of periods is integration by substitution (change of variables). The interpretation of the transpose
	            $$\varphi_*=(\varphi^*)^\vee : H'^\vee\rightarrow H^\vee\ .$$
	            in terms of differential forms under the isomorphisms
	            $$H^\vee\simeq H^{2n-k}(X\min B\;\mathrm{mod}\;A)(n) \;\;\mbox{ and }\;\; H'^\vee\simeq H^{2n'-k}(X'\min B'\;\mathrm{mod}\; A')(n') \ ,$$
	            is not obvious in general. It should be thought of as `integration along the fibers' of $\varphi$ whenever this makes sense. 
	            In the setting of Theorem \ref{thm:sv recipe}, the functoriality of $\s$ now translates as the formula
	            \begin{equation}\label{eq:functoriality sv recipe}
	            \int_{X'(\CC)} \nu'\wedge\overline{\varphi^*(\omega)}= \int_{X(\CC)} \varphi_*(\nu')\wedge\overline{\omega}\ ,
	            \end{equation}
	            which is an instance of a projection formula.

            \subsection{Cauchy and Stokes}\label{par: cauchy stokes}

                Since the cohomological interpretations of Cauchy's theorem and Stokes's theorem are Poincar\'{e} dual to each other, their interpretations in terms of our formula for the single-valued map are one and the same. A simple special case is as follows. If $X$ is a smooth projective variety of dimension $n$ over $\QQ$ and $D$ a smooth divisor in $X$, we have a morphism in $\mathcal{H}$:
                $$u:H^{n-1}(D)\rightarrow H^n(X,D)\ .$$
                The functoriality of our formula for the single-valued map is contained in the following  formula (see  the proof of Proposition \ref{prop:sign residue}):
                $$\int_{X(\CC)}\nu\wedge d\widetilde{\omega} = (-1)^{n-1} \int_{D(\CC)} \mathrm{Res}(\nu)\wedge \omega\ .$$
                It holds for $\omega\in\Gamma(D,\A^{n-1}_D)$ and $\nu\in\Gamma(X,\A^n_X(\log D))$ closed forms (in the setting of Theorem \ref{thm:sv recipe general}, $\omega$ should be replaced with $\mathrm{conj}^*(\omega)$.). It could be called a `Cauchy--Stokes theorem' and is an important tool in the computation of single-valued periods. 
                \begin{example}
               Let $X=\PP^1$, $D=\{0,\infty\}$. For any smooth function $\psi$ on $\PP^1(\CC)$ 
                $$\frac{1}{2\pi i}\int_{\PP^1(\CC)} \frac{dz}{z}\wedge d\psi(\overline{z}) = \psi(0)-\psi(\infty)\ .$$
                It is the single-valued analogue of the formula $\int_\infty^0 d\psi(z)= \psi(0)-\psi(\infty)\ .$
            \end{example} 
                
            \subsection{Fubini}\label{par:fubini}
                Let $(X_1,A_1,B_1)$ and $(X_2,A_2,B_2)$ satisfy $(\star)_\QQ$ with $X_j$ of dimension $n_j$ and set $X=X_1\times X_2$, $A=A_1\times X_2\cup X_1\times A_2$ and $B=B_1\times X_2\cup X_1\times B_2$.
                For classes $\gamma_j\in H_{n_j}^\B(X_j\min A_j\;\mathrm{mod}\;B_j)$ we denote by $\gamma_1\times\gamma_2$ the class induced in  $H_{n_1+n_2}^\B(X\min A\;\mathrm{mod}\; B)$ by the K\"{u}nneth formula. By Theorem \ref{coro:c0vee}, the differential forms $\nu_{\gamma_1}=c_0^\vee(\gamma_1)$, $\nu_{\gamma_2}=c_0^\vee(\gamma_2)$ and $\nu_{\gamma_1\times\gamma_2}=c_0^\vee(\gamma_1\times\gamma_2)$ are related by 
                
                \begin{align*}
                \nu_{\gamma_1\times\gamma_2} & = (-1)^{\frac{(n_1+n_2)(n_1+n_2-1)}{2}}(-1)^{\frac{n_1(n_1-1)}{2}}(-1)^{\frac{n_2(n_2-1)}{2}}\nu_{\gamma_1}\wedge\nu_{\gamma_2} \\
                & = (-1)^{n_1n_2}\nu_{\gamma_1}\wedge\nu_{\gamma_2}\ .
                \end{align*}
                
                For differential forms $\omega_j\in \Gamma(X_j,\Omega^{n_j}_{X_j}(\log A_j))$ we thus get
                \begin{align*}
                 \int_{X_1(\CC)\times X_2(\CC)}\nu_{\gamma_1\times\gamma_2}\wedge\overline{\omega_1\wedge\omega_2} & = (-1)^{n_1n_2} \int_{X_1(\CC)\times X_2(\CC)} \nu_{\gamma_1}\wedge\nu_{\gamma_2}\wedge \overline{\omega_1}\wedge\overline{\omega_2}\\    
                 & = \int_{X_1(\CC)\times X_2(\CC)} (\nu_{\gamma_1}\wedge\overline{\omega_1})\wedge(\nu_{\gamma_2}\wedge\overline{\omega_2}) \\
                 & = \int_{X_1(\CC)}\nu_{\gamma_1}\wedge\overline{\omega_1}\; \int_{X_2(\CC)} \nu_{\gamma_2}\wedge\overline{\omega_2}\ .
                \end{align*}
                This, after inserting powers of $2\pi i$,  is the single-valued analogue of Fubini's theorem
                $$\int_{\gamma_1\times\gamma_2}\omega_1\wedge\omega_2 = \int_{\gamma_1}\omega_1\int_{\gamma_2}\omega_2\ .$$

\section{Some examples} \label{sect: examples}

We illustrate the definitions above with some simple examples.

    \subsection{Single-valued \texorpdfstring{$2\pi i$}{2 pi i}}\label{par: example lefschetz}
    We wish to compute explicitly the single-valued Lefschetz period $\s(\LL^{\mm,\dR})$ where $\LL^{\mm,\dR} = [H^1(\GG_m), [\frac{dz}{z}]^{\vee}, [\frac{dz}{z}]]^{\mm,\dR}$. Since 
    $$H^1(\PP^1 \backslash \{0,\infty\})^{\vee} \simeq H^1(\PP^1, \{0,\infty\})(1)$$ the dual class $[\frac{dz}{z}]^{\vee}$ can be represented by $\frac{1}{2\pi i} [\nu]$ where $\nu$ is a closed smooth $1$-form on $\PP^1(\CC)$ whose relative cohomology class generates  $H^1(\PP^1, \{0,\infty\})$ and satisfies
    $$\int_{\PP^1(\CC)} \nu\wedge\frac{dz}{z}  = 2\pi i\ .$$
    Since $H^1(\PP^1)=0$, we necessarily have  $\nu = df$, for some smooth function $f$ satisfying $f(\infty)-f(0)=1$. For example, we may take  $f= \frac{|z|^2}{1+ |z|^2}$ to be radial, giving 
    $$\nu = df =   \frac{  z d \overline{z}+ \overline{z} dz }{(1+|z|^2)^2}$$ and we verify by a well-known calculation in polar coordinates, for example, that 
    \begin{equation} \label{Fubini-Study} \int_{\PP^1(\CC)} \nu\wedge\frac{dz}{z} = - \int_{\PP^1(\CC)} \frac{dz \wedge d\overline{z}}{(1+|z|^2)^2} = 2 \pi i\ .\end{equation} 
    We note that $\nu$ is a global section of $\A^1_{\PP^1(\CC)}(\log\{0,\infty\})(-\{0,\infty\})$. It follows from Theorem \ref{thm:sv recipe} that 
    $$\s(\LL^{\mm,\dR}) =\frac{1}{2\pi i} \int_{\PP^1(\CC)} \nu\wedge\frac{d\overline{z}}{\overline{z}} \ .$$
    But since $\nu$ is radial, $\nu\wedge d \log |z|^2 =0$ which implies that  
    $$\nu\wedge \frac{d \overline{z}}{\overline{z}}  = - \nu\wedge \frac{dz}{z} $$ and we find from (\ref{Fubini-Study}) that $\s(\LL^{\mm,\dR})=-1$, as expected.

  \subsection{Universal elliptic curve and non-holomorphic modular forms} \label{par: UnivElliptic}
  For any field $k\subset \CC$ and  $u,v \in  k$  such that $u^3-27 v^2\neq 0$, let $E$ denote the non-singular elliptic curve over $k$ defined by the affine equation $y^2 = 4x^3 -ux -v$. Let $E' = E \backslash \{0\}$ denote the punctured curve.
   The de Rham cohomology of $E$  over $k$ satisfies $H^1_{\dR}(E) = H^1_{\dR}(E')$ and is spanned by the classes of the forms
\begin{equation} \label{EllipticdRbasis}  \frac{dx}{y}  \ \in \  \Gamma(E, \Omega^1_{E}) \quad \hbox{ and } \quad  x \frac{dx}{y} \  \in \  \Gamma(E', \Omega^1_{E'}) \ .\end{equation}
  For any framing, i.e., choice of  ordered basis $\{\alpha, \beta\}$ of the homology $H_1(E'(\CC);\ZZ) \cong H_1(E(\CC);\ZZ)$ such that the intersection pairing $\alpha\cdot \beta =1$,  the corresponding period matrix is 
  $$P = \begin{pmatrix} \omega_1 & \eta_1 \\ \omega_2 & \eta_2  \end{pmatrix} $$
  where $\tau= \omega_2/\omega_1 $ lies in the upper half plane  $\mathbb{H} = \{\tau \in \CC: \mathrm{Im}\, \tau>0\}$. The numbers $\eta_1, \eta_2$ are integrals of $x\frac{dx}{y}$  along $\alpha$ and $\beta$ and are traditionally called quasi-periods.  As is  well-known, there is an isomorphism
  $ \CC /  (\ZZ + \tau \ZZ) \overset{\sim}{\rightarrow} E(\CC)  $
  induced by $ z \mapsto \left(\frac{\wp_{\tau}(z)}{\omega_1^2},  \frac{\wp_{\tau}'(z)}{\omega_1^3}\right) $,  where $\wp_{\tau}(z)$ denotes the Weierstrass $\wp$-function and $\wp_{\tau}'(z)$ its derivative with respect to $z$.  The pullback of  the differential $dx/y$ under this map is  $\omega_1 \,dz$, and the classes $\alpha, \beta$ are represented by the straight paths from $0$ to  $1, \tau$ respectively.

Fricke  and Legendre respectively proved that
\begin{equation}\label{Fricke} \GG_2(\tau) = - \frac{1}{2} \frac{\omega_1 \eta_1 }{(2\pi i)^2} \qquad \hbox{ and } \qquad \omega_1 \eta_2 - \eta_1 \omega_2 = 2\pi i\ ,
\end{equation} 
where, if we denote $q= \exp(2 \pi i \tau)$ as usual, 
$$\GG_2(\tau) =-  \frac{1}{24} +\sum_{n\geq 1} \sigma_1(n) q^n = - \frac{1}{24} + q+ 3 q^2 + 4q^3 + 7 q^4 + \cdots $$  
is the normalised Eisenstein series of weight $2$.  As is well-known, it does not transform like a modular form.  However, its modified real analytic version 
$$\GG^*_2(\tau)  = \GG_2(\tau) + \frac{1}{8 \pi \,\mathrm{Im} \, \tau} $$
does transform like a modular form of weight $2$.  We shall see that it can be interpreted as a single-valued period of the universal elliptic curve.

\begin{corollary} 
The single-valued period matrix $\overline{P}^{-1} P$ equals  
\begin{equation} \label{EllipticSVmatrix}    \begin{pmatrix} \overline{\lambda}^{-1} & 0 \\ 0 & \overline{\lambda}
\end{pmatrix} 
\begin{pmatrix}   \overline{\mm (\tau)}     &    (4 \pi\, \mathrm{Im} (\tau) )^{-1} \left( \mm (\tau)  \overline{\mm 
(\tau)} -1 \right) \\ - 4 \pi \,\mathrm{Im} (\tau)    &      - \mm (\tau)     \end{pmatrix} 
\begin{pmatrix} \lambda & 0 \\  0 & \lambda^{-1} 
\end{pmatrix}   \end{equation}
where we write $\omega_1 = \lambda \, 2 \pi i $ and 
$ \mm (\tau) =  - 8 \pi\,\mathrm{Im}(\tau)\, \GG_2^*(\tau) $. 
\end{corollary}

\begin{proof}  By \eqref{Fricke}, we may write 
$$P =      \begin{pmatrix} 2\pi i  &  - 4 \pi i \,\mathbb{G}_2(\tau) \\  2\pi i \tau   & 1 - 4\pi i \,\tau \mathbb{G}_2(\tau)  \end{pmatrix} \begin{pmatrix} \lambda &  0   \\  0    &  \lambda^{-1}  \end{pmatrix}  $$   
and the rest is a straightforward calculation.
\end{proof}
The quantity $4 \pi \, \mathrm{Im} (\tau)$ is proportional to the area of the period lattice, and the function $\mm(\tau)$ plays an important role in the differential theory of modular forms.

Call a function $f$ on $\mathbb{H}$  modular of weights $(r,s)$ if it satisfies
    $$f(\gamma \, \tau)  = (c\tau+ d)^r (c \overline{\tau} +d)^s f(\tau)$$
    for all $\gamma = \left(\begin{smallmatrix}  a & b \\ c &  d \end{smallmatrix} \right) \in \mathrm{SL}_2(\ZZ)$. The entries of the single-valued period matrix proportional  to  $\lambda^r \overline{\lambda}^s$
        are  modular with weights $(-r,-s)$, where $(r,s)$ take the four possible  values $( \pm 1, \pm 1)$. These weights are the weights with respect to the splitting of the Hodge filtration defined by the de Rham basis associated to the forms \eqref{EllipticdRbasis}.
    This observation  is the starting point for a theory of real analytic modular forms with two weights  obtained by replacing the cohomology of $\mathcal{E}$ with its unipotent fundamental group (see \cite{brownCNHMF3} and prequels).

    \begin{remark}  The four single-valued periods are given by the  integrals:
    $$\frac{1}{2\pi i} \int_{E(\CC)} \nu \wedge \overline{\omega}  \qquad   \hbox{  where }\ \omega  , \nu  \  \in  \   \left\{ \frac{dx}{y} \ , \   x \frac{dx}{y} \right\} \ , $$
    which are regularised according to Remark \ref{rem: higher order poles}, since $x \frac{dx}{y}$ has a double pole. A simple computation confirms that the convergent  single-valued period 
    $$ \frac{1}{2\pi i} \int_{E(\CC)} \frac{dx}{y} \wedge \overline{\frac{dx}{y}}  = - \frac{|\omega_1|^2}{\pi} \, \mathrm{Im} (\tau)= - 4 \pi \, \lambda \overline{\lambda}  \, \mathrm{Im} (\tau)$$
    is proportional to the area of the period lattice $\ZZ \omega_1 +  \ZZ \omega_2$ of $E$.
  \end{remark}
  
 \begin{remark}
 Note that the determinant of the single-valued period matrix is  always equal to $-1$, as expected, since this is the single-valued period of $\bigwedge^2 H^1(E)\cong \QQ(-1)$. However, the trace of the single-valued period matrix \eqref{EllipticSVmatrix}  is not always zero. The reason for this is that we have computed the single-valued period of a (universal) family of elliptic curves -- the trace of the single-valued matrix is only guaranteed to vanish when the fibers correspond to an elliptic curve which is defined over the real numbers (the matrix defined in \eqref{Pssigmaformula} does not have vanishing trace in general). 
 \end{remark}
    
    \subsection{Logarithms} \label{par:example log}
    
    		Let $k$ be a subfield of $\CC$ and  $a\in k^\times\min \{1\}$. A path $\gamma$ from $1$ to $a$ in $\CC^\times$ defines a determination of the logarithm of $a$:
    		$$\log(a)=\int_{\gamma}\omega\ ,\qquad \hbox{ where }  \omega=\frac{dz}{z}\ .$$
    		The relevant cohomology group is
    		$H=H^1(\PP^1\min \{0,\infty\},\{1,a\})$. 
    		The $\mathcal{H}(k)$-period corresponding to $\log(a)$ is then defined to be:
    		$$\log^\mm(a) = [H,[\gamma],[\omega]]^\mm\ .$$ 
    		By Theorem \ref{coro:c0vee} and Example \ref{ex:dR projection P1}, its image under the de Rham projection is
    		$$\log^{\mm,\dR}(a) = [H,[\nu(1)],[\omega]]^{\mm,\dR} \qquad \hbox{ where } \qquad \nu=d\log\left(\frac{z-a}{z-1}\right)\ .$$
    		By Theorem \ref{thm:sv recipe}  and Remark \ref{rem:sv recipe number field} its single-valued period is 
    		$$\langle[\nu(1)],\s[\omega]\rangle =\frac{1}{2\pi i}\int_{\PP^1(\CC)}d\log\left(\frac{z-a}{z-1}\right)\wedge\frac{d\overline{z}}{\overline{z}}\ .$$
    		
    		The integrand is smooth on $\PP^1(\CC)$ except around the points $0$, $\infty$, $1$, $a$. For some sufficiently small  $\eps>0$ we let $P_\eps$ denote the complement of the union of the open disks of radius $\eps$ around those four points inside $\PP^1(\CC)$, and set
    		\begin{align*}
    		I_\eps  &=  \frac{1}{2\pi i}\int_{P_\eps} d\log\left(\frac{z-a}{z-1}\right)\wedge\frac{d\overline{z}}{\overline{z}}\\
    		& =  \frac{1}{2\pi i}\int_{P_\eps}d\log\left(\frac{z-a}{z-1}\right)\wedge\left(\frac{dz}{z}+\frac{d\overline{z}}{\overline{z}}\right)\\
    		& =  -\frac{1}{2\pi i}\int_{P_\eps}d\left(\log|z|^2 \,d\log\left(\frac{z-a}{z-1}\right)\right) \\
    		& = \frac{1}{2\pi i}\left(\int_{\partial D_0(\eps)}+\int_{\partial D_\infty(\eps)}+\int_{\partial D_1(\eps)} +\int_{\partial D_a(\eps)}\right) \log|z|^2 \,d\log\left(\frac{z-a}{z-1}\right)
    		\end{align*}

		The last equality follows from Stokes' theorem, where  $\partial D_i(\eps)$ denotes the positively oriented circle of radius $\eps$ around $i$.	The contributions at $0$ and $\infty$ vanish. The contributions at $a$ and $1$ equal
		$$\frac{1}{2\pi i} \int_{\partial D_a(\eps)} \log |z|^2 \frac{dz}{z-a} =\log|a|^2 + O(\varepsilon)$$
		and $\log|1|^2=0$, respectively, and so $I_\varepsilon = \log|a|^2 + O(\varepsilon)$. Letting $\varepsilon$ go to zero we deduce the  formula predicted by Remark \ref{introExlog}:
		$$\frac{1}{2\pi i}\int_{\PP^1(\CC)}d\log\left(\frac{z-a}{z-1}\right)\wedge\frac{d\overline{z}}{\overline{z}}= \log|a|^2\ . $$

\subsection{Green's functions, and N\'{e}ron--Tate heights for curves} \label{par: Greens}  The following discussion is a  variant of classical results \cite{grossheights,Lang}.  Our presentation is close  in spirit to the interpretation of heights due to Bloch, Beilinson, and Scholl  \cite{blochheight, BeilinsonHeights, Scholl} although our exact formulation does not seem to be in the literature.  The height is usually expressed  as an integral of a one-form with complex coefficients. In order to make the connection with the theory of periods, we must insist upon algebraic forms, and we find that the height is a quotient of  two (determinants of) single-valued periods. The computations of \cite[\S 3.5]{kontsevichzagier} suggests that a similar statement in fact holds for the full regulator on an elliptic curve over $\QQ$, and can also be modified to incorporate contributions from the local height pairings at finite primes.

Let $X$ be a smooth projective curve of genus $g$ over a number  field $k$, for simplicity (this assumption plays no role for defining local archimedean heights).  Let  
$\mathrm{Div}^0(X)$  denote the vector space of divisors of degree zero on $X$ which are defined over $k$ and  have coefficients in $\QQ$. It consists of degree zero Galois-equivariant formal linear combinations of points of $X$ defined over a finite extension of $k$.

\begin{lemma}\label{lem:existence log forms}  For any divisor $D \in \mathrm{Div}^0(X)$ of degree zero with  support
contained in  a   subscheme  $P \subset X$ of dimension $0$ defined over $k$, there exists a logarithmic differential form 
$$\nu_{D} \quad \in \quad \Gamma(X, \Omega_{X/k}^1(\log P)) $$
with the property that $\mathrm{Res}\, \nu_{D}  = D$. 
It is unique up to addition of a global regular form in  $\Gamma(X, \Omega^1_{X/k}) \cong F^1 H^1_\dR(X)$, which has dimension $g$.
\end{lemma}

\begin{proof} The divisor $D$ is given by a rational linear combination of points in $X(K)$ for some finite extension $K$ of $k$. We can assume that it is Galois with $G= \mathrm{Gal}(K/k)$. Since $D$ is $G$-invariant,  the support $P$ of $D$ is indeed defined over $k$. 

We work over the field $K$, and denote  $X\times_k K$, $P\times_k K$ and so on by $X, P$ for simplicity. All vector spaces below are over $K$.  Proposition \ref{prop:global log Hodge} implies that 
$$\Gamma(X, \Omega_{X/K}^1(\log P)) \cong F^1 H^1_\dR(X\min P) \qquad \hbox{ and } \qquad \Gamma(X, \Omega_{X/K}^1) \cong F^1 H^1_\dR(X) \ .$$
To the Gysin (residue) sequence
$$    0 \To   H^1(X) \To H^1(X \backslash P) \overset{\mathrm{Res}}{\To} H^0_\dR(P)(-1) \To  H^2_\dR(X) \ $$
one can apply the exact functor $F^1$  to deduce that there is a short exact sequence
$$  0 \To   F^1 H^1_\dR(X) \To F^1 H^1_\dR(X \backslash P) \overset{\mathrm{Res}}{\To} \mathrm{ker} \left( H^0_\dR(P)(-1) \To  H^2_\dR(X)\right)\To 0  \ .$$
Interpreting $\mathrm{ker} \left( H^0_\dR(P)(-1) \To  H^2_\dR(X)\right)$ as the $K$-vector space of divisors supported on $P$ of degree $0$, we deduce that there exists $\nu \in \Gamma(X, \Omega_{X/K}^1(\log P))$
such that  $\mathrm{Res} \, \nu =D$.  Now define $\nu_D = |G|^{-1} \sum_{g\in G} g(\nu)$. It is $G$-invariant, and hence defined over $k$. Since $D$ is also $G$-invariant, it satisfies $\mathrm{Res} \, \nu_D =D$. 
\end{proof}
It follows from the lemma that, given any choice of $k$-basis for the regular differentials $\omega_1,\ldots, \omega_g$ of $\Gamma(X,\Omega^1_{X/k})$, the element
$$\nu_D \wedge \omega_1 \wedge  \cdots \wedge \omega_g \quad \in \quad \bigwedge^{g+1} F^1 H^1_\dR(X \backslash P)$$
is well-defined up to scalar multiple in $k^{\times}$. Now consider two  divisors of degree zero
$$ D , E \quad \in  \quad \mathrm{Div}^0(X)$$
  with  disjoint supports $A ,B \subset X$  respectively. 
Consider the  object  in $\mathcal{H}(k)$ defined by ($M_{\dR}, \left(M_{\B, \sigma}\right)_{\sigma}$, $\left(\mathrm{comp}_{\sigma,\dR}\right)_{\sigma})$ where 
$$M=  H^1(X \backslash A, B)\ .$$
By Proposition \ref{prop:grW 0 and 2n},  $F^1 M_{\dR} = F^1  H^1(X \backslash A)$ and therefore 
$$\Gamma(X, \Omega^1_{X/k}(\log A)) \overset{\sim}{\To}      F^1 M_{\dR}\ .$$
 Now since $\mathcal{H}(k)$ is an abelian tensor category, the objects $M^{\otimes m}$, $\bigwedge^m M$, and so on, are all objects of $\mathcal{H}(k)$.
The dual of $M$ is the object $M^{\vee} = H^1(X \backslash B, A)(1)$. 
In particular there is an isomorphism 
$$\Gamma(X, \Omega^1_{X/k}(\log B)) \overset{\sim}{\To}  F^0 M^{\vee}_{\dR}$$
which we denote by $\omega \mapsto \omega(1)$.  
Denote the $m$th tensor power of this map by $x\mapsto x(m):   \Gamma(X, \Omega^1_{X/k}(\log(B)))^{\otimes m} \To  (M^{\vee}_{\dR})^{\otimes m} .$ For any basis $\omega_1,\ldots, \omega_g$ of $\Gamma(X,\Omega^1_{X/k})$ write 
$$\omega_{\mathrm{vol}}
= \omega_1 \wedge \cdots \wedge \omega_g \quad \in \quad  \bigwedge^{g} M_{\dR}.$$ 
Denote the corresponding element in the dual space  $\bigwedge^{g} M^{\vee}_{\dR}$ by $\omega_{\vol}(g)$. For some other basis $\eta_1,\ldots,\eta_g$ of $\Gamma(X,\Omega^1_{X/k})$ (possibly the same) denote  the corresponding elements by $\eta_\vol$ and $\eta_\vol(g)$.

\begin{definition} \label{definitionDRheightpairing} Given two such forms $\omega_{\vol}$ and $\eta_{\vol}$  define the  \emph{de Rham  height pairing} for  any pair of divisors $D,E$ with disjoint supports by 
$$\langle D, E \rangle_{\omega, \eta}^{\mm,\dR} = \left[\bigwedge^{g+1} M \  , \   (\nu_{D} \wedge \omega_{\vol}) (g+1)  \ ,  \   \nu_{E} \wedge \eta_{\vol}  \right]^{\mm,\dR} \quad \in \quad \Pe^{\mm,\dR}_{\mathcal{H}(k)}\ .  $$
It is well-defined. It only depends on the choice of  $\omega_{\vol}$, $\eta_\vol$, and not the choice of representatives $\nu_{D}, \nu_E$, nor $A,B$.  Define also
$$\langle\omega_{\vol},\eta_{\vol}\rangle^{\mm,\dR} =  \left[\bigwedge^{g} H^1(X) \  , \     \omega_{\vol} (g) \  ,  \   
\eta_{\vol}  \right]^{\mm,\dR} \quad \in \quad \Pe^{\mm,\dR}_{\mathcal{H}(k)}\ .  $$
Note that the de Rham height pairing defined above will correspond to the usual N\'eron--Tate height pairing multiplied by the factor $\langle\omega_{\vol},\eta_{\vol}\rangle^{\mm,\dR}$. 
\end{definition}

\begin{remark} \label{remarkExpandHeight} 
By expanding the determinant, we have 
\begin{equation} \label{dRheightexpand} \langle D, E \rangle_{\omega, \eta}^{\mm,\dR} = \left[ M , \nu_{D}(1), \nu_{E} \right]^{\mm,\dR}  \cdot \langle \omega_{\vol} ,\eta_{\vol}\rangle^{\mm,\dR} + \xi 
\end{equation}
where  $\xi$ is the signed sum  of products of $(g+1)$ terms:
\begin{equation}  \label{dRheightjunk} 
\pm \, \left[M, \nu_{D}(1), \eta_{j_1} \right]^{\mm,\dR} \cdot   \left[M, \omega_{i_1}(1),  \nu_{E} \right]^{\mm,\dR}  \cdot \, \prod_{r=2}^g  \left[M, \omega_{i_r}(1), \eta_{j_r} \right]^{\mm,\dR}\end{equation}
where $\{i_1,\ldots, i_g\} = \{j_1,\ldots, j_g\} = \{1,\ldots, g\}. $
\end{remark}

\begin{proposition} \label{lemdRHeightproperties} The de Rham height pairing has the following properties: 
\begin{enumerate}
    \item It is  bilinear in $D, E.$
    \item It is symmetric: 
    $ \langle D, E \rangle_{\omega, \eta}^{\mm,\dR} = \langle E, D \rangle_{\eta, \omega}^{\mm,\dR}\ .$
\item Let $f: X \rightarrow \PP^1$ be a rational function on $X$ and let  $D= \sum_{j} n_j x_j$ be a degree zero divisor with support  disjoint from $f^{-1}(\{0,\infty\})$. In this case,  
$$\langle D, \mathrm{div}(f)\rangle_{\omega, \eta}^{\mm,\dR} =    \langle \omega_{\vol}, \eta_{\vol}\rangle^{\mm,\dR} \left(  \sum_j  n_j \log^{\mm,\dR} (f(x_j) ) \right)   \ .  $$
    \end{enumerate}
\end{proposition}

\begin{proof}
 For $(1)$, the class of $\nu_{D} \in F^1 M_{\dR}/ F^1 H_{\dR}^1(X)$ is  uniquely defined and linear in $D$. Therefore $\nu_D \wedge \omega_{\vol}$ is well-defined and depends linearly on $D$. 
  
  Statement $(2)$ follows from Poincar\'e--Verdier duality:
$$\bigwedge^{g+1}  H^1(X\backslash A, B)^{\vee}  = \bigwedge^{g+1}\left( H^1(X\backslash B, A)(1)\right)   $$
which induces the required equivalence of de Rham periods.

 For $(3)$,  let $B$ denote the support of $D$, which is disjoint from $A=f^{-1}(\{0,\infty\})$.  Thus there is a 
morphism of pairs
$$f: (X \backslash A ,B) \To  (\PP^1\min\{0,\infty\}, f(B)) $$
which gives rise to a morphism $f^*:  H^1(\PP^1\min\{0,\infty\}, f(B)) \rightarrow M$ in the category $\mathcal{H}(k)$.
Since $d\log f =f_\dR^* (d\log z)$, where $z$ is the coordinate on $\PP^1$,  we deduce an equivalence of $\mathcal{H}(k)$-periods
$$  [M,  \omega(1), d\log f]^{\mm,\dR}= [ H^1(\PP^1\min\{0,\infty\}, f(B)), (f_{\dR}^*)^{\vee} \omega(1), d \log z ]^{\mm,\dR}$$
for any $\omega \in \Gamma(X,\Omega^1_X(\log B))$. 
Now suppose that $\omega \in \Gamma(X,\Omega^1_X)$ is regular. Then its class lies in $F^1H^1_\dR(X)\simeq F^1H^1_\dR(X,A)$, and its image under $(f_\dR^*)^\vee$ lies in $F^1H^1_\dR(\PP^1)=0$. We thus have $[M,  \omega(1), d\log f]^{\mm,\dR}=0$ 
and every term in $\xi$ in the expansion \eqref{dRheightexpand} vanishes, by \eqref{dRheightjunk},  leaving only 
$$\langle D,  d\log f \rangle_{\omega, \eta}^{\mm,\dR} =   \langle \omega_{\vol} , \eta_{\vol} \rangle^{\mm,\dR}\,[H^1(\PP^1\min\{0,\infty\},f(B)),(f_\dR^*)^\vee\nu_D(1),d\log z]^{\mm,\dR}\ .$$
It remains to compute the second term on the right. 
The commutative square
$$\xymatrix{
 F^1H^1_\dR(X\min B)\ar[r]^-{\mathrm{Res}}\ar[d]_{(f_\dR^*)^\vee}& H^0_\dR(B)\ar[d] \\
  F^1H^1_\dR(\PP^1\min f(B)) \ar[r]_-{\mathrm{Res}}& H^0_\dR(f(B))
}$$
implies  that $(f^*_{\mathrm{dR}})^\vee\nu_D$ and $\sum_jn_j\,d\log(z-f(x_j))$ have the same residue along $f(B)$, hence they are equal since there are no global regular forms on $\PP^1$. One then recognizes the de Rham logarithms from \S\ref{par:example log} and the claim follows.
\end{proof}

\subsubsection{Archimedean height}
Let us now fix an  embedding $\sigma: k \hookrightarrow \CC$ and let $M$ be as above. Without mention to the contrary, $X(\CC)$ denotes $X_{\sigma}(\CC)$ and for any form $\omega \in \Gamma(X, \Omega^1_{X/k})$ we denote its image under $\sigma$ simply by $\omega$.

The  single-valued period induces an isomorphism
$$\s_{\sigma}:  \left(\bigwedge^{r} M_{\dR} \right) \otimes_{k, \sigma} \CC \overset{\sim}{\To} \left( \bigwedge^{r} M_{\dR} \right)\otimes_{k, \overline{\sigma}} \CC\ ,  $$
for all $r\geq 0$, and similarly with $M$ replaced by $H^1(X)$.  If $r=g$, the single-valued period
$ \langle \omega_{\vol}(g)   ,   \s_{\sigma} \eta_{\vol} \rangle$
is non-zero, since it is proportional to 
$\langle \omega_{\vol}(g)   ,   \s_{\sigma} \omega_{\vol} \rangle  < 0$.  
This  follows from the  Hodge--Riemann bilinear relations: the  pairing 
\begin{equation} \label{RiemannPairing} 
\nu \otimes \omega \quad  \mapsto \quad   \langle \nu(1), \s_{\sigma}  \omega \rangle  =    -\frac{1}{2\pi i} \int_{X(\CC)}  \overline{\nu} \wedge \omega
\end{equation}
restricted to the space of  holomorphic differential forms $\Gamma(X, \Omega^1_{X(\CC)})$ is a  negative  definite Hermitian form, and $\langle \omega_{\vol}(g)   ,   \s_{\sigma} \omega_{\vol} \rangle $ is  its determinant.  

\begin{definition} Define the \emph{normalised height pairing}  by 
$$\langle D, E \rangle^{\sigma} =    \frac{   \langle \nu_D \wedge \omega_{\vol}(g+1), \s_{\sigma} (\nu_E \wedge \eta_{\vol} )\rangle }{   \langle \omega_{\vol}(g)   ,   \s_{\sigma} \eta_{\vol} \rangle  }   \qquad \in  \; \CC \ .$$
It is evidently well-defined (independent of the choice of $\omega_\vol$ and $\eta_\vol$ and  the choices of representatives for $\nu_{D}, \nu_{E}$). 
Note the slight difference in notation compared to Definition \ref{definitionDRheightpairing}, which is `motivic'  but not canonical, since it depends on the choice of $\omega_{\vol}, \eta_{\vol}$. By contrast, the normalised height pairing  $\langle D, E \rangle^{\sigma}$ is canonical, but is not a period, rather a quotient of two single-valued  periods. 
\end{definition}

This follows from Proposition \ref{lemdRHeightproperties} that $\langle D, E \rangle^{\sigma}$ is bilinear, symmetric, and satisfies 
$$\langle D, \mathrm{div}(f)\rangle^{\sigma} = \sum_j n_j  \log |\sigma(f(x_j))|^2 $$ 
whenever  $D = \sum_j n_j x_j$ has disjoint support  from $f^{-1}(\{0,\infty\}).$ To compute it, let $\omega_1,\ldots, \omega_g$ and $\eta_1,\ldots, \eta_g$ denote two bases for $\Gamma(X,\Omega^1_X)$, and denote by $P$ the $g\times g$ matrix with entries
$$(P)_{p,q}= \langle \omega_p, \s_{\sigma} \eta_q\rangle  =   -\frac{1}{ 2 \pi i}  
\int_{X(\CC)} \overline{\omega_p} \wedge \eta_q
\ .$$
It is a submatrix of a single-valued period matrix of $H^1(X)$, with determinant $ \det(P)=\langle \omega_{\vol}, \eta_{\vol}   \rangle^{\sigma} . $
\begin{proposition} \label{prop: DEsigmaformula} 
We have the  following formula for the normalised height pairing:
  $$\langle D, E\rangle^{\sigma} = -\frac{1}{2\pi i} \int_{X(\CC)} \overline{\nu_D} \wedge \nu_E + \sum_{p,q=1}^g \alpha_{p,q}  \int_{X(\CC)} \overline{\nu_D} \wedge \eta_p \int_{X(\CC)} \overline{\omega_q} \wedge \nu_E  $$
where  the constants $\alpha_{p,q}$ do not depend on $D,E$ and are 
$$\alpha_{p,q} =  \frac{(-1)^{p+q+1}}{(2\pi i)^2}  \frac{ \det (P(p,q))   }  {  \det(P) }  $$
where $P(p,q)$ denotes the matrix $P$ with row $p$ and column $q$ removed.
\end{proposition}
\begin{proof}
It follows from Remark \ref{remarkExpandHeight}, Theorem \ref{thm:sv recipe} and Remark \ref{rem:sv recipe number field}. 
\end{proof}

\subsubsection{Computation} The  formula in Proposition \ref{prop: DEsigmaformula} for the normalised height pairing simplifies drastically if we allow  differential forms with complex coefficients. 
Let $D$ be a degree zero divisor on $X$ with support in $P\subset X$.

\begin{lemma} \label{lem: complexlogform123} 
There exists a unique logarithmic form over $\CC$
\begin{equation} \label{nuDreal} \nu^{\sigma}_D \quad \in \quad \Gamma(X, \Omega^1_X(\log P))\otimes_{k, \sigma}  \CC 
\end{equation} 
such that $\mathrm{Res}\,  \nu^{\sigma}_D=D$ and the following equivalent conditions are satisfied:
\begin{enumerate}
    \item The single-valued periods of  $H^1(X\backslash P)$  of the form 
$$ \langle  \omega(1)  , \s_{\sigma} \nu_D^{\sigma} \rangle $$
vanish for all  $\omega \in \Gamma(X,\Omega^1_{X/k})$.
    \item  For all $\omega \in \Gamma(X,\Omega^1_{X/k})$, we have $$\int_{X_{\sigma}(\CC)}  \overline{\omega}\wedge\nu^{\sigma}_D  =0 \ , $$
    i.e., $\nu_D^{\sigma}$ is orthogonal to all regular forms on $X$. 
        \item  For all closed cycles $\gamma \subset  (X \backslash P)_{\sigma}(\CC) $ 
        $$\mathrm{Re} \,  \int_{\gamma} \nu^{\sigma}_D =0 \ ,$$
        i.e., $\nu_D^{\sigma}$ has imaginary periods. 
\end{enumerate}
\end{lemma}
 
\begin{proof} Note that conditions $(1)$ and $(2)$ are equivalent by Theorem \ref{thm:sv recipe general} and Remark \ref{rem:sv recipe number field}. 
 The existence and uniqueness in $(2)$ follow from the non-degeneracy of the pairing \eqref{RiemannPairing}. More precisely, let $\mu_D\in\Gamma(X,\Omega^1_{X/k}(\log P))$ be any logarithmic form such that $\mathrm{Res}_D(\mu_D)=D$ as in Lemma \ref{lem:existence log forms}. There exists a unique regular form $\xi^\sigma\in\Gamma(X_\sigma,\Omega^1_{X_\sigma})$ such that
 $$\int_{X_\sigma(\CC)}\overline{\omega}\wedge\mu_D = \int_{X_\sigma(\CC)}\overline{\omega}\wedge\xi^\sigma$$
 for every $\omega\in\Gamma(X_\sigma,\Omega^1_{X_\sigma})$. We can then set $\nu_D^\sigma=\mu_D-\xi^\sigma$. 

Finally, the existence and uniqueness of a form satisfying $(3)$ is classical. The condition in $(3)$ is satisfied if $\gamma$ is a small loop enclosing a point of $P$ and it is enough to check it for classes of closed cycles $\gamma$ in the image of a chosen splitting $H_1^{\B,\sigma}(X)\rightarrow H_1^{\B,\sigma}(X\min P)$. The pairing $H_1^{\B,\sigma}(X;\RR)\otimes \Gamma(X_\sigma,\Omega^1_{X_\sigma})\rightarrow \RR$ given by $\gamma\otimes\omega\mapsto \mathrm{Re}\int_\gamma\omega$ is a perfect pairing of real vector spaces by Hodge symmetry. Thus, for some logarithmic form $\mu_D$ as before, there exists a unique regular form $\xi^\sigma\in\Gamma(X_\sigma,\Omega^1_{X_\sigma})$ such that
$$\mathrm{Re}\int_\gamma\mu_D = \mathrm{Re}\int_\gamma\xi^\sigma$$
for every class $\gamma\in H_1^{\B,\sigma}(X;\RR)$. We can then set $\nu_D^\sigma=\mu_D-\xi^\sigma$. To conclude, it is enough to prove that condition (3) implies condition (1).
For simplicity we may assume that $D$ is of the form $(p)-(q)$.
A basis of $H^1_\dR(X\backslash \{p,q\})$ is obtained by adding to a basis of $H^1_\dR(X)$ the class of $\nu^\sigma_D$ satisfying condition (3). We choose a basis of $H_1^\B(X\backslash P)$ whose last vector is the class $\delta$ of a small positive loop around $\sigma(p)$. In those bases, the period matrix $A_\sigma$ of $H^1(X\backslash P)$ has the following shape:

$$\left(\begin{array} {ccc|c} &&& \\ & M_\sigma &&  ? \\  &&&\\ \hline & 0 & &2\pi i
\end{array}\right)
$$
where $M_\sigma$ is a period matrix for $H^1(X)$ and the last column consists entirely of imaginary numbers by assumption. Thus, by performing row operations with real coefficients, one can always assume that $A_\sigma$ is block-diagonal, i.e., $?=0$ in the above matrix. These operations do not change the single-valued period matrix $S_\sigma=\overline{A_\sigma}^{-1}A_\sigma$ which is thus block-diagonal. This implies condition (1).
\end{proof}

\begin{corollary}\label{coro: formula height pairing holomorphic}
Let $D, E$ be as above.   The normalised height pairing is given by the following integral of complex logarithmic differential forms  \eqref{nuDreal}:
\begin{equation} \label{heightassvintegral} \langle D, E \rangle^{\sigma} = -\frac{1}{2\pi i} \int_{X_{\sigma}(\CC)} \overline{\nu^{\sigma}_{D}} \wedge \nu^{\sigma}_{E}\ . 
\end{equation}
\end{corollary}

\begin{proof} 
We can replace $\nu_{D}, \nu_{E}$ with $\nu^{\sigma}_{D}, \nu^{\sigma}_{E}$ in Proposition  \ref{prop: DEsigmaformula}, since determinants are invariant under row and column operations. 
All terms in the sum on the right-hand side of that formula vanish by Lemma \ref{lem: complexlogform123} (2).
\end{proof}

The proof shows that we also have the `half-algebraic' formulae:  
$$\langle D, E \rangle^{\sigma} = - \frac{1}{2\pi i} \int_{X_{\sigma}(\CC)} \overline{\nu_D^{\sigma}} \wedge \nu_E =  - \frac{1}{2\pi i} \int_{X_{\sigma}(\CC)} \overline{\nu_D} \wedge \nu^{\sigma}_E\ . $$
 
\begin{remark}
We deduce from \eqref{heightassvintegral} that the normalised height pairing  $\langle D,E\rangle^\sigma$ is the classical archimedean height pairing of the divisors $D$ and $E$. 
Indeed, since $\nu_D^\sigma$ has purely imaginary periods, we can write 
$$\nu_D^\sigma+\overline{\nu_D^\sigma}=d(g_D^\sigma)$$ where $g_D^\sigma:X_\sigma(\CC)\backslash |D|\rightarrow \mathbb{R}$ is harmonic (it is classically called the Green's function of $D$). We can thus compute the integral \eqref{heightassvintegral} via the Cauchy--Stokes theorem (see \S\ref{par: cauchy stokes}, compare with \S\ref{par:example log}):
\begin{align*}
\langle D,E\rangle^\sigma & = -\frac{1}{2\pi i} \int_{X(\CC)} \overline{\nu_D^\sigma} \wedge \nu_E^\sigma \\
& = -\frac{1}{2\pi i}\int_{X(\CC)}(\nu_D^\sigma+\overline{\nu_D^\sigma})\wedge \nu_E^\sigma \\
& = -\frac{1}{2\pi i}\int_{X(\CC)} d(g_D^\sigma\nu_E^\sigma) \\
& = \sum_jn_j\,g_D^\sigma(\sigma(x_j))\ ,
\end{align*}
where $E=\sum_jn_jx_j$. This shows \emph{a posteriori} that $\langle D,E\rangle^\sigma$ is a real number, which is obvious from our definition only if $\sigma$ is a real embedding. 
\end{remark}

\subsection{Explicit single-valued multiple zeta values}\label{par: SVMZV} Let $n_1,\ldots, n_r\in \ZZ_{\geq 1}$ with $n_r\geq 2$. The corresponding  multiple zeta value is defined by  the convergent sum
$$\zeta(n_1,\ldots,n_r)=\sum_{1\leq k_1<\cdots <k_r}\frac{1}{k_1^{n_1}\cdots k_r^{n_r}}\ \ . $$
It  can be represented by the  integral
\begin{equation}
\label{MZV integral} \zeta(n_1,\ldots, n_r) = \int_{0< t_1<\cdots < t_n <1} \omega_{n_1,\ldots, n_r} 
\end{equation}
with $n= n_1+\cdots+ n_r$,  and 
$$\omega_{n_1,\ldots, n_r} = (-1)^r \frac{dt_1}{t_1-e_1} \wedge \cdots \wedge \frac{dt_n}{t_n-e_n}$$
where  $(e_1,\ldots, e_n) = (1, 0^{n_{1}-1}, 1, 0^{n_{2}-1}, \ldots, 1, 0^{n_r-1})$ and where $0^k$ denotes a string of $0$'s of length $k$. It  lifts to a  motivic multiple zeta value \cite{brownSVMZV}: 
$$\zeta^{\mm}(n_1,\ldots, n_r) = [ \mathcal{O}(\pi_1^{\mathrm{mot}} (\PP^1 \backslash \{0,1,\infty\},\overset{\rightarrow}{1}_0, -\overset{\rightarrow}{1}_1)), \mathrm{dch}, \omega_{n_1,\ldots, n_r}]^{\mm} $$ 
where $\mathcal{O}(\pi_1^{\mathrm{mot}} (\PP^1 \backslash \{0,1,\infty\},\overset{\rightarrow}{1}_0, -\overset{\rightarrow}{1}_1))$ is the motivic fundamental groupoid of the projective line minus three points \cite{deligneP1, delignegoncharov} relative to tangential basepoints at $0$ and $1$, and  $\mathrm{dch}$ denotes the image of the straight line path from $0$ to $1$.  The motivic multiple zeta value $\zeta^{\mm}(n_1,\ldots, n_r)$
 lies in $ \mathcal{P}^{\mm}_{\MT(\mathbb{Z})}$, where $\MT(\ZZ)\subset \MT(\QQ)$ is the Tannakian category of mixed Tate motives over $\ZZ$, and its image under the period homomorphism is  the number $\zeta(n_1,\ldots, n_r).$
 
 Since the underlying motive is both effective and separated (a fortiori it is mixed Tate and so all non-vanishing Hodge numbers are of type $(p,p)$),  one can define the de Rham versions $\zeta^{\mm,\dR} = \pi^{\mm,\dR} \zeta^{\mm}$,  where $\pi^{\mm,\dR}$  denotes the    de Rham projection  of Definition \ref{defi:dR projection}.  The single-valued multiple zeta values \cite{brownSVMZV} are defined to be 
$$\zeta^{\s}(n_1,\ldots, n_r) = \s\,   \zeta^{\mm,\dR}(n_1,\ldots, n_r ) \ , $$
and arise naturally in a variety of contexts. They were actually defined using the variant $\sv$ of the single-valued period map, which produces the same result in view of Remark \ref{rem: s vs sv MT}.) 
Although there are several different recipes for computing these numbers (for example, as the values at $1$ of single-valued multiple polylogarithms) they are  far from  explicit and require solving a complicated equation involving the Drinfeld associator.

However, we can now write down an explicit formula for the $\zeta^{\s}$ using Theorem \ref{thm:sv recipe}. 
For convenience, we shall actually use a different interpretation of the integrals \eqref{MZV integral},  due  to
 Goncharov and Manin,  as  periods of motives of the moduli space $\mathcal{M}_{0,n+3}$ \cite{goncharovmanin}. 
  The corresponding motivic periods  are no doubt the same as  the motivic multiple zeta values defined via the  motivic fundamental groupoid of the projective line minus 3 points, although this has not been written explicitly in the literature.  Nevertheless, we can also directly  compute the single-valued periods of the latter using a limiting  argument to deal with the tangential base points  (see  \cite[Section 6]{BDLauricella}).
It leads, as expected, to the same  formula.

\begin{theorem} The single-valued multiple zeta value is an integral:
$$\zeta^{\s}(n_1,\ldots, n_r) = \frac{ (-1)^{\frac{n(n+1)}{2}} }{ (2\pi i)^n} \int_{\CC^n} \frac{ d t_1 \wedge \cdots \wedge d t_n}{t_1(t_2-t_1)\cdots (t_n-t_{n-1})(1-t_n)  } \wedge \overline{\omega_{n_1,\ldots, n_r}} \ .$$
\end{theorem}

\begin{proof}
Using the geometric interpretation of \cite{goncharovmanin}, 
we can define the  motivic multiple zeta value to be:
$$\zeta^{\mm} (n_1,\ldots, n_r)= [H^n(\overline{\mathcal{M}}_{0,n+3}\backslash A_{n_1,\ldots,n_r}, B), [X^{\delta}], [\omega_{n_1,\ldots, n_r} ]]^{\mm}\ ,$$
where $A_{n_1,\ldots,n_r}$ and $B$ are divisors in $\overline{\mathcal{M}}_{0,n+3}$ such that $(\overline{\mathcal{M}}_{0,n+3},A_{n_1,\ldots,n_r},B)$ satisfies $(\star)_\QQ$. It is proved in \cite{goncharovmanin} that the corresponding relative cohomology group is the realisation of a mixed Tate motive over $\mathbb{Z}$. Here $[X^\delta]$ is the homology class of the topological closure in $\overline{\mathcal{M}}_{0,n+3}(\mathbb{R})$ of the locus $\{0<t_1<\cdots <t_n<1\}$, whose image under the projection $c_0^\vee$ is computed in \cite[\S 2]{BD2}. It is given in coordinates by the following form with logarithmic singularities along $B$:
$$\nu = (-1)^{\frac{n(n+1)}{2}} \prod_{i=0}^n (t_{i+1} - t_{i})^{-1} dt_1\wedge \cdots \wedge dt_n \ ,$$
where $t_0=0$, $t_{n+1}=1$. It follows from Theorem \ref{thm:sv recipe} that
$$ \zeta^{\s}(n_1,\ldots, n_r)=  \frac{1}{(2\pi i)^n } \int_{\overline{\mathcal{M}}_{0,n+3}(\CC)} \nu \wedge \overline{\omega_{n_1,\ldots, n_r}}   \ .$$
Since $\CC^n \supset \mathcal{M}_{0,n+3}(\CC) \subset \overline{\mathcal{M}}_{0,n+3}(\CC)$ all differ by  sets of Lebesgue measure zero, the integral above can be viewed as an integral on $\CC$, and the result follows.
\end{proof}

\begin{remark} We know that $\zeta^{\s}(2)=0$. It would be interesting to prove directly via elementary means that the following  integral vanishes: 
$$\zeta^{\s}(2) = -\frac{1}{(2\pi i)^2}\int_{\CC^2} \frac{dt_1\wedge dt_2}{t_1(t_2-t_1)(1-t_2)} \wedge \frac{d\overline{t}_1\wedge d\overline{t}_2}{(1-\overline{t}_1)\overline{t}_2} \ ,$$
and similarly for more general families of (motivic) multiple zeta values which are known to evaluate to multiples of powers of $2\pi i$. 
\end{remark}

\bibliographystyle{alpha}

\bibliography{biblio}

\begin{thebibliography}{HMS17}

\bibitem[BD94]{beilinsondeligne}
A.~Beilinson and P.~Deligne.
\newblock Interpr\'{e}tation motivique de la conjecture de {Z}agier reliant
  polylogarithmes et r\'{e}gulateurs.
\newblock In {\em Motives ({S}eattle, {WA}, 1991)}, volume~55 of {\em Proc.
  Sympos. Pure Math.}, pages 97--121. Amer. Math. Soc., Providence, RI, 1994.

\bibitem[BD19a]{BDLauricella}
F.~Brown and C.~Dupont.
\newblock Lauricella hypergeometric functions, unipotent fundamental groups of
  the punctured {R}iemann sphere, and their motivic coactions.
\newblock {\em Preprint: arXiv:1907.06603}, 2019.

\bibitem[BD19b]{BD2}
F.~Brown and C.~Dupont.
\newblock Single-valued integration and superstring amplitudes in genus zero.
\newblock {\em arXiv preprint, to appear}, 2019.

\bibitem[Bei87]{BeilinsonHeights}
A.~A. Beilinson.
\newblock Height pairing between algebraic cycles.
\newblock In {\em {$K$}-theory, arithmetic and geometry ({M}oscow,
  1984--1986)}, volume 1289 of {\em Lecture Notes in Math.}, pages 1--25.
  Springer, Berlin, 1987.

\bibitem[Blo84]{blochheight}
S.~Bloch.
\newblock Height pairings for algebraic cycles.
\newblock In {\em Proceedings of the {L}uminy conference on algebraic
  {$K$}-theory ({L}uminy, 1983)}, volume~34, pages 119--145, 1984.

\bibitem[Bro14]{brownSVMZV}
F.~Brown.
\newblock Single-valued motivic periods and multiple zeta values.
\newblock {\em Forum Math. Sigma}, 2:e25, 37, 2014.

\bibitem[Bro17]{brownnotesmot}
F.~Brown.
\newblock Notes on motivic periods.
\newblock {\em Commun. Number Theory Phys.}, 11(3):557--655, 2017.

\bibitem[Bro18]{brownCNHMF3}
F.~Brown.
\newblock A class of non-holomorphic modular forms {III}: real analytic cusp
  forms for {$\mathrm{SL}_2(\mathbb Z)$}.
\newblock {\em Res. Math. Sci.}, 5(3):5:34, 2018.

\bibitem[BT82]{bottandtu}
R.~Bott and L.~W. Tu.
\newblock {\em Differential forms in algebraic topology}, volume~82 of {\em
  Graduate Texts in Mathematics}.
\newblock Springer-Verlag, New York-Berlin, 1982.

\bibitem[DDP12]{Dixon2012}
L.~J. Dixon, C.~Duhr, and J.~Pennington.
\newblock Single-valued harmonic polylogarithms and the multi-{R}egge limit.
\newblock {\em Journal of High Energy Physics}, 2012(10):74, Oct 2012.

\bibitem[Del71]{delignehodge2}
P.~Deligne.
\newblock Th{\'e}orie de {H}odge. {II}.
\newblock {\em Inst. Hautes {\'E}tudes Sci. Publ. Math.}, 40:5--57, 1971.

\bibitem[Del74]{delignehodge3}
P.~Deligne.
\newblock Th{\'e}orie de {H}odge. {III}.
\newblock {\em Inst. Hautes {\'E}tudes Sci. Publ. Math.}, 44:5--77, 1974.

\bibitem[Del89]{deligneP1}
P.~Deligne.
\newblock Le groupe fondamental de la droite projective moins trois points.
\newblock In {\em Galois groups over {${\bf Q}$} ({B}erkeley, {CA}, 1987)},
  volume~16 of {\em Math. Sci. Res. Inst. Publ.}, pages 79--297. Springer, New
  York, 1989.

\bibitem[Del06]{deligneequadiff}
P.~Deligne.
\newblock {\em {\'E}quations diff{\'e}rentielles {\`a} points singuliers
  r{\'e}guliers}, volume 163.
\newblock Springer, 2006.

\bibitem[DG05]{delignegoncharov}
P.~Deligne and A.~B. Goncharov.
\newblock Groupes fondamentaux motiviques de {T}ate mixte.
\newblock {\em Ann. Sci. \'{E}cole Norm. Sup. (4)}, 38(1):1--56, 2005.

\bibitem[DR31]{derhamthese}
G.~De~Rham.
\newblock {\em Sur l'analysis situs des vari\'et\'es \`a $n$ dimensions}.
\newblock Doctorat d'\'etat, Facult\'{e} des Sciences de Paris, 1931.

\bibitem[FK16]{felderkazhdanRiemann}
G.~Felder and D.~Kazhdan.
\newblock Divergent integrals, residues of {D}olbeault forms, and asymptotic
  {R}iemann mappings.
\newblock {\em International Mathematics Research Notices},
  2017(19):5897--5918, 2016.

\bibitem[FK18]{felderkazhdanregularization}
G.~Felder and D.~Kazhdan.
\newblock Regularization of divergent integrals.
\newblock {\em Selecta Mathematica}, 24(1):157--186, 2018.

\bibitem[Fon94]{fontainecorps}
J.-M. Fontaine.
\newblock Le corps des p{\'e}riodes $p$-adiques.
\newblock {\em Ast{\'e}risque}, 223:59--111, 1994.

\bibitem[Gil11]{gillamrealoriented}
W.~D. Gillam.
\newblock Oriented real blowup.
\newblock {\em Preprint}, 2011.
\newblock {A}vailable on the author's webpage
  \url{http://www.math.boun.edu.tr/instructors/wdgillam/orb.pdf}.

\bibitem[GM04]{goncharovmanin}
A.~B. Goncharov and Yu.~I. Manin.
\newblock Multiple {$\zeta$}-motives and moduli spaces
  {$\overline{\mathcal{M}}_{0,n}$}.
\newblock {\em Compos. Math.}, 140(1):1--14, 2004.

\bibitem[Gon16]{goncharovexponential}
A.~B. Goncharov.
\newblock Exponential complexes, period morphisms, and characteristic classes.
\newblock In {\em Annales de la Facult{\'e} des sciences de Toulouse:
  Math{\'e}matiques}, volume~25, pages 619--681, 2016.

\bibitem[Gro66]{Grothendieck}
A.~Grothendieck.
\newblock On the de {R}ham cohomology of algebraic varieties.
\newblock {\em Inst. Hautes \'{E}tudes Sci. Publ. Math.}, 29:95--103, 1966.

\bibitem[Gro86]{grossheights}
B.~H. Gross.
\newblock Local heights on curves.
\newblock In {\em Arithmetic geometry ({S}torrs, {C}onn., 1984)}, pages
  327--339. Springer, New York, 1986.

\bibitem[HMS17]{hubermuellerstach}
A.~Huber and S.~M\"uller-Stach.
\newblock {\em Periods and {N}ori motives (With contributions by B. Friedrich
  and J. von Wangenheim)}, volume~65 of {\em Ergebnisse der Mathematik und
  ihrer Grenzgebiete.}
\newblock Springer, Cham, 2017.

\bibitem[Ive86]{iversen}
B.~Iversen.
\newblock {\em Cohomology of sheaves}.
\newblock Universitext. Springer-Verlag, Berlin, 1986.

\bibitem[Joy12]{joycecorners}
D.~Joyce.
\newblock On manifolds with corners.
\newblock In {\em Advances in geometric analysis}, volume~21 of {\em Adv. Lect.
  Math. (ALM)}, pages 225--258. Int. Press, Somerville, MA, 2012.

\bibitem[KZ01]{kontsevichzagier}
M.~Kontsevich and D.~Zagier.
\newblock Periods.
\newblock In {\em Mathematics unlimited -- 2001 and beyond}, pages 771--808.
  Springer, 2001.

\bibitem[Lan88]{Lang}
S.~Lang.
\newblock {\em Introduction to {A}rakelov theory}.
\newblock Springer-Verlag, New York, 1988.

\bibitem[Lee03]{leeintrosmoothmanifolds}
J.~M. Lee.
\newblock {\em Introduction to smooth manifolds}, volume 218 of {\em Graduate
  Texts in Mathematics}.
\newblock Springer-Verlag, New York-Berlin, 2003.

\bibitem[Lev93]{levinetatemotives}
M.~Levine.
\newblock Tate motives and the vanishing conjectures for algebraic
  ${K}$-theory.
\newblock In {\em Algebraic $K$-theory and algebraic topology}, pages 167--188.
  Springer, 1993.

\bibitem[Lev05]{levineK}
M.~Levine.
\newblock Mixed {M}otives. in the \emph{{H}andbook of $K$-theory}, vol. 1, {E}.
  {M}. {F}riedlander, {D}. {R}. {G}rayson, eds., 429--522, 2005.

\bibitem[Pet17]{petersen}
D.~Petersen.
\newblock A spectral sequence for stratified spaces and configuration spaces of
  points.
\newblock {\em Geom. Topol.}, 21(4):2527--2555, 2017.

\bibitem[Pha11]{phambook}
F.~Pham.
\newblock {\em Singularities of integrals}.
\newblock Universitext. Springer, London; EDP Sciences, Les Ulis, 2011.
\newblock Homology, hyperfunctions and microlocal analysis, With a foreword by
  Jacques Bros, Translated from the 2005 French original, With supplementary
  references by Claude Sabbah.

\bibitem[Sai90]{saitoMHM}
M.~Saito.
\newblock Mixed {H}odge modules.
\newblock {\em Publ. Res. Inst. Math. Sci.}, 26(2):221--333, 1990.

\bibitem[Sch94]{Scholl}
A.~J. Scholl.
\newblock Height pairings and special values of {$L$}-functions.
\newblock In {\em Motives ({S}eattle, {WA}, 1991)}, volume~55 of {\em Proc.
  Sympos. Pure Math.}, pages 571--598. Amer. Math. Soc., Providence, RI, 1994.

\bibitem[Sch14]{schnetzgraphical}
O.~Schnetz.
\newblock Graphical functions and single-valued multiple polylogarithms.
\newblock {\em Commun. Number Theory Phys.}, 8(4):589--675, 2014.

\bibitem[ST14]{StiebergerTaylor}
S.~Stieberger and T.~R. Taylor.
\newblock Closed string amplitudes as single-valued open string amplitudes.
\newblock {\em Nuclear Phys. B}, 881:269--287, 2014.

\bibitem[Sti14]{stiebergersvMZV}
S.~Stieberger.
\newblock Closed superstring amplitudes, single-valued multiple zeta values and
  the {D}eligne associator.
\newblock {\em Journal of Physics A: Mathematical and Theoretical},
  47(15):155401, 2014.

\bibitem[Voi07]{voisin}
C.~Voisin.
\newblock {\em Hodge theory and complex algebraic geometry. {I}}, volume~76 of
  {\em Cambridge Studies in Advanced Mathematics}.
\newblock Cambridge University Press, Cambridge, english edition, 2007.
\newblock Translated from the French by Leila Schneps.

\bibitem[ZG00]{ganglzagier}
D.~Zagier and H.~Gangl.
\newblock Classical and elliptic polylogarithms and special values of
  {$L$}-series.
\newblock In {\em The arithmetic and geometry of algebraic cycles ({B}anff,
  {AB}, 1998)}, volume 548 of {\em NATO Sci. Ser. C Math. Phys. Sci.}, pages
  561--615. Kluwer Acad. Publ., Dordrecht, 2000.

\end{thebibliography}

\end{document}